\documentclass[10pt]{article}
\usepackage{graphicx}
\usepackage{subfigure}
\usepackage[utf8]{inputenc}
\usepackage{anyfontsize}
\usepackage{newtxmath}
\usepackage{mathrsfs}
\usepackage{amsthm}
\usepackage{nameref}
\usepackage{answers}
\usepackage{soul}
\usepackage{natbib} 
\usepackage{url}
\usepackage{stmaryrd}
\usepackage[hidelinks]{hyperref}
\usepackage{fancyhdr}
\usepackage[left=1.2cm,right=1.2cm,top=2cm,bottom=2cm]{geometry}
\usepackage[dvipsnames, svgnames]{xcolor}
\usepackage{tikz}
\usepackage[framemethod=tikz]{mdframed}
\usepackage{lipsum}

\usepackage{amssymb}
\usepackage{bm}
\usepackage{authblk}

\usepackage{amsthm}
\usepackage{mathtools}
\usepackage{enumerate}
\usepackage{xcolor}
\usepackage{mathrsfs}
\usepackage{newtxmath}
\usepackage{aliascnt}

\newcommand{\R}[0]{\mathbb R}
\renewcommand{\P}[0]{\mathbb P}
\newcommand{\Q}[0]{\mathbb Q}
\newcommand{\N}[0]{\mathcal{N}}
\newcommand{\W}[0]{\mathcal{W}_2}
\newcommand{\w}[0]{\mathcal{W}_1}
\renewcommand{\S}{\mathcal{S}}
\renewcommand{\H}{\mathcal{H}}
\renewcommand{\N}{\mathbb{N}}
\newcommand{\C}[0]{\mathscr{C}}

\newcommand{\F}[0]{\mathcal{F}}
\renewcommand{\L}[1]{\text{L}^{#1}}

\newcommand{\1}[1]{\mathbf{1}_{#1}}
\renewcommand{\~}[1]{\Tilde{#1}}
\newcommand{\E}[0]{\mathbb E}
\renewcommand{\sup}[1]{\underset{#1}{\text{sup }}}
\renewcommand{\inf}[1]{\underset{#1}{\text{inf }}}
\renewcommand{\lim}[1]{\underset{#1}{\text{lim }}}

\newcommand{\Pcal}[0]{\mathcal{P}}
\newcommand{\loi}[1]{{\seg*{#1}}}
\renewcommand{\v}[1]{\Check{#1}}

\renewcommand{\exp}[1]{\text{exp}\left( #1 \right)}

\newcommand{\Tr}[1]{\text{Tr}\left( #1 \right)}

\newcommand{\A}[0]{\mathscr{A}}

\DeclarePairedDelimiterX\braket[2]{\langle}{\rangle}{#1 , #2}

\DeclarePairedDelimiter\set{\{}{\}}

\DeclarePairedDelimiter\seg{[ }{ ]}

\DeclarePairedDelimiter\abs{\lvert}{\rvert}
\DeclarePairedDelimiter\norm{\lVert}{\rVert}

\renewcommand{\d}{\mathrm{d}}

\newcommand{\LL}[0]{\mathscr{L}}

\renewcommand{\b}[0]{\mathfrak{b}}
\renewcommand{\H}[0]{\mathscr{H}}

\newcommand{\EE}[0]{\mathcal{E}}
\newcommand{\UU}[0]{\mathcal{U}}

\newcommand{\WW}[0]{\mathcal{W}}
\newcommand{\X}[0]{\mathcal{X}}

\newtheorem{thm}{Theorem}[section]
\newtheorem*{thm*}{Theorem}

\newaliascnt{Proposition}{thm}

\newtheorem*{prop*}{Proposition}
\newtheorem{prop}[Proposition]{Proposition}

\newaliascnt{Lemma}{thm}

\newtheorem*{lem*}{Lemma}
\newtheorem{lem}[Lemma]{Lemma}

\newaliascnt{cor}{thm}
\newtheorem{corollaire}[cor]{Corollary}
\newtheorem*{corollaire*}{Corollary}

\newaliascnt{def}{thm}
\theoremstyle{definition}
\newtheorem*{definition*}{Definition}

\counterwithin{equation}{section}
\theoremstyle{definition}   
\newtheorem{hyp}{Assumption}
\newtheorem*{hyp*}{Assumption}

\newaliascnt{Remark}{thm}
\newtheorem{rmq}[Remark]{Remark}
\newtheorem*{rmq*}{Remark}

\newaliascnt{Example}{thm}
\theoremstyle{remark}
\newtheorem{ex}[Example]{Example}
\newtheorem*{ex*}{Example}

\sethlcolor{yellow}

\title{Ergodic distribution dependent BSDE and application to long-time behavior of finite horizon distribution dependent BSDE}
\author[1]{Kaplan Desbouis}
\author[1]{Adrien Richou}

\affil[1]{Univ. Bordeaux, CNRS, INRIA, Bordeaux INP, IMB, UMR
5251, F-33400 Talence, France}
\begin{document}

\maketitle

\begin{abstract}
    After proving existence and uniqueness of  ergodic distribution dependent Backward Stochastic Differential Equations (BSDEs) under strong and weak dissipativity regimes for the underlying McKean–Vlasov SDE, we are able to leverage this new framework to investigate the long-time behavior of distribution-dependent BSDEs on a finite-time horizon. Finally, we apply our results to solve an ergodic McKean-Vlasov stochastic control problem and study the long-time behavior of the value function of a finite-horizon McKean-Vlasov stochastic control problem.   
    
\end{abstract}

{\small\textbf{Keywords:} McKean-Vlasov SDE, Ergodic BSDE, Ergodic optimal control problem
\textbf{Acknowledgment}: The authors acknowledge
funding from the ANR project ReLISCoP (ANR-21-CE40-0001).}

\section{Introduction}\label{Intro}

We study in this paper the following distribution dependent ergodic backward stochastic differential equation (EBSDE for short) in infinite horizon
\begin{equation}\label{Intro: EBSDE}
    \begin{array}{lr}\displaystyle
      Y^{\loi\theta}_t=Y^{\loi\theta}_T +\int^T_t \left( f(X^{\loi\theta}_s,\LL({X^{\loi\theta}_s}),Z^{\loi\theta}_s) -\lambda \right) \d s -\int^T_t Z^{\loi\theta}_s \d W_s,   &  \forall\, 0\leq t\leq T<\infty,
    \end{array}
\end{equation}
where $(Y^{\loi\theta},Z^{\loi\theta},\lambda)$ is the unknown with,
\begin{itemize}
    \item $Y,Z$ are progressively measurable processes, respectively, $\R$-valued and  $\R^{1\times d}$-valued,
    \item $\lambda\in\R$
\end{itemize}
and the given data are 
\begin{itemize}
    \item $W$ an $\R^d$-valued Brownian motion defined on ($\Omega, \F,\P$) a filtered probability space,
    \item  $\theta$ is an $\L{2}(\Omega)$ random variable independent with $W$,
    \item $X^{\loi\theta}$ is an $\R^d$-valued process starting from $\theta$ and solving for all $t\geq 0$ the following McKean-Vlasov's type SDE (MV-SDE for short)
    \begin{equation}\label{Intro: EDS MKV}
        \d X^{\loi\theta}_t=b(X^{\loi\theta}_t,\loi{X^{\loi\theta}_t})\d t+\sigma(X^{\loi\theta}_t,\loi{X^{\loi\theta}_t})\d W_t,
    \end{equation}
    where, for a given r.v. $X$, $\loi{X}$ denotes the distribution of $X$,
    \item $f$ a deterministic measurable function. 
\end{itemize}

This type of BSDE was first introduced in the non-distribution dependent framework by Fuhrman, Hu and Tessitore in \cite{Furhman-Tessitore-Hu_EBSDE_Banach} in order to solve some ergodic stochastic control problem. 
In their paper, they assumed that $\sigma$ is constant and $b$ has a strong dissipativity assumption, that is to say,
\begin{equation*}\begin{array}{lc}
    \exists ~ \eta>0 ,~\forall ~ x,~x'\in\R^d &  \braket*{x-x'}{b(x)-b(x')}\leq -\eta\abs*{x-x'}^2,
\end{array}
\end{equation*}
which allows to get an exponential confluence of trajectories.
The strong dissipative assumption was later dropped by Debussche, Hu and Tessitore in \cite{DEBUSSCHE2011407} for a weakly dissipative one: In other words, $b$ can be written as a sum of a bounded function and a strongly dissipative one. Equivalently, it can be seen as being dissipative at a sufficiently large range. Since the confluence of trajectories no longer holds, this paper is heavily built around the so-called \textit{basic coupling estimate} and a gradient estimate based on the Bismut-Elworthy-Li formula for BSDE; see, e.g. \cite{Furhmann-Tessitore_BE-Formula}. Later on, Hu, Madec and Richou in \cite{Hu-Madec-Richou} make a link between EBSDE and the long time behaviour of finite horizon BSDE given, for $T>0$ fixed, by
\begin{equation}\label{BSDE classic}\begin{array}{lr}\displaystyle
    Y^{T,x}_t=g(X^{x}_T) +\int^T_t f(X^x_s,Z^{T,x}_s) \d s -\int^T_t Z^{T,x}_s \d W_s,
 &  \forall t\in\seg*{0,T},
\end{array}
\end{equation}
where $X$ is the solution to an infinite dimension SDE and $g$ is a given function. Namely, they proved that there exist some constants $L \in \R$, $C>0$ and $\eta>0$ such that
\begin{equation}
\label{intro:expo decroissance}
    |Y^{T,x}_0-\lambda T-Y^x_0 -L| \leq Ce^{-\eta T},
\end{equation}
where $(Y^x,Z^x,\lambda)$ is the solution to the following (non-distribution dependent) EBSDE
\begin{equation}\label{cvgence-intro}\begin{array}{lr}\displaystyle
    Y^{x}_t=Y^{x}_T +\int^T_t \left(f(X^x_s,Z^{x}_s)-\lambda\right) \d s -\int^T_t Z^{x}_s \d W_s,
 &  \forall 0\leq t \leq T.
\end{array}
\end{equation}
It is well known that BSDEs give a probabilistic representation of semi-linear parabolic PDEs, so \eqref{cvgence-intro} also gives the long time behavior of the solution to the following HJB equation
\begin{equation*}
    \left\{\begin{array}{lr}
        \partial_t  u^T(t,x) + \mathcal{L}u^T(t,x) +f\left((x,\nabla_x u^T(t,x)\sigma(x)\right)=0 &  t\in\seg*{0,T},x\in\R^d,\\
        u^T(T,.)=g, & 
    \end{array}\right.
\end{equation*} where $\mathcal{L}$ is the infinitesimal generator of the semi-group associated to the SDE $X$. 
A few years later, Hu and Lemmonier in \cite{Hu-Lemmonier} extended the previous results for the weak dissipative assumption in finite dimension, assuming now that $\sigma$ is no longer constant.

\par  The novelty of the present work comes from the fact that we now assume a distribution dependency for the generator $f$ and $X$ is now the solution to a MV-SDE. This difference leads to additional difficulties. In particular, due to the distribution dependency, $X$ is no longer Markovian but the couple $(X,\loi X)$ is. Taking into account this, we start by decoupling the MV-SDE \eqref{Intro: EDS MKV} as follows: for $x\in\R^d$ and $\theta$ an $\L{2}(\Omega)$ r.v., we consider
\begin{equation*}
    \d X^{x,\loi\theta}_t=b(X^{x,\loi\theta}_t,\loi{X^{\loi\theta}_t})\d t +\sigma(X^{x,\loi\theta}_t,\loi{X^\loi\theta_t})\d W_t,\quad\quad\quad X^{x,\loi\theta}_0=x,
\end{equation*}
where $X^\loi\theta$ solves \eqref{Intro: EDS MKV} with starting point $\theta$. We stress that this new equation is not a McKean-Vlasov SDE but is close to it since $X^{\theta,\loi\theta}$ gives a solution to \eqref{Intro: EDS MKV}. Consequently, instead of \eqref{Intro: EBSDE} we will study the following equation:
\begin{equation}\label{Intro: decoupled EBSDE}
Y^{x,\loi\theta}_t=Y^{x,\loi\theta}_T+\int^T_t\left(f(X^{x,\loi\theta}_s,\loi{X^\loi\theta_s},Z^{x,\loi\theta}_s)-\lambda\right)\d s -\int^T_t Z^{x,\loi\theta}_s\d W_s.
\end{equation}

The first aim of the paper is to establish some existence and uniqueness results for \eqref{Intro: decoupled EBSDE} under two different sets of assumptions: a strong dissipativity assumption on one side and a weak dissipativity assumption with a distribution free $\sigma$ on the other side (see (\nameref{SDE: H1}) and (\nameref{SDE: H2-2})). Then we are able to use this EBSDE in order to study the long time behavior of \eqref{BSDE classic} in our new framework, that is to say, 
\begin{equation}\label{BSDE}
\begin{array}{lr}\displaystyle
    Y^{T,x,{\loi\theta}}_T=g(X^{T,x,{\loi\theta}}_T) +\int^T_t f(X^{x,\loi\theta}_s,\loi{X^{\loi\theta}_s},Z^{T,x,{\loi\theta}}_s) \d s -\int^T_t Z^{T,x,{\loi\theta}}_s \d W_s,
 &  \forall t\in\seg*{0,T}.
\end{array}
\end{equation}
In particular, we obtain an exponential convergence of $Y_0^{T,x,\loi\theta}-\lambda T - Y_0^{x,\loi\theta}$ toward a constant as in \eqref{intro:expo decroissance}. 
In an equivalent way, this result allows to obtain the long-time behavior of the solution of a McKean-Vlasov HJB equation given by
\begin{equation*}
    \left\{\begin{array}{ll}
        \partial_t  u^T(t,x,\loi\theta) + \mathcal{L}u^T(t,x,\loi\theta) +f\left(x,\loi\theta,\nabla_x u^T(t,x,\loi\theta)\sigma(x,\loi\theta)\right)+\mathcal{I}(u^T(t,x,\loi\theta))=0, \quad x \in \R^d, t\geq0,\, \theta \in\L{2}(\Omega)&\\
        u^T(T,x,\loi\theta)=g(x,\loi\theta),\quad x \in \R^d,\, \theta \in\L{2}(\Omega) & 
    \end{array}\right.
\end{equation*}
where $$\mathcal{I}(u^T(t,x,\loi\theta)):=\int_{\R^d}\left( \partial_{\mu}u^T(t,x,\loi\theta)(y)+\partial_y\partial_\mu u^T(t,x,\loi\theta)(y)\mu(\d y)\right),\quad \mu = [\theta]$$
and $\partial_\mu\phi(\mu)$ denotes the Lions' derivative of a function $\phi$ with respect to its distribution valued variable, see e.g. \cite[Section 5.2]{Delarue-Carmona_Book} for further details.

The EBSDE \eqref{Intro: decoupled EBSDE} can also be used to solve an ergodic optimal control problem
where we optimize an ergodic reward with a controlled MV-SDE,
    \begin{equation}\label{Intro: EDS control decoupled}
        \d X^{x,\loi\theta,\bm a}_t =\left(b(X^{x,\loi\theta,\bm a}_t,\loi{X^{\loi\theta,0}_t})+\sigma(X^{x,\loi\theta,\bm a}_t,\loi{X^{\loi\theta,0}_t})Ra_t\right)\d t +\sigma(X^{x,\loi\theta,\bm a}_t,\loi{X^{\loi\theta,0}_t})\d W_t
    \end{equation}
over a set of admissible controls $\bm{a}=(a_s)_{s \geq 0}$ taking values in a bounded subset of some $\R^k$.
This stochastic control problem is also linked to the following ergodic McKean-Vlasov Hamilton-Jacobi-Bellman equation (HJB for short) given by:
\begin{equation*}
    \begin{array}{ll} 
\mathcal{L}u(x,\loi\theta)+f(x,\loi\theta,\nabla_xu(x,\loi\theta))+\E_{U\sim\loi\theta}\seg*{\partial_{\loi\theta}u(t,x,\loi\theta)(U)}+\E_{U\sim\loi\theta}\seg*{\partial_u\partial_\loi\theta u(t,x,\loi\theta)(U)}=\lambda,\quad x \in \R^d,\,  t\geq0,\, \theta\in\L{2}(\Omega),
    \end{array}
\end{equation*} 
  where $\mathcal{L}$ is the infinitesimal generator defined by 
$$\mathcal{L}\phi(x,\loi\theta) :=b(x,\loi\theta)\nabla_x\phi(x,\loi\theta) + \frac{1}{2}\Tr{\sigma\sigma^\top(x,\loi\theta)\nabla^2_x\phi(x,\loi\theta)},
$$  and $f$ is the Hamiltonian associated to the control problem. 
We stress the fact that the problem considered do not allow to control the law in the decoupled equation \eqref{Intro: EDS control decoupled}. Consequently, the studied framework is not the classical McKean-Vlasov ergodic control problem and it explains why we will talk about a \textit{partial} McKean-Vlasov ergodic control problem. If we see a McKean-Vlasov SDE as the limit of an interacting particle system, we consider a problem were we control only one particle, which explains why this control will have no impact on the distribution, whereas the classical McKean-Vlasov ergodic control problem is looking for an optimal control identically applied to all particles at the same time.  

\textbf{Comparison with other works.}
Up to our knowledge, there is no work dealing with distribution dependent ergodic BSDEs. 
However several paper investigate some questions related to close models, as mean field games (MFG for short) or McKean-Vlasov control problems, by using other tools. Let us detail here some of them.

A recent paper of Fuhrman and Rudà, see \cite{Ruda-Furhman}, obtain result on ergodic control problem on the Wasserstein space by using a control approach for strongly dissipative and controlled McKean-Vlasov's SDE. Namely, the following SDE, 
\begin{equation*}
    \d X_t=b(X_t,\loi{X_t},a_t)\d t+\sigma(X_t,\loi{X_t},a_t)\d W_t.
\end{equation*}  
Their main result is the existence and (partial) uniqueness of a viscosity solution to 
an ergodic McKean-Vlasov HJB equation on the Wasserstein space. They link it with an optimal ergodic McKean-Vlasov control problem and show that $\lambda$ (the constant that is part of the solution of the ergodic PDE) gives the first order term when we look at the long time behavior of a finite horizon McKean-Vlasov control problem, as in \autoref{LTB: LTB 1}. Nevertheless, compared to our work, they only consider a strong dissipativity condition, and ask their generator $f$ to be bounded and Lipschitz with respect to $x$ where we have instead a polynomial growth and a locally Hölder condition. Moreover, they do not have the precise exponential convergence \eqref{intro:expo decroissance}. Same kind of results, under some close assumptions, were also obtained by Bao and Tang in \cite{Bao-Tang}.

Bayraktar and Jian \cite{Bayraktar-Jian} developed ergodicity and turnpike properties, i.e. exponential convergence toward an ergodic problem, for linear-quadratic Mean Field control problems.  Apart from the linear-quadratic framework, the main difference on the MV-SDE is that theirs only depend on the law through the expectation. 
\par Cecchin, Conforti, Durmus and Eichinger in \cite{conforti} establish several turnpike properties, under different smoothness assumptions and weak dissipativity assumption, of the long time behavior of solutions of mean field PDE systems related to MFG. These results generalize some previous papers \cite{Cardaliaguet-12,Cardaliaguet-13,Cardaliaguet-19} restricted to compact domains of $\R^d$ and assuming a Lasry-Lions monotonicity assumption.

\textbf{Organization.} The paper is organized as follows. In Section \ref{SDE}, we state some useful estimates on McKean-Vlasov SDEs under two different sets of assumptions (see (\nameref{SDE: H1}) and (\nameref{SDE: H2-2})). Section \ref{EBSDE} is devoted to the proof of an existence and uniqueness result for the EBSDE \eqref{Intro: decoupled EBSDE} given by \autoref{EBSDE: Existence and uniqueness}. In Section \ref{LTB}, we study the long time behavior of the finite horizon distribution dependent BSDE \eqref{BSDE} thanks to our EBSDE solution: see \autoref{LTB: LTB 1}, \autoref{LTB: LTB 2} and \autoref{LTB: LTB 3}. In Section \ref{OCP}, we apply results of Sections \ref{EBSDE} and \ref{LTB} to a partial McKean-Vlasov ergodic control problem and the long-time behavior of a partial McKean-Vlasov control problem. The remaining of the introduction is devoted to notations.

\textbf{Notations.} Let a given filtered complete probability space $(\Omega,\mathbb{F}:=(\mathbb{F}_t)_{t\in\R_+},\P)$ which supports a $d-$Brownian motion denoted $W$. Let also, see \cite{Ruda}, $\mathcal{G}$ be a sub-$\sigma$-algebra of $\mathbb{{F}}$ such that $\mathcal{G}$ is generated by $U\sim Unif(\seg*{0,1})$ and is independent of $W$. Consequently, we define the filtration $\F =(\F_t)_{t\in\R_+}$ by $\F_t=\sigma(\mathcal{G}\cup\mathbb{F}_t)$, for all $t \geq 0$. We also set the following:
\begin{enumerate}
    \item \underline{Spaces}: For all $p\geq 1$, we set:
    \begin{enumerate}
    \item $\L{p}(\Omega,\R^d)$ the set of $\R^d$-valued random variable admitting finite moment of order $p$ endowed with the norm \begin{equation*}
        \norm*{\theta}_p^p:=\norm*{\theta}_{\L{p}(\Omega,\R^d)}^p=\E\seg*{\abs*{\theta}^p}.
    \end{equation*}
        \item $\L{2}_\Pcal(\Omega; \L{2}(\seg*{0,T}; \R^d))$ the set of predictable processes $X$ on $\seg*{0,T}$ such that 
        \begin{equation*}
            \norm*{X}^2_{\L{2}}:=\norm*{X}^2_{\L{2}_\Pcal(\seg*{0,T}; E)}=\E\seg*{\int^T_0 \abs*{X_s}^2 \d s}.
        \end{equation*}
        \item $\L{2}_\text{loc}(\Omega; \L{2}(\R_+; \R^d))$ the set of predictable processes $X\in\L{2}_\Pcal(\Omega; \L{2}(\seg*{0,T}; \R^d))$ for any $T>0$.
        \item $\S_T^2$ the set of continuous, $\F-$adapted processes $X$ on $\seg*{0,T}$ with value in $\R^2$ satisfying \begin{equation*}
        \norm*{X}_{\S_T^p}:=\E\seg*{\sup{t\in\seg*{0,T}}{\abs*{X_s}^2}}^{1/2}<\infty. 
    \end{equation*}
    \item  $\Pcal_p(\R^d)$ be the set of probabilities measures  having finite $p-$th moment over $\R^d$ endowed with the $p$-Wasserstein distance defined by: for $\nu,\mu\in\Pcal_p(\R^d)$
    \begin{equation*}
        \mathcal{W}_p(\nu,\mu)=\inf{\Pi\in\Pcal_{\nu,\mu}}\left(\int_{\R^d\times\R^d} \abs*{x-y}^p\Pi(\d x,\d y) \right)^{1/p},
    \end{equation*}
    where $\Pcal_{\nu,\mu}$ is the set of all couplings of probability distributions $\nu$ and $\mu$, i.e. the set of all probability distributions on $\R^d \times \R^d$ with marginals $\nu,\mu$.
    \item $\C^{k,m}(\R^d\times\Pcal(\R^d);\R)$ be the set of $k$ times continuously differentiable w.r.t the first variable and $m$ times differentiable w.r.t to the second one.
    \end{enumerate}
    We also denote $\nabla_xf$ the gradient of $f$ w.r.t $x$ which will be seen as an element of $\R^{1\times d}$ 
    
    \item \underline{Norms:} We naturally endow $\R^d$ with the Euclidean norm denoted $\abs*{\cdot}$ and for a $\sigma\in\R^{d\times d}$ we denote $\norm*{\sigma}$ the Frobenius norm. Namely, $\norm*{\sigma}^2:=\Tr{\sigma\sigma^\top }$ 
    \item \underline{Constants:} For $\varphi:\R^d\times\Pcal_p(\R^d)\mapsto H$ a Lipschitz function w.r.t both variables and $H$ an Hilbert space, we set, up to transformations, $K^\varphi_x$ (resp. $K^\varphi_\LL$) the Lipschitz constant for the space variable (resp. for the distribution variable). The constants $C$ will be denoted with a index $k$ to enlighten the dependency w.r.t $k$. The constants $C$ might change from one line to another. 
\end{enumerate}

\section{The forward McKean-Vlasov SDE}\label{SDE}
In this section, we establish several exponential convergence results for McKean-Vlasov SDEs under two different sets of assumptions (\nameref{SDE: H1}) and (\nameref{SDE: H2-2}). Let us consider the following McKean-Vlasov SDE:
\begin{equation}\label{SDE: MKV SDE}\tag{MV-SDE}\left\{\begin{array}{lr}
     \d X^{s,\loi\theta}_t =b(t,X^{s,\loi\theta}_t,\loi{X^{s,\loi\theta}_t})\d t +\sigma(X^{s,\loi\theta}_t,\loi{X^{s,\loi\theta}_t})\d W_t,& t\geq s,   \\
     X_s^{s,\loi\theta}=\theta\in{ \L{p}}(\Omega,\R^d) \text{ is $\mathcal{G}$-measurable },&
\end{array}\right.
\end{equation}
where $p\geq 2$, $b$ and $\sigma$ satisfy the following assumptions.
\begin{hyp}\label{SDE: Assumption existence uniqueness SDE MKV}
$b:\R_+\times\R^d\times\Pcal_{ p}\mapsto\R^d$ and $\sigma:\R^d\times\Pcal_{ p}\mapsto\R^{d\times d}$ are two measurable functions such that
    \begin{enumerate}
        \item for all $t\geq 0$ $b(t,\cdot,\cdot)$ is uniformly Lipschitz,
        \item for all $x \in \R^d$ and $\mu \in \Pcal_{ p}$, $t \mapsto b(t,x,\mu)$ is bounded,
        \item $\sigma$ is Lipschitz.
    \end{enumerate}
\end{hyp}
We emphasize that \autoref{SDE: Assumption existence uniqueness SDE MKV} is only here for the well-posedness of \eqref{SDE: MKV SDE} and this set of assumptions is supposed to be fulfilled throughout the article. If we consider a r.v. $\theta' \in {  \L{p}}(\Omega,\R^d)$ that is not $\mathcal{G}$-measurable, then we can construct a $\mathcal{G}$-measurable r.v. $\theta_\mu$ that only depends on $\mu:=\loi{\theta'}$ and such that $\loi\theta=\mu=\loi{\theta'}$ since the probability space $(\Omega,\mathcal{G},\P)$ is atomless: see e.g. \cite[Page 352]{Delarue-Carmona_Book}. In this case, we define it by a slight abuse of the notation $X^{s,\loi{\theta'}}:=X^{s,\loi{\theta_{\mu}}}$. Moreover, we can remark that if $\theta'$ is $\mathcal{F}_s$ -measurable but not $\mathcal{G}$ -measurable, then $X^{s,\loi{\theta'}}$ is not equal to the solution of \eqref{SDE: MKV SDE} starting from $\theta'$ at time $s$ but these two processes are equal in distribution.

For the latter use, we also introduce the \textit{decoupled} McKean-Vlasov SDE as 
\begin{equation}\label{SDE: Decoupled MKV SDE}\tag{Decoupled SDE}\left\{\begin{array}{lr}
      \d X^{s,x,\loi{\theta}}_t=b(t,X^{s,x,\loi{\theta}}_t,\loi{X^{s,\loi\theta}_t})\d t +\sigma(X^{s,x,\loi{\theta}}_t,\loi{X^{s,\loi\theta}_t})\d W_t, & t\geq s,  \\
      X_s^{s,x,\loi\theta}=x\in\R^d,&
   \end{array}\right.
\end{equation}where $x\in\R^d,\theta\in{  \L{p}}(\Omega,\R^d)$, for $p\geq 2$ and $X^{s,\loi\theta}$ is the solution to \eqref{SDE: MKV SDE}. Let us remark that \eqref{SDE: Decoupled MKV SDE} is \underline{not} a McKean-Vlasov SDE but we can easily return to \eqref{SDE: MKV SDE} by evaluating the function $X^{s,\cdot,[\theta]}$ in $\theta$, i.e. $X^{s, \theta,\loi\theta}$ is solution to \eqref{SDE: MKV SDE}. Finally, we set $X^\loi\theta:=X^{0,\loi\theta}$ and $X^{x,\loi\theta}:=X^{0,x,\loi\theta}$ to simplify the notation.
\par Under \autoref{SDE: Assumption existence uniqueness SDE MKV} the following result holds: See \cite[Proposition 2.1]{Ruda-Furhman}, with $\alpha=0$ in the cited reference, for the first point; while the second point is a consequence of standard results for classical SDEs. 
\begin{thm}\label{SDE: thm: Existence uniqueness MKV}
\begin{enumerate}
    \item For all $s \geq 0$ and $\theta\in{  \L{p}}(\Omega;\R^d)$, there exists a unique solution $(X^{s,\loi\theta}_t)_{t\geq s}$ to \eqref{SDE: MKV SDE} such that $(X^{s,\loi\theta}_t)_{t\in\seg*{s,T}}\in\S_T^{{ p}}$ for all $T \geq s$.
    \item For all $s \geq 0$, $\theta\in{  \L{p}}(\Omega;\R^d)$ and $x \in \mathbb{R^d}$, there exists a unique solution $(X^{s,x,\loi\theta}_t)_{t\geq s}$ to \eqref{SDE: MKV SDE} such that $(X^{s,x,\loi\theta}_t)_{t\in\seg*{s,T}}\in\S_T^{{ p}}$ for all $T \geq s$.
\end{enumerate}
\end{thm}
We also recall in the next proposition a flow property for McKean-Vlasov SDEs that will be useful throughout the paper: see for instance \cite[Proposition 2.1]{McMurray-Crisan} and \cite[Lemma 3.1]{PhamDPP}.
\begin{prop}
    For all $p\geq 2$, $0\leq s\leq t$, $x\in\R^d$ and $\theta\in{ \L{p}}(\Omega;\R^d)$ 
    \begin{equation*}
        \left(X^{0,x,\loi{\theta}}_{t},\loi{X^{0,\loi\theta}_{t}}\right)=\left(X^{s,X^{x,\loi\theta}_s,\loi{X^{\loi\theta}_s}}_t,\loi{X^{s,\loi{X^{\loi\theta}_s}}_t}\right), \quad\quad\quad\P-a.s.
    \end{equation*}
\end{prop}

We also define, for all $t\geq 0$, the non-linear and non-homogeneous semigroup, see again \cite[Proposition 2.1]{McMurray-Crisan} and \cite[Lemma 3.1]{PhamDPP}, associated to $X^{x,\loi\theta}$, $x\in\R^d$ and $\theta\in\L{p}(\Omega;\R^d)$,
\begin{equation}\label{SDE: Semigroup X decoupled}
    \Pcal_t\seg*{\phi}(x,\loi\theta)=\E\seg*{\phi(X^{x,\loi\theta}_t,\loi{X^{\loi\theta}_t})}.
\end{equation}

\subsection{The strong dissipative framework}\label{SDE: Strong dissip}
In this subsection, we prove some exponential convergence results for \eqref{SDE: MKV SDE} and \eqref{SDE: Decoupled MKV SDE} under the following strong dissipative assumptions on $b$ and $\sigma$.
\begin{hyp*}[$\H_{SDE}1$]\label{SDE: H1}Let $p\geq 2$.
    \begin{enumerate}
        \item $b$ is $\nu$-L-dissipative. Namely, there exists $\nu>0$ such that for all $U,\,U'\in\L{2}(\Omega;\R^d)$, $t\geq 0$, 
        $$
        \E\seg*{\braket*{U-U'}{b(t,U,\loi{U})-b(t,U',\loi{U'})}}\leq -\nu\E\seg*{\abs*{U-U'}^2},
        $$
       \item $b$ is $\eta$-dissipative w.r.t $x$ and is $K^b_\LL$-Lipschitz w.r.t the distribution. Namely, there exists $\eta>0$ and $K^b_\LL\geq0$ such that for all $x,x'\in\R^d$ ,$\theta,\theta'\in\L{2}(\Omega;\R^d)$ and $t\geq 0$,
       $$
       \braket*{x-x'}{b(t,x,\loi\theta)-b(t,x',\loi{\theta})}\leq -\eta\abs*{x-x'}^2,
       $$
       $$
      \abs*{ b(t,x,\loi\theta)-b(t,x,\loi{\theta'})}\leq K^b_\LL\W(\loi{\theta},\loi{\theta'}).
       $$
        \item $\sigma$ is bounded and there exists $K^\sigma_x,K^\sigma_\LL>0$ such that for all $x,\,x'\in\R^d,\,\theta,\,\theta'\in\L{2}(\Omega;\R^d)$, $t\geq 0$, 
        $$
        \frac{1}{2}\norm*{\sigma(x,\loi\theta)-\sigma(x',\loi{\theta'})}^2\leq K^\sigma_x\abs*{x-y}^2 +K^\sigma_\LL\W(\loi\theta,\loi{\theta'})^2, 
        $$
        \item $\nu>K^\sigma_x+K^\sigma_\LL$ and $\eta \geq \nu$.
    \end{enumerate}\end{hyp*}
    \begin{rmq}
    \label{rem: assumpt HSDE1}
        \begin{enumerate}
            \item (\nameref{SDE: H1})-1. depends upon the choice of the probability space only through its atomless property. Indeed, when a probability space is atomless, we know that for any joint distribution $\pi\in\Pcal_2(\R^d\times\R^d)$ there exists $(U,U')\sim\pi$, see e.g. \cite[Page 352]{Delarue-Carmona_Book}. Additionally, the formulation 'L-dissipative' is here to echo with the L-monotonicity, see, for instance, \cite[Definition 3.31]{Delarue-Carmona_Book}, and the classical dissipativity notion.
            \item (\nameref{SDE: H1})-3. implies that $\sigma$ is $\sqrt{2K^\sigma_x}$-Lipschitz w.r.t. $x$ and $\sqrt{2K^\sigma_\LL}$-Lipschitz w.r.t. the distribution.
            \item There are some entanglements between the assumptions (\nameref{SDE: H1})-1. and (\nameref{SDE: H1})-2. Indeed, by doing same computations as in Example \ref{SDE: H1: example 1}, it is easy to see that, if $\eta >K^b_\LL$ then (\nameref{SDE: H1})-2. implies (\nameref{SDE: H1})-1. with $\nu := \eta -K^b_\LL$. Nevertheless, this upper-bound could be really rough. In order to  illustrate the difference between these two assumptions, we give two examples below.
        \end{enumerate}
        
    \end{rmq}
        \begin{ex}\label{SDE: H1: example 1}
            Let $X^{\loi\theta}$ be the solution to the following Ornstein-Ulhenbeck MV-SDE
            $$
            \d X^\loi\theta_t =b(X^\loi\theta_t,\loi{X^\loi\theta_t})\d t +\d W_t:= -\eta X^\loi\theta_t \d t + K^b_\LL \E\seg*{X^\loi\theta_t}\d t +\d W_t.
            $$
        We can compute that, for all $U,\,U'\in\L{2}(\Omega;\R^d)$, 
        \begin{align}
        \braket*{U-U'}{b(U,\loi{U})-b(U',\loi{U'})}&=-\eta\abs*{U-U'}^2 + K^b_\LL \braket*{ U-U'}{\E\seg*{U-U'}},\notag\\
        \E\seg*{\braket*{U-U'}{b(U,\loi{U})-b(U',\loi{U'})}}&=-\eta\E\seg*{\abs*{U-U'}^2} + K^b_\LL \left| \E\seg*{U-U'}\right|^2
        \leq -(\eta-K^b_\LL)\E\seg*{\abs*{U-U'}^2}
        .\label{SDE: H1: ex1: eq 2}
        \end{align}
        So, $b$ is $(\eta-K^b_\LL)$-L-dissipative as soon as $\eta >K^b_\LL$.
        Now let us consider ${X'}^\loi\theta$ the solution to the following Ornstein-Ulhenbeck MV SDE
        $$
            \d {X'_t}^\loi\theta =b({X'_t}^\loi\theta,\loi{{X'_t}^\loi\theta})\d t +\d W_t:= -\eta {X'_t}^\loi\theta \d t - K^b_\LL \E\seg*{{X'_t}^\loi\theta}\d t +\d W_t.
        $$
        Doing same computations as before leads to 
        \begin{align}
            \E\seg*{\braket*{U-U'}{b(U,\loi{U})-b(U',\loi{U'})}}&=-\eta\E\seg*{\abs*{U-U'}^2} - K^b_\LL \left|\E\seg*{U-U'}\right|^2 \leq -\eta\E\seg*{\abs{U-U'}^2},\label{SDE: H1: ex1: eq 1}
        \end{align}
        and so, $b$ is always $\eta$-L-dissipative.
        If we compare these two cases, we can see that it is not restrictive to assume that $\eta \geq \nu$ and we do not penalize the dissipativity constant $\eta$ if $K^b_\LL$ has a minus in front, whereas if we only use the Lipschitz property, we lose the fact that the distribution term could 'helps'. 
        \end{ex}
        {
        \begin{ex}
            Let now $X^{\loi\theta}$ be the solution to the following MV-SDE
            $$ \d X^\loi\theta_t= b(X^\loi\theta_t,\loi{X^\loi\theta_t})\d t +\d W_t,
            $$
            where $b$ is defined for any $(x,\mu)\in\R^d\times \Pcal_2$ as $b(x,\mu)=\int_{\R^d} h_1(x-y)\d \mu(y)+h_2(x)=\mathbb{E}[h_1(x-X)]+h_2(x)$ with $h_1: \R^d \to \R^d$, $h_2: \R^d \to \R^d$  and $X\sim \mu$. We assume that $h_1$ is $0$-dissipative (i.e. $-h_1$ is monotone), $K^b_\LL$-Lipschitz and $h_2$ is $\eta$-dissipative (with $\eta>0$). These assumptions on $h_1$ and $h_2$ imply, by some direct computations, that $b$ is $\eta$-dissipative and $K^b_\LL$-Lipschitz. Thus, if $\eta > K^b_\LL$ we can obtain that $b$ is $(\eta-K^b_\LL)$-$L$-dissipative. Nevertheless, it is also possible to show that $b$ is $\eta$-$L$-dissipative by assuming that $h_1$ is odd. Indeed we have, for $X \sim \mu$, $X' \sim \mu'$ and $(\tilde{X},\tilde{X'})$ an independent copy of $(X,X')$,
            \begin{align*}
                \E\seg*{\braket*{X-X'}{b(X,\mu)-b(X',\mu')}}&=\E\seg*{\braket*{X-X'}{h_1(X-\tilde{X})-h_1(X'-\tilde{X'})}}+\E\seg*{\braket*{X-X'}{h_2(X)-h_2(X')}}.
            \end{align*} 
            Then, the $\eta$-dissipativity of $h_2$ gives
            $$\E\seg*{\braket*{X-X'}{h_2(X)-h_2(X')}} \leq -\eta \E[|X-X'|^2].$$
            Moreover, since $h_1$ is odd and $(X,X',\tilde{X},\tilde{X'})$ has the same law as $(\tilde{X},\tilde{X'},X,X')$, we get
            \begin{align*}
                \E\seg*{\braket*{X-X'}{h_1(X-\tilde{X})-h_1(X'-\tilde{X'})}} =& \E\seg*{\braket*{\tilde{X}-\tilde{X'}}{h_1(\tilde{X}-{X})-h_1(\tilde{X'}-{X'})}} = -\E\seg*{\braket*{\tilde{X}-\tilde{X'}}{h_1(X-\tilde{X})-h_1(X'-\tilde{X'})}}
            \end{align*}
            which implies that
            \begin{align*}
                \E\seg*{\braket*{X-X'}{h_1(X-\tilde{X})-h_1(X'-\tilde{X'})}}  =& \frac12 \left( \E\seg*{\braket*{X-X'}{h_1(X-\tilde{X})-h_1(X'-\tilde{X'})}} - \E\seg*{\braket*{\tilde{X}-\tilde{X'}}{h_1(X-\tilde{X})-h_1(X'-\tilde{X'})}} \right)\\
                =& \frac12 \E\seg*{\braket*{X-\tilde{X}-(X'-\tilde{X'})}{h_1(X-\tilde{X})-h_1(X'-\tilde{X'})}} \leq \E [|X-\tilde{X}-X'+\tilde{X'}|^2] \leq 0
            \end{align*}
            by using the $0$-dissipativity of $h_1$.
        \end{ex}}
        
        Before giving the exponential convergence results, we recall some useful uniform bounds on moments of solutions to \eqref{SDE: MKV SDE} and \eqref{SDE: Decoupled MKV SDE}.

    \begin{prop}\label{SDE: H1: prop: Estim unif T} Under (\nameref{SDE: H1}), for all $p\geq 2$, there exists $C=C(p,\nu,K^\sigma_x,K^\sigma_\LL,K^b_\LL,\sigma(0,\delta_0),{ \|\sigma\|_{\infty} \1{p>2}},\sup{t}\abs*{b(t,0,\delta_0)})$ such that, for all $x \in \R^d$, $\theta \in {  \L{p}}(\Omega;\R^d)$ 
    \begin{align}\label{SDE: Estim unif T coupled}
        \sup{t\in\R_+}\E\seg*{\abs*{X^\loi\theta_t}^p} \leq& C(1+\norm*{\theta}_p^p),\\
        \sup{t\in\R_+}\E\seg*{\abs*{X^{x,\loi\theta}_t}^p}\leq& C(1+\abs*{x}^p+\norm*{\theta}^p_{ p}). \label{SDE: Estim unif T decoupled}
    \end{align}
    \end{prop}
    \begin{proof}By classical results on MV-SDEs, see for instance \cite[Proposition 2.1]{McMurray-Crisan} we know that, there exists $C_t \geq 0$ that might depend on time (in an increasing way) such that
    \begin{equation}\label{SDE: borne}
    \E\seg*{\abs*{X^\loi\theta_t}^p} +\E\seg*{\abs*{X^{x,\loi\theta}_t}^p}\leq  C_t(1+\abs*{x}^p+\norm*{\theta}^p_{ p}).
    \end{equation}
    So it remains to prove that this upper bound is uniform in time. Let us prove first \eqref{SDE: Estim unif T coupled}. {{We start by proving it for $p=2$.}}
    By applying Itô's formula to $\abs*{X^\loi\theta_t}^2$ and Young's inequality we have, for all $\varepsilon>0$,
    \begin{align*}
        \d \abs*{X^\loi\theta_t}^2 &=2\braket*{X^{\loi\theta}_t}{b(t,X^{\loi\theta}_t,\loi{X^{\loi\theta}_t})}\d t +\norm*{\sigma(X^{\loi\theta}_t,\loi{X^\loi\theta_t})}^2\d t +2\braket*{X^\loi\theta_t}{\sigma(X^\loi\theta_t,\loi{X^\loi\theta_t})\d W_t}\\
        &\leq 2\braket*{X^{\loi\theta}_t}{b(t,X^{\loi\theta}_t,\loi{X^{\loi\theta}_t})-b(t,0,\delta_0)}\d t +(1+\varepsilon)\norm*{\sigma(X^{\loi\theta}_t,\loi{X^\loi\theta_t})-\sigma(0,\delta_0)}^2\d t \\
        &\quad + 2\sup{t}\abs*{b(t,0,\delta_0)}\abs*{X^\loi\theta_t}\d t+ (1+\varepsilon^{-1})\norm*{\sigma(0,\delta_0)}^2\d t +2\braket*{X^\loi\theta_t}{\sigma(X^\loi\theta_t,\loi{X^\loi\theta_t})\d W_t}.
    \end{align*}
    Thanks to (\nameref{SDE: H1})-4, we can set $0<\varepsilon < (\nu - (K^\sigma_x+K^\sigma_\LL))(1+K^\sigma_x+K^\sigma_\LL)^{-1}$ and 
     $0<\lambda<\nu-(\varepsilon+({K^\sigma_x}+K^\sigma_\LL)(1+\varepsilon))$. Thus, by Young's inequality we get, 
    \begin{align}\label{SDE: H1: proof: exp + cst}
         \d e^{2\lambda t}\abs*{X^\loi\theta_t}^2 &\leq  \left(  2(\lambda+\varepsilon)\abs*{X^\loi\theta_t}^2+ 2\braket*{X^{\loi\theta}_t}{b(t,X^{\loi\theta}_t,\loi{X^{\loi\theta}_t})-b(t,0,\delta_0)}\d t +(1+\varepsilon)\norm*{\sigma(X^{\loi\theta}_t,\loi{X^\loi\theta_t})-\sigma(0,\delta_0)}^2  \right)e^{2\lambda t}\d t \notag\\&\quad+(\varepsilon^{-1}\sup{t}\abs*{b(t,0,\delta_0)}^2+(1+\varepsilon^{-1})\norm*{\sigma(0,\delta_0)}^2)e^{2\lambda t}\d t +2e^{2\lambda t}\braket*{X^\loi\theta_t}{\sigma(X^\loi\theta_t,\loi{X^\loi\theta_t})\d W_t}.
    \end{align}
    Due to \eqref{SDE: borne}, the last term of the previous inequality is a martingale, we finally obtain by (\nameref{SDE: H1})-1-2 and the definition of $\lambda$,
    \begin{align*}
    e^{2\lambda t}\E\abs*{X^\loi\theta_t}^2&\leq \norm*{\theta}^2_2+2(\lambda-\nu+(1+\varepsilon)K^\sigma_x+\varepsilon)\int^t_0  e^{2\lambda s}\E\abs*{X^\loi\theta_s}^2\d s  \\
    &\quad+ 2(1+\varepsilon)K^{\sigma}_\LL\int^t_0  e^{2\lambda s}\W(\loi{X^\loi\theta_s},\delta_0)^2\d s + \frac{C}{2\lambda}(e^{2\lambda t}-1)\\
    &\leq \norm*{\theta}_2^2-2(\nu-((1+\varepsilon)(K^\sigma_\LL+K^\sigma_x)+\varepsilon)-\lambda)\int^t_0  e^{2\lambda s}\E\abs*{X^\loi\theta_s}^2\d s+\frac{C}{2\lambda}e^{2\lambda t}
    \leq \norm*{\theta}^2_2 + \frac{C}{2\lambda}e^{2\lambda t},
    \end{align*}
    which gives us \eqref{SDE: Estim unif T coupled}.
    { Let us now prove it for any $p\geq 2$.  By applying Itô's formula to $\abs*{X^\loi\theta}^p$, we obtain 
    \begin{align*}
        \d \abs*{X_t}^p = p\abs*{X^\loi\theta_t}^{p-2}\braket*{X^\loi\theta_t}{b(t,X^\loi\theta_t,\loi{X^\loi\theta_t})}\d t + \frac{p(p-1)}{2}\abs*{X^\loi\theta_t}^{p-2}\norm*{\sigma(X^\loi\theta_t,\loi{X^\loi\theta_t})}^2 \d t+ \abs*{X^\loi\theta}^{p-2}\braket*{X^\loi\theta_t}{\sigma(X^\loi\theta_t,\loi{X^\loi\theta_t})\d W_t}.
    \end{align*}
    We set $\d M_t:= \abs*{X^\loi\theta_t}^{p-2}\braket*{X^\loi\theta_t}{\sigma(X^\loi\theta_t,\loi{X^\loi\theta_t})\d W_t}$ which is a martingale due to \eqref{SDE: borne}. Thus, by using (\nameref{SDE: H1}) and the estimate for $p=2$, we obtain
    \begin{align*}
        \d \abs*{X^\loi\theta_t}^p &\leq -\eta p\abs*{X^\loi\theta_t}^{p}\d t + p\abs*{X^\loi\theta_t}^{p-1}\abs*{b(t,0,\delta_0)}\d t +  p\abs*{X^\loi\theta_t}^{p-1}K^b_\LL\W(\loi{X^\loi\theta_t},\delta_0)+ \frac{p(p-1)}{2}\abs*{X^\loi\theta_t}^{p-2}\norm*{\sigma}_\infty^2 \d t+ \d M_t\\
        &\leq   -\eta p\abs*{X^\loi\theta_t}^{p}\d t + p\abs*{X^\loi\theta_t}^{p-1}\abs*{b(t,0,\delta_0)}\d t +  p\abs*{X^\loi\theta_t}^{p-1}K^b_\LL C(1+\norm*{\theta}_2)\d t+ \frac{p(p-1)}{2}\abs*{X^\loi\theta_t}^{p-2}\norm*{\sigma}_\infty^2 \d t+ \d M_t.
    \end{align*}
    { Now, using Young's inequality with $0<\varepsilon<p\eta/3$, gives
    \begin{align*}
        \d \abs*{X^\loi\theta_t}^p & \leq (-p\eta+3\varepsilon) \abs*{X_t^\loi\theta}^p \d t + \left(\varepsilon^{-(p-1)}(p\abs*{b(t,0,\delta_0)})^p+\varepsilon^{-(p-1)}(pK^b_\LL C(1+\norm*{\theta}_2))^p+\varepsilon^{-\frac{p-2}{2}}\left(\frac{p(p-1)}{2} \norm*{\sigma}_\infty^2 \right)^{p/2}\right) \d t +\d M_t
    \end{align*}
    If we set $0<\lambda = \eta - \frac{3\varepsilon}{p}$, then applying Itô's formula to $\d e^{p\lambda t}\abs*{X^\loi\theta}^p$ gives, by taking the expectation
    \begin{align*}
        e^{p\lambda t}\E\seg*{\abs*{X^\loi\theta_t}^p }&\leq \norm*{\theta}_p^p + \int_0^t C(1+ \|\theta\|_p^p)e^{\lambda s} \d s,
        \end{align*}
        which gives the wanted result.}
    }
    We prove \eqref{SDE: Estim unif T decoupled} by using same computations and \eqref{SDE: Estim unif T coupled}. 
    \end{proof}
    Now we give the first exponential convergence result for the solution to \eqref{SDE: MKV SDE}.
    \begin{thm}\label{SDE: H1: thm: Wp exponential contractivity}
        Assume that (\nameref{SDE: H1}) is fulfilled. Then, for all $p\geq 2$, $\theta,\theta'\in\L{2}(\Omega;\R^d)$, $t\geq 0$ and for $\Lambda:=\nu-(K^\sigma_x+K^\sigma_\LL)>0$,
        \begin{align}
            \W(\loi{X^\loi\theta_t},\loi{X^\loi{\theta'}_t})&\leq \W(\loi\theta,\loi{\theta'})e^{-\Lambda t}.
        \end{align}
    \end{thm}
    \begin{proof}Let $\theta,\,\theta'\in\L{2}(\Omega;\R^d)$.
        By applying Itô's formula to $e^{2\Lambda t}\abs*{\~X_t}^2$, where $\~X=X^\loi\theta-X^\loi{\theta'}$, we obtain due to (\nameref{SDE: H1})-3,
        \begin{align*}
            \d e^{2\Lambda t}\abs*{\~X_t}^2 &\leq 2(\Lambda+K^\sigma_x) e^{2\Lambda t}\abs*{\~X_t}^2\d t + 2\braket*{X^{\loi\theta}_t-X^\loi{\theta'}_t}{b(t,X^{\loi\theta}_t,\loi{X^{\loi\theta}_t})
            -b(t,X^{\loi{\theta'}}_t,\loi{X^{\loi{\theta'}}_t})}\d t\\&\quad+2K^\sigma_\LL\W(\loi{X^\loi\theta_t},\loi{X^\loi{\theta'}_t})^2e^{2\Lambda t}\d t+2e^{2\Lambda t}\braket{\~X_t}{(\sigma(X^{\loi\theta}_t,\loi{X^{\loi\theta}_t})-\sigma(X^{\loi{\theta'}}_t,\loi{X^{\loi{\theta'}}_t}))\d W_t}.
        \end{align*}
        Thanks to Proposition \ref{SDE: H1: prop: Estim unif T}, the stochastic integral in the previous estimate is a martingale. Thus, taking the expectation gives by (\nameref{SDE: H1})-1 and the definition of $\Lambda$,
        \begin{align*}
            e^{2\Lambda t}\E\seg*{\abs*{\~X_t}^2} &\leq \E\seg*{\abs*{\bar \theta-\bar \theta'}^2} -2(\nu-(\Lambda+K^\sigma_x)) \int^t_0 \E\seg*{\abs*{\~X_s}^2} e^{2\Lambda s}\d s + 2K^\sigma_\LL\int^t_0 \W(\loi{X^{\loi\theta}_s},\loi{X^\loi{\theta'}_s})^2e^{2\Lambda s}\d s\\
            &\leq \E\seg*{\abs*{\bar \theta-\bar \theta'}^2},
        \end{align*}
        where $\bar \theta$, $\bar\theta'$ are $\mathcal{G}$-measurable r.v. such that $\loi\theta=\loi{\bar \theta}$ and $\loi{\theta'}=\loi{\bar \theta'}$. Then we obtain the wanted result by taking an optimal coupling of $(\loi\theta,\loi{\theta'})$ for the $\W$-Wasserstein distance.
    \end{proof}
    \begin{rmq}
        Thanks to Remark \ref{rem: assumpt HSDE1}, this result still holds if we assume that $\eta > K^b_\LL$ and we drop the L-dissipative assumption. This case is already known in the literature, see for instance \cite[Proposition 3.1]{Ruda-Furhman} by replacing, in the cited reference, $\gamma$ by $\eta-K^\sigma_x$ and by considering our slightly different assumption on the Lipschitz constants of $\sigma.$
    \end{rmq}
    
    \begin{corollaire}
    \label{SDE: H1: cor: Wp exponential contractivity}
        { Assume that (\nameref{SDE: H1}) is fulfilled.
        Let $(\beta_t)_{t\geq 0}$ be a bounded progressively measurable process such that, by Girsanov's theorem,  $W^{\~\Q}:=W-\int^\cdot_0\beta_s \d s $ is a ${\~\Q}-$Brownian motion for ${\~\Q}$ the associated Girsanov's probability. Let us assume that  {$\nu >K^\sigma_x+ \sqrt{2K^\sigma_x}\abs*{\beta}_\infty$} and set $\varepsilon,\gamma$ such that $0<\varepsilon<\nu-K^\sigma_x - \sqrt{2K^\sigma_x}\abs*{\beta}_\infty$, $0<\gamma<\nu-K^\sigma_x - \sqrt{2K^\sigma_x}\abs*{\beta}_\infty-\varepsilon$ and $\gamma < \nu - (K^\sigma_x+K^\sigma_\LL)$.
        Then, there exists $C>0$ such that, for all $x,x'\in\R^d$, $\theta,\theta'\in\L{2}(\Omega;\R^d)$ and $t\geq 0$,
        \begin{equation}
            \E^{\~\Q}\seg*{\abs*{X^{x,\loi\theta}_t-X^{x',\loi{\theta'}}_t}^2}^{1/2}\leq C\left(\abs*{x-x'}+\W(\loi{\theta},\loi{\theta'})\right)e^{-\gamma t}.
        \end{equation}}
    \end{corollaire}
    \begin{proof}
       We write the equation satisfied by $X^{x,\loi\theta}$ under $\~\Q$ as 
       $$
       \d X^{x,\loi\theta}_t =b(t,X^{x,\loi\theta}_t,\loi{X^\theta_t})\d t + \sigma(X^{x,\loi\theta}_t,\loi{X^\loi\theta_t})\beta_t \d t+ \sigma(X^{x,\loi\theta}_t,\loi{X^\loi\theta_t})\d W^{\~\Q}_t.
       $$
       Then, by applying Itô's formula to $e^{2\gamma  t}\abs*{\~X_t}^2$, with $\~X=X^{x,\loi\theta}-X^{x',\loi{\theta'}}$, we have by using (\nameref{SDE: H1})-2, Young's inequality, the fact that $\nu\leq \eta$ and \autoref{SDE: H1: thm: Wp exponential contractivity},    
       \begin{align*}
            \d\left( e^{2\gamma t} \abs*{\~X_t}^2 \right) &\leq 2\braket*{X^{x,\loi\theta}_t-X^{x',\loi{\theta'}}_t}{b(t,X^{x,\loi\theta}_t,\loi{X^{\loi\theta}_t})-b(t,X^{x',\loi{\theta'}}_t,\loi{X^{\loi{\theta'}}_t})}e^{2\gamma t}\d t+ 2(K^\sigma_x+\sqrt{2K^\sigma_x}\abs*{\beta}_\infty+\gamma)e^{2\gamma t}\abs*{\~X_t}^2\d t \\&\quad
           + 2\abs*{\beta}_\infty\sqrt{2K^\sigma_\LL}e^{2\gamma t}\abs*{\~X_t}\W(\loi{X^\loi\theta_t},\loi{X^\loi{\theta'}_t}) \d t + 2K^\sigma_\LL e^{2\gamma t} \W(\loi{X^\loi\theta_t},\loi{X^\loi{\theta'}_t})^2\d t \\&\quad+2e^{2\gamma t}\braket{\~X_t}{(\sigma(X^{x,\loi\theta}_t,\loi{X^{x,\loi\theta}_t})-\sigma(X^{x',\loi{\theta'}}_t,\loi{X^{x',\loi{\theta'}}_t}))\d W^{\~\Q}_t}\\
           &\leq  -2\left(\eta - \left(K^\sigma_x+\sqrt{2K^\sigma_x}\abs*{\beta}_\infty+\gamma\right)\right)e^{2\gamma t}\abs*{\~X_t}^2\d t + 2(\abs*{\beta}_\infty\sqrt{2K^\sigma_\LL}+K^b_\LL)e^{2\gamma t}\abs*{\~X_t}\W(\loi{X^\loi\theta_t},\loi{X^\loi{\theta'}_t}) \d t \\&\quad+ 2K^\sigma_\LL e^{2\gamma t} \W(\loi{X^\loi\theta_t},\loi{X^\loi{\theta'}_t})^2\d t +2e^{2\gamma t}\braket{\~X_t}{(\sigma(X^{x,\loi\theta}_t,\loi{X^{x,\loi\theta}_t})-\sigma(X^{x',\loi{\theta'}}_t,\loi{X^{x',\loi{\theta'}}_t}))\d W^{\~\Q}_t}\\
           &\leq -2\left(\eta -\left(K^\sigma_x+\sqrt{2K^\sigma_x}\abs*{\beta}_\infty+\gamma+\varepsilon\right)\right)e^{2\gamma t}\abs*{\~X_t}^2\d t + \left((\abs*{\beta}_\infty\sqrt{2K^\sigma_\LL}+ K^b_\LL)^2(2\varepsilon)^{-1} + 2K^\sigma_\LL\right)\W(\loi{X^\loi\theta_t},\loi{X^\loi{\theta'}_t})^2e^{2\gamma t}\d t   \\&\quad+2e^{2\gamma t}\braket{\~X_t}{(\sigma(X^{x,\loi\theta}_t,\loi{X^{x,\loi\theta}_t})-\sigma(X^{x',\loi{\theta'}}_t,\loi{X^{x',\loi{\theta'}}_t}))\d W^{\~\Q}_t}\\
           &\leq -2\left(\nu -\left(K^\sigma_x +\sqrt{2K^\sigma_x}\abs*{\beta}_\infty+\gamma+\varepsilon \right)\right)e^{2\gamma t}\abs*{\~X_t}^2\d t + C \W(\loi{X^\loi\theta_t},\loi{X^\loi{\theta'}_t})^2e^{2\gamma t}\d t
           \\&\quad+2e^{2\gamma t}\braket{\~X_t}{(\sigma(X^{x,\loi\theta}_t,\loi{X^{x,\loi\theta}_t})-\sigma(X^{x',\loi{\theta'}}_t,\loi{X^{x',\loi{\theta'}}_t}))\d W^{\~\Q}_t}\\
           &\leq C\W(\loi{\theta},\loi{\theta'})^2 e^{2(\gamma-\Lambda) t}\d t + 2e^{2\gamma t}\braket{\~X_t}{(\sigma(X^{x,\loi\theta}_t,\loi{X^{x,\loi\theta}_t})-\sigma(X^{x',\loi{\theta'}}_t,\loi{X^{x',\loi{\theta'}}_t}))\d W^{\~\Q}_t}.
        \end{align*}
        Using \autoref{SDE: H1: prop: Estim unif T Q}, the stochastic integral is a martingale. Hence, by taking the expectation w.r.t ${\~\Q}$ and recalling that $\gamma < \Lambda$, we have
        \begin{align*}
            e^{2\gamma t}\E^{\~\Q}\seg*{\abs*{\~X_t}^2}&\leq \abs*{x-x'}^2 + C\W(\loi\theta,\loi\theta')^2,
        \end{align*}
        which gives the desired result.    \end{proof}
        The next result extend \autoref{SDE: H1: prop: Estim unif T} under the new probability $\~\Q$.

    \begin{prop}\label{SDE: H1: prop: Estim unif T Q}
       { Assume that (\nameref{SDE: H1}) is fulfilled. Let $\beta$ and $\~\Q$ be defined as in \autoref{SDE: H1: cor: Wp exponential contractivity} and let us assume that  $\nu >K^\sigma_x + \sqrt{2K^\sigma_x}\abs*{\beta}_\infty$.} Then, for all $p\geq 2$, there exists $C=C(p,\nu,K^\sigma_x,K^\sigma_\LL,K^b_\LL,\sigma(0,\delta_0),{ \|\sigma\|_{\infty} \1{p>2}},\sup{t}\abs*{b(t,0,\delta_0)})$  such that for all $x\in\R^d,\,\theta\in{  \L{p}}(\Omega;\R^d)$,
     \begin{align}\label{SDE: H1: Estim unif T coupled Q}
        \sup{t\in\R_+}\E^{\~\Q}\seg*{\abs*{X^\loi\theta_t}^p} \leq& C(1+\norm*{\theta}_{{ p}}^p),\\
        \sup{t\in\R_+}\E^{\~\Q}\seg*{\abs*{X^{x,\loi\theta}_t}^p}\leq& C(1+\abs*{x}^p+\norm*{\theta}_{{ p}}^p). \label{SDE: H1: Estim unif T decoupled Q}
    \end{align}
    \end{prop}
    \begin{proof}
        This is a straight forward adaptation of the proof of \autoref{SDE: H1: prop: Estim unif T} where we replace \eqref{SDE: H1: proof: exp + cst} by the following inequality, for $\varepsilon$ small enough and { $0<\gamma<\nu-(1+\varepsilon)K^\sigma_x-\sqrt{2K^\sigma_x}\abs*{\beta}_\infty -\varepsilon$},
        {
        \begin{align*}
         \d e^{2\gamma t}\abs*{X^\loi\theta_t}^2 &\leq  \left(  2(\gamma+\varepsilon+\sqrt{2K^\sigma_x}\abs*{\beta}_\infty )\abs*{X^\loi\theta_t}^2+ 2\braket*{X^{\loi\theta}_t}{b(t,X^{\loi\theta}_t,\loi{X^{\loi\theta}_t})-b(t,0,\delta_0)}\right)e^{2\gamma t}\d t\\&\quad +(1+\varepsilon)\norm*{\sigma(X^{\loi\theta}_t,\loi{X^\loi\theta_t})-\sigma(0,\delta_0)}^2 e^{2\gamma t}\d t +(\varepsilon^{-1}\sup{t}\abs*{b(t,0,\delta_0)}^2+(1+\varepsilon^{-1})\norm*{\sigma(0,\delta_0)}^2)e^{2\gamma t}\d t \\&\quad+2e^{2\gamma t}\braket*{X^\loi\theta_t}{\sigma(X^\loi\theta_t,\loi{X^\loi\theta_t})\d W^{\~\Q}_t}.\\
          &\leq  \left(  2(\gamma+\varepsilon+\sqrt{2K^\sigma_x}\abs*{\beta}_\infty +(1+\varepsilon)K^\sigma_x)\abs*{X^\loi\theta_t}^2+ 2\braket*{X^{\loi\theta}_t}{b(t,X^{\loi\theta}_t,\loi{X^{\loi\theta}_t})-b(t,0,\delta_0)}\right)e^{2\gamma t}\d t\\&\quad +\left(2K_\LL^\sigma (1+\varepsilon) C(1+\|\theta\|_2^2)+\varepsilon^{-1}\sup{t}\abs*{b(t,0,\delta_0)}^2+(1+\varepsilon^{-1})\norm*{\sigma(0,\delta_0)}^2\right)e^{2\gamma t}\d t \\&\quad+2e^{2\gamma t}\braket*{X^\loi\theta_t}{\sigma(X^\loi\theta_t,\loi{X^\loi\theta_t})\d W^{\~\Q}_t},
    \end{align*}
    applying \eqref{SDE: Estim unif T coupled} in the last inequality.} From here, the remaining of the proof is the same as for \autoref{SDE: H1: prop: Estim unif T}.
    \end{proof}

\subsection{The weak dissipative framework}


In this subsection, we consider the equation \eqref{SDE: MKV SDE} in a simpler framework where $\sigma$ does not depend on the distribution, namely:
\begin{equation}\label{SDE: H2-2: Distribution free sigma}
    \d X_t =b(t,X_t,\loi{X_t})\d t +\sigma(X_t)\d W_t, \quad t\geq 0.   
    \end{equation}
It should be possible to extend some results of this subsection when $\sigma$ also depend on the distribution but at the cost of additional strong constraints on the Lipschitz and monotonicity constants of $b$ and $\sigma$, see \autoref{SDE: H2-2: rmq}.

The aim of this subsection is to relax the strong dissipativity assumption on $b$ into a weak one: Generally speaking, $b$ is now allowed to be dissipative only outside a ball. 
In order to obtain the same results as \autoref{SDE: H1: thm: Wp exponential contractivity} and \autoref{SDE: H1: cor: Wp exponential contractivity} in this new framework, we use a classical approach based on a reflection coupling as in \cite{conforti} and \cite{Huang-Ma}, themselves based on the seminal paper \cite{Eberle}.

\begin{hyp*}[$\H_{SDE}2$]\label{SDE: H2-2}
$\vspace{0pt}$
    \begin{enumerate}
        \item There exists $R,\,K^b_\LL,\, K^b_x,\, \eta>0$ such that for all $x,x'\in\R^d,\, \mu,\mu'\in\Pcal_{ 1}(\R^d)$ and $t\geq 0$
        \begin{align*}
        \abs{b(t,x,\mu)-b(t,x',\mu')}&\leq K^b_x\abs*{x-x'} +K^b_\LL \w(\mu,\mu'), \\
        \braket*{x-x'}{b(t,x,\mu)-b(t,x',\mu)}&\leq -\eta\abs*{x-x'}^2\1{\abs*{x-x'}{ > } R}+K^b_x\abs*{x-x'}^2\1{\abs*{x-x'}{ \leq }R}, 
        \end{align*}
        \item $\sigma$ is bounded and uniformly elliptic, i.e. there exists $\sigma_0>0$ such that for all $x\in\R^d$: $\sigma(x)\sigma(x)^\top \geq \sigma_0^2I_d$, 
        \item there exists $K^\sigma_x>0$ such that for all $x,x'\in\R^d$: $\norm*{\sigma(x)-\sigma(x')}\leq \sqrt{2K^\sigma_x}\abs*{x-x'}$,
        \item $ K^b_\LL< (\eta-K^\sigma_x)e^{-\frac{\eta+2K^b_x}{2\sigma^2_0}R^2}.$
    \end{enumerate}
\end{hyp*}

We start by giving some uniform estimates on the moments of solutions of \eqref{SDE: MKV SDE} and \eqref{SDE: Decoupled MKV SDE}.
\begin{prop}\label{SDE: H2: prop: Estim unif T}
Assume that (\nameref{SDE: H2-2}) holds.
    For all $p\geq 2$, there exists  $C=C(p,\eta,K^\sigma_x,K^b_\LL,R,\sigma(0), { \|\sigma\|_{\infty} \1{p>2}}, \sup{t}\abs*{b(t,0,\delta_0)})$  such that for all $x\in\R^d,\,\theta\in{  \L{p}}(\Omega;\R^d)$, 
     \begin{align}\label{SDE: H2: Estim unif T coupled}
        \sup{t\in\R_+}\E\seg*{\abs*{X^\loi\theta_t}^p} \leq& C(1+\norm*{\theta}^p_{ p}),\\
        \sup{t\in\R_+}\E\seg*{\abs*{X^{x,\loi\theta}_t}^p}\leq& C(1+\abs*{x}^p+\norm*{\theta}^p_{ p}). \label{SDE: H2: Estim unif T decoupled}
    \end{align}
\end{prop}
\begin{proof}
    The proof follows same computations as for \autoref{SDE: H1: prop: Estim unif T} using the fact that $K^b_\LL<\eta-K^\sigma_x$.
\end{proof}
We can now obtain the first exponential convergence result for \eqref{SDE: MKV SDE}.
{
\begin{thm}\label{SDE: H2-2: thm: Wp exponential contractivity}
    Assume that (\nameref{SDE: H2-2}) holds and let $p\geq 2$. Then, there exists $C= C(\eta ,\sigma_0,K^b_x,K^b_\LL,K^\sigma_x,R)>0$ and $\~{\eta}=\~\eta(\eta ,\sigma_0,K^b_x,K^b_\LL,K^\sigma_x,R)>0$  such that, for all $\theta,\theta'\in{  \L{2p-1}}(\Omega;\R^d)$ and $t \geq 0$,
\begin{align}\label{SDE: H2-2: Wp exponential estimate}
\WW_1(\loi{X^{\loi\theta}_t},\loi{X^{\loi{\theta'}}_t})&\leq C \w(\loi\theta,\loi{\theta'})e^{-\~\eta t},\\
\WW_p(\loi{X^{\loi\theta}_t},\loi{X^{\loi{\theta'}}_t})&\leq C(1+\norm*{\theta}_{2p-1}^{1-\frac{1}{2p}}+\norm*{\theta'}_{2p-1}^{1-\frac{1}{2p}}) \w(\loi\theta,\loi{\theta'})^{1/p}e^{-\~\eta t/p}.\label{SDE: H2-2: Wp exponential estimate moment theta}
\end{align}
\end{thm} 
\begin{proof}
 Firstly, we prove \eqref{SDE: H2-2: Wp exponential estimate} by using \autoref{Appendix: thm general}. By setting $\b(t,\cdot)=b(t,\cdot,\loi{X^\loi\theta})$, $\b'(t,\cdot)=b(t,\cdot,\loi{X^\loi{\theta'}})$ and $(\bm X_0,\bm X'_0)=(\theta,\theta')$, we have that ($\bm X,\bm X'$) is a coupling of $(X^{\loi\theta},X^\loi{\theta})$. Moreover, we have $M_\b$=$K^b_xR$, $\EE_t\leq K^b_\LL\E\seg*{\abs*{\X_t-\X'_t}}$ and we can take $\bm{c}=K^b_\LL$ since \eqref{Appendix: constraint} is satisfied due to (\nameref{SDE: H2-2})-4. Hence, by \autoref{Appendix: thm general}, \eqref{SDE: H2-2: Wp exponential estimate} holds.
 In order to obtain the results for $p>1$, we use the following suitable Cauchy-Schwarz's inequality. Let $(X,X')$ be an optimal coupling of $(X^\loi\theta,X^\loi{\theta'})$ for the $\w$-distance. By using, \autoref{SDE: H2: prop: Estim unif T} with $2p-1$, we obtain
\begin{align*}
    \WW_p(\loi{X^\loi\theta_t},\loi{X^\loi{\theta'}_t})=\WW_p(\loi{X_t},\loi{X'_t})&\leq \E\seg*{\abs*{X_t-X'_t}^p}^{1/p}= \E\seg*{\abs*{X_t-X'_t}^{p-1/2}\abs*{X_t-X'_t}^{1/2}}^{1/p}\\ 
    &\leq \E\seg*{\abs*{X_t-X'_t}^{2p-1}}^{1/2p}\E\seg*{\abs*{X_t-X'_t}}^{1/p}\\
    &\leq C\left(\E\seg*{\abs*{X_t}^{2p-1}}^{1/2p}+\E\seg*{\abs*{X'_t}^{2p-1}}^{1/2p}\right)\E\seg*{\abs*{X_t-X'_t}}^{1/p}\\
    &\leq C(1+\norm*{\theta}_{2p-1}^{1-\frac{1}{2p}}+\norm*{\theta'}_{2p-1}^{1-\frac{1}{2p}})\w(\loi{X^\loi\theta_t},\loi{X^\loi{\theta'}_t})^{1/p}.
\end{align*}
The proof concludes by using \eqref{SDE: H2-2: Wp exponential estimate}.
 \end{proof}
}
The next theorem concerns some exponential convergence results for \eqref{SDE: Decoupled MKV SDE}.

{\begin{thm}\label{SDE: H2-2: thm: Wp exponential estimate decoupled}
    Assume that (\nameref{SDE: H2-2}) holds and let $p \geq 2$.
    Then, there exists $C= C(\eta ,\sigma_0,K^\sigma_x,K^b_\LL,K^b_x,R)>0$ and $\hat{\eta}=\hat\eta(\eta ,\sigma_0,K^b_x,K^b_\LL,K^\sigma_x,R)>0$  such that, for all $x,x'\in\R^d$, $\theta,\theta'\in{  \L{2p-1}}(\Omega;\R^d)$ and $t \geq 0$,
\begin{align}\label{SDE: H2-2: Wp exponential estimate decoupled}
\w(\loi{X^{x,\loi{\theta}}_t},\loi{X^{x',\loi{\theta'}}_t})&\leq C \left(\abs*{x-x'}+\w(\loi\theta,\loi{\theta'})\right)e^{-\hat\eta t}\\
\WW_p(\loi{X^{x,\loi\theta}_t},\loi{X^{x',\loi{\theta'}}_t})&\leq C(1+\abs*{x}^{1-\frac{1}{2p}}+\abs*{x'}^{1-\frac{1}{2p}}+\norm*{\theta}_{2p-1}^{1-\frac{1}{2p}}+\norm*{\theta'}_{2p-1}^{1-\frac{1}{2p}})\left( \abs*{x-x'}^{1/p}+\w(\loi\theta,\loi{\theta'})^{1/p}\right)e^{-\hat\eta t/p},\label{SDE: H2-2: Wp exponential estimate decoupled moment theta}
\end{align}
\end{thm}
\begin{proof}
    First, due to triangular inequality, one has
    $$\w(\loi{X^{x,\loi{\theta}}_t},\loi{X^{x',\loi{\theta'}}_t})\leq\w(\loi{X^{x,\loi{\theta}}_t},\loi{X^{x',\loi{\theta}}_t})+\w(\loi{X^{x',\loi{\theta}}_t},\loi{X^{x',\loi{\theta'}}_t}).
    $$ 
    For the first term on the right-hand side, we apply \autoref{Appendix: thm general} by setting  $\b(t,\cdot)=\b'(t,\cdot)=b(t,\cdot,\loi{X^\loi\theta_t})$ and $(\bm X_0,\bm X'_0)=(x,x')$ which ensures that $(\bm X,\bm X')$ is a coupling of $(X^{x,\loi\theta},X^{x',\loi\theta})$. Moreover we have $M_\b$=$K^b_xR$ and $\EE_t=\bm{c}=0$ for all $t\geq 0$.
    Then,
    $$\w(\loi{X^{x,\loi{\theta}}_t},\loi{X^{x',\loi{\theta}}_t}) \leq C |x-x'|e^{-\bm{\hat\eta} t}.$$
    For the second term, we set $\b(t,\cdot)=b(t,\cdot,\loi{X^\loi\theta})$, $\b'(t,\cdot)=b(t,\cdot,\loi{X^\loi{\theta'}_t})$ with $M_\b=K^b_xR$ and $\bm X_0=\bm X'_0=x'$ which ensures that ($\bm X,\bm X')$ is a coupling of $(X^{x',\loi{\theta}},X^{x',\loi{\theta'}})$. Then, \autoref{Appendix: thm general} gives us
    $$\w(\loi{X^{x',\loi{\theta}}_t},\loi{X^{x',\loi{\theta'}}_t}) =\w(\loi{\bm X_t},\loi{\bm X'_t})\leq C\w(\loi{\bm X_0},\loi{\bm X'_0})e^{-\bm{\hat\eta} t} +Ce^{-\bm{\hat\eta} t}\int^t_0 \EE_s
    e^{\bm{\hat\eta} s}\d s = C e^{-\bm{\hat\eta} t}\int^t_0 \EE_s
    e^{\bm{\hat\eta} s}\d s.
    $$
    Since $\EE_t\leq C\w(\loi{X^\loi{\theta}_t},\loi{X^\loi{\theta'}_t})$, we obtain by using \eqref{SDE: H2-2: Wp exponential estimate} that $$\w(\loi{X^{x',\loi{\theta}}_t},\loi{X^{x',\loi{\theta'}}_t})\leq  C e^{-\bm{\hat\eta} t} \int^t_0 
    e^{(\bm{\hat\eta} -\hat\eta) s}\d s.
    $$
    Moreover, by using the following trivial fact: for all $c_1,c_2>0$, for all $t\geq 0$,
    \begin{equation}\label{Comparaison exponentielle}
     e^{-c_1 t}\int^t_0 e^{(c_1-c_2)s} \d s \leq C_{c_1,c_2}e^{-\frac{\min{}(c_1,c_2)}{2} t},
    \end{equation}
    we can conclude that there exists $\hat\eta<\min{}(\bm{\hat\eta},\~\eta)$ such that
    $$
    \w(\loi{X^{x',\loi{\theta}}_t},\loi{X^{x',\loi{\theta'}}_t}) \leq   C\w(\loi\theta,\loi{\theta'})e^{-\hat\eta t}.$$
    Finally, the proof of \eqref{SDE: H2-2: Wp exponential estimate decoupled moment theta} follows with the same Cauchy-Schwarz's trick used in \autoref{SDE: H2-2: thm: Wp exponential contractivity}.
\end{proof}
}
{
\begin{prop}
    \label{SDE: H2-2: prop: estim moments}
    Assume that (\nameref{SDE: H2-2}) is fulfilled.
        Let $(\~\beta_t)_{t\geq 0}$ be a bounded progressively measurable process such that, by Girsanov's theorem,  $W^{\~\Q}:=W-\int^\cdot_0\~\beta_s \d s $ is a ${\~\Q}-$Brownian motion for ${\~\Q}$ the associated Girsanov's probability. for all $p\geq 2$,
there exists { $C= C(p,\eta ,K^\sigma_x,K^b_x,K^b_\LL,R,\abs*{\~\beta}_\infty,\norm*{\sigma}_\infty)>0$} such that, for all $x \in \R^d$, $\theta \in  \L{p}(\Omega;\R^d)$,
\begin{align}\label{SDE: H2: Estim unif T coupled Q}
        \sup{t\in\R_+}\E^{\~\Q}\seg*{\abs*{X^\loi\theta_t}^p} \leq& C(1+\norm*{\theta}^p_{ p}),\\
        \sup{t\in\R_+}\E^{\~\Q}\seg*{\abs*{X^{x,\loi\theta}_t}^p}\leq& C(1+\abs*{x}^p+\norm*{\theta}^p_{ p}). \label{SDE: H2: Estim unif T decoupled Q}
    \end{align}
\end{prop}
\begin{proof}
   The proof is the same as the one did for \autoref{SDE: H1: prop: Estim unif T Q}. 
\end{proof}

Now, let us consider a bounded function $\beta : \R_+ \times\R^d \times \Pcal_{2q+2} \to \R^d$ such that $\beta(t,.,.)$ is Lipschitz uniformly with respect to $t$ and let us denote $\Pcal^\beta$ the semi group associated to the family of SDEs, for all $x \in \R^d$, $\theta \in  \L{ 2q+2}(\Omega;\R^d)$,
\begin{equation}\label{SDE: H2: SDE Q}
    \d X^{\beta,x,\loi \theta}_t = b(t,X^{\beta,x,\loi \theta}_t,\loi{X^{\loi \theta}_t}) +\sigma(X^{\beta,x,\loi \theta}_t)\beta(t,X^{\beta,x,\loi \theta}_t,\loi{X^{\loi\theta}_t})\d t +\sigma(X^{\beta,x,\loi \theta}_t)\d W_t, \quad t \geq 0, \quad X^{\beta,x,\loi \theta}_0=x.
\end{equation}
Let us emphasize that \eqref{SDE: H2: SDE Q} is not a decoupled McKean-Vlasov's SDE since the law dependence in $b$ and $\beta$ corresponds to the original McKean-Vlasov's SDE, i.e. the SDE where $\beta=0$.
\begin{thm}\label{SDE: H2-2: thm: coupling estimates decoupled}
 Assume that (\nameref{SDE: H2-2}) is fulfilled and let us consider a bounded function $\beta : \R_+ \times\R^d \times \Pcal_{2q+2} \to \R^d$, for $q\geq 0$, such that $\beta(t,.,.)$ is Lipschitz uniformly with respect to $t$.
Let $\epsilon\in(0,1]$ and let us consider a measurable function $\psi : \R^d \times \Pcal_{ 2q+2 } \rightarrow  \R$ that satisfies for all $x,x'\in\R^d,\mu,\mu' \in\Pcal_{ 2q+2 }(\R^d)$ \begin{equation}\label{SDE: H2-2: eq: def psi}
        \abs*{\psi(x,\mu)-\psi(x',\mu')}\leq c(1+\abs*{x}^q+\abs*{x'}^q+\mathcal{W}_{2q+2}(0,\mu)^q+\mathcal{W}_{2q+2}(0,\mu')^q)\left(\abs*{x-x'}^\epsilon + \w(\mu,\mu')^\epsilon\right).
    \end{equation}
Then, there exists $C= C(q,\eta ,\sigma_0,K^\sigma_x,K^b_x,K^b_\LL,R,\abs*{\beta}_\infty,\norm*{\sigma}_\infty)>0$ and $\hat{\eta}=\hat\eta(\eta ,\sigma_0,K^b_x,K^b_\LL,K^\sigma_x,R,\abs*{\beta}_\infty,\norm*{\sigma}_\infty)>0$  such that, for all $x,x'\in\R^d$, $\theta \in  \L{ 2q+2}(\Omega;\R^d)$ and $t \geq 0$,
\begin{equation}\label{SDE: H2-2: Coupling estimate decoupled Q}
    \abs*{\Pcal^\beta_t\seg*{\psi}(x,\loi\theta)-\Pcal^\beta_t\seg*{\psi}(x',\loi{\theta})}\leq C(1+\abs*{x}^{q+1/2} + \abs*{x'}^{q+1/2}+\norm*{\theta}^{{ q+1/2}}_{ 2q+2})e^{-\hat\eta\epsilon t/2}\abs*{x-x'}^\frac{\epsilon}{2}.
\end{equation}
\end{thm} 

\begin{proof}
    Let us prove \eqref{SDE: H2-2: Coupling estimate decoupled Q}. {  We set $(\UU_t,\UU'_t)$ an optimal coupling of $(X^{\beta,x,\loi\theta}_t,X^{\beta,x',\loi\theta}_t)$ for the $\w-$distance.} By using the locally Hölder property of $\psi$ and Young's inequality, we have
\begin{align}
    \abs*{\Pcal^\beta_t\seg*{\psi}(x,\loi\theta)-\Pcal^\beta_t\seg*{\psi}(x',\loi\theta)}&\leq \abs*{\E\seg*{\psi(\UU_t,\loi{X^{\loi\theta}_t})-\psi(\UU'_t,\loi{X^\loi\theta_t})}}\notag\\
            &\leq C\E\seg*{\left(1+\abs*{\UU_t}^q+\abs*{\UU'_t}^q+\norm*{X^\loi\theta_t}^q_{ 2q+2} \right)\abs*{\UU_t-\UU'_t}^\epsilon }\notag\\
            &\leq C\E\seg*{\left(1+\abs*{\UU_t}^{q+\epsilon/2}+\abs*{\UU'_t}^{q+\epsilon/2}+\norm*{X^\loi\theta_t}^{q+\epsilon/2}_{ 2q+2}\right)\abs*{\UU_t-\UU'_t}^{\epsilon/2} }\label{SDE: temp globally lip}\\
            &\leq  C\left(1+\E\seg*{\abs*{\UU_t}^{2q+\epsilon}}^{1/2}+\E\seg*{\abs*{\UU'_t}^{2q+\epsilon}}^{1/2}+\norm*{X^{\loi\theta}_t}^{q+\epsilon/2}_{ 2q+2}\right)\E\seg*{\abs*{\UU_t-\UU_t}}^{\epsilon/2}\notag \\
            &= {   C\left(1+\E\seg*{\abs*{X^{\beta,x,\loi\theta}_t}^{ 2q+\epsilon}}^{1/2}+\E\seg*{\abs*{X^{\beta,x',\loi\theta}_t}^{ 2q+\epsilon}}^{1/2}+\norm*{X^{\loi\theta}_t}^{{ q+\epsilon/2}}_{ 2q+2}\right)\w(X^{\beta,x,\loi\theta}_t,X^{\beta,x',\loi\theta}_t)^{\epsilon/2}.\notag }
\end{align}
{ Since $\beta$ is bounded, we can set $\~\beta_t := \beta(t,X^{x,\loi \theta}_t,\loi{X^{\loi\theta}_t})\d t$ and remark that 
$$\E\seg*{\abs*{X^{\beta,x,\loi\theta}_t}^{ 2q+\epsilon}}^{1/2} = \E^{\tilde\Q}\seg*{\abs*{X^{x,\loi\theta}_t}^{ 2q+\epsilon}}^{1/2},$$
where $\tilde{\Q}$ is defined in \autoref{SDE: H2-2: prop: estim moments}.
Then, by using \eqref{SDE: H2: Estim unif T coupled}, with $p=2q+\epsilon \geq 2$ if $q\geq 1$ and with $p=2$ and Jensen's inequality otherwise, \eqref{SDE: H2: Estim unif T decoupled Q}, Jensen's and Young's inequalities we obtain
\begin{align}
    \abs*{\Pcal^\beta_t\seg*{\psi}(x,\loi\theta)-\Pcal^\beta_t\seg*{\psi}(x',\loi\theta)}&\leq C(1+\abs*{x}^{q+1/2} + \abs*{x'}^{q+1/2}+\norm*{\theta}^{{ q+1/2}}_{2q+2}) \w(X^{\beta,x,\loi\theta}_t,X^{\beta,x',\loi\theta}_t)^{\epsilon/2}.\label{SDE: H2-2: preuve coupling: eq: temp}
\end{align}
}
Finally, we deal with the term $\w(X^{\beta,x,\loi\theta}_t,X^{\beta,x',\loi\theta}_t)^{\epsilon/2}$ by using \autoref{Appendix: thm general}. We set $\b(t,\cdot)=\b'(t,\cdot):=b(t,\cdot,\loi{X^\loi\theta_t})+\sigma(\cdot)\beta(t,\cdot,\loi{X^\loi\theta_t})$ and we check that it satisfies \eqref{Appendix: hyp: dissipatif}.
  By using (\nameref{SDE: H2-2}) we have, for all $R' \geq R$,
    \begin{align*}
        &\braket{x-x'}{b(t,x,\loi{X^{\loi\theta}_t})-b(t,x',\loi{X^\loi\theta_t})+\sigma(x)\beta(t,x,\loi{X^\loi\theta_t})-\sigma(x')\beta(t,x',\loi{X^\loi\theta_t})}\\
        \leq & -\eta\abs*{x-x'}^2\1{\abs*{x-x'}> R'}+K^b_x\abs*{x-x'}^2\1{\abs*{x-x'}\leq R'}+2\abs*{\beta}_\infty\norm*{\sigma}_\infty\abs*{x-x'}\\
        \leq & -(\eta-\frac{2\abs*{\beta}_\infty\norm*{\sigma}_\infty}{R'})\abs*{x-x'}^2\1{\abs*{x-x'}> R'} +  (K^b_xR'+2\abs*{\beta}_\infty\norm*{\sigma}_\infty)\abs*{x-x'}\1{\abs*{x-x'}\leq R'}.
    \end{align*}

    Thus, by taking $R'$ large enough, $\b$ satisfies \eqref{Appendix: hyp: dissipatif} with $\bm\eta=\eta-\frac{2\abs*{\beta}_\infty\norm*{\sigma}_\infty}{R'}>K^\sigma_x$ and $ M_\b=K^b_xR'+2\abs*{\beta}_\infty\norm*{\sigma}_\infty$. Hence, we can apply \autoref{Appendix: thm general} with $\EE_t=\bm{c}=0$, in order to obtain 
    {
    \begin{equation}\label{SDE: H2-2: W1 Q coupled}
        \w(\loi{X^{\beta,x,\loi\theta}_t}, \loi{X^{\beta,x',\loi{\theta}}_t})\leq C\abs*{x-x'}e^{-\~\eta t}.
    \end{equation}}
    It just remains to plug the previous estimate into \eqref{SDE: H2-2: preuve coupling: eq: temp} in order to get \eqref{SDE: H2-2: Coupling estimate decoupled Q}.
\end{proof}

\begin{rmq}\label{SDE: H2-2: rmq global lip}
    If $\psi$ is globally $\epsilon$-Hölder, i.e. $q=0$ in \eqref{SDE: H2-2: eq: def psi}, then \eqref{SDE: temp globally lip}
    can be replaced by the better estimate
    \begin{align*}
         \abs{\Pcal^\beta_t\seg*{\psi}(x,\loi\theta)-\Pcal^\beta_t\seg*{\psi}(x',\loi\theta)}&\leq C\E\seg*{\abs*{\UU_t-\UU'_t}^{\epsilon}}\leq C\E\seg*{\abs*{\UU_t-\UU'_t}}^{\epsilon}=C\w(\loi{X^{\beta,x,\loi\theta}_t},\loi{X^{\beta,x',\loi\theta}_t})^{\epsilon}.
    \end{align*}  
    So, \eqref{SDE: H2-2: Coupling estimate decoupled Q} can be improved in order to obtain
    \begin{align*}
    \abs*{\Pcal^\beta_t\seg*{\psi}(x',\loi{\theta})-\Pcal^\beta_t\seg*{\psi}(x',\loi{\theta})}&\leq C\abs*{x-x'}^\epsilon e^{-\hat\eta t}.
    \end{align*}
\end{rmq}
\begin{rmq}\label{SDE: H2-2: rmq}$\vspace{0pt}$
\begin{enumerate}
    \item If $b$ is strongly dissipative, i.e. $R \rightarrow 0$, then (\nameref{SDE: H2-2})-4 corresponds to the assumption $K^b_\LL <\eta-K^\sigma_x$ which is the usual constraint needed in the strong dissipative framework: see for instance \cite[Assumption \textbf{(E)}]{Huang-Wang} or \autoref{SDE: Strong dissip}.
    \item Under (\nameref{SDE: H2-2}), we no longer asked $b$ to be Lipschitz under the $\W$-Wasserstein distance, but under the $\w$ one, which is a stronger assumption than the one used in the strong dissipativity framework.
     \item It is possible to add a distribution dependency on $\sigma$ by using same computations as in the proof of \cite[Theorem 3.1]{Huang-Li-Mu}. However, these computations need a stronger complicated assumption where Lipschitz constants appear in a strong intricated way: see assumptions appearing in \cite[Theorem 2.1]{Huang-Li-Mu}. In order to ensure the readability of this article, we decided to consider a simpler framework, i.e. $\sigma$ that does not depend on the distribution. Nevertheless assumption (\nameref{SDE: H1}) in the previous subsection allows a distribution dependency of $\sigma$.
\end{enumerate}
\end{rmq}

\section{The Ergodic Distribution Dependent BSDE}\label{EBSDE}
The aim of this section is to prove, under the two different sets of assumptions (\nameref{SDE: H1}) and (\nameref{SDE: H2-2}), a result of existence and uniqueness of the following distribution dependent ergodic BSDE

\begin{equation}\label{EBSDE: EBSDE}\tag{EBSDE}
    Y^{\loi \theta}_t=Y^{\loi \theta}_T +\int^T_t \left(f(X^{\loi \theta}_s,\loi{X^{\loi \theta}_s},Z^{\loi \theta}_s)-\lambda\right)\d s-\int^T_0 Z^{\loi \theta}_s \d W_s, \quad\quad\quad \forall~ 0\leq t\leq T
\end{equation}
where the unknown is the triple $(Y^{\loi \theta},Z^{\loi \theta},\lambda)$ and with $X^{\loi \theta}$ the solution to \eqref{SDE: MKV SDE} with $s=0$ and $b,f$ satisfying following assumptions.
\begin{hyp*}[$\H_0$]\label{EBSD: H0} There exists $q\geq 0$ such that
\begin{enumerate}
\item $f : (x,\mu,z)\ni \R^d\times\Pcal_{ 2q+2}(\R^d)\times\R^d \mapsto \R$ is uniformly $K^f_z$-Lipschitz w.r.t $z$,
\item 
\begin{enumerate}
    \item there exists $C>0$ such that for all $x\in\R^d,\theta\in{ \L{2q+2} }(\Omega;\R^d)$, \begin{equation*}
       \abs*{f(x,\loi\theta,0)}\leq C(1+\abs*{x}^{q+1}+\norm*{\theta}_{{ 2q+1}}^{q+1}),
\end{equation*}
\item there exists $\epsilon\in(0,1]$ and $C>0$ such that for all $x,x'\in\R^d$, $\theta,\theta'\in{ \L{2q+2} }(\Omega; \R^d)$ and all $z\in\R^d$, \begin{equation}
        \label{hyp-f}
        \abs*{f(x,\loi\theta,z)-f(x',\loi{\theta'},z)}\leq C( 1+\abs*{x}^q+\abs*{x'}^q+\norm*{\theta}^q_{ 2q+2}+\norm*{\theta'}^q_{ 2q+2})\left( {\w}(\loi{\theta},\loi{\theta'})^\epsilon +\abs*{x-x'}^\epsilon \right).
    \end{equation}
\end{enumerate} 
\item for $b$ defined in \eqref{SDE: MKV SDE}, $b$ no longer depends on time.
\end{enumerate}
\end{hyp*}

{\begin{rmq}
    We could work in the larger space $\Pcal_{(2q+\epsilon)\vee 2}$ instead of $\Pcal_{2q+2}$ but, for readability reasons, we decided to stay in $\Pcal_{2q+2}$, underlying the fact that the restriction is very weak.
\end{rmq}
In order to prove the existence of a solution $(Y,Z,\lambda)$ to \eqref{EBSDE: EBSDE}, we will work with the associated decoupled EBSDE given by
     \begin{align}\tag{Decoupled EBSDE}\label{EBSDE: Decoupled EBSDE}
     Y^{x,\loi{\theta}}_t&=Y^{x,\loi{\theta}}_T +\int^T_t \left(f(X^{x,\loi\theta}_s,\loi{X^\loi\theta_s},Z^{x,\loi{\theta}}_s)-\lambda\right)\d s -\int^T_t Z^{x,\loi{\theta}}_s\d W_s,\quad \quad \forall 0\leq t \leq T,
     \end{align}
 and we will follow the classical proof that introduce a discounted approximation of \eqref{EBSDE: Decoupled EBSDE}, see for instance \cite{DEBUSSCHE2011407},\cite{Hu-Lemmonier},\cite{Hu-Madec-Richou}: Let $\alpha>0$, 
     \begin{align}\label{EBSDE: Decoupled alpha BSDE}\tag{$\alpha-$BSDE}
     Y^{\alpha,x,\loi{\theta}}_t&=Y^{\alpha,x,\loi{\theta}}_T +\int^T_t \left(f(X^{x,\loi\theta}_s,\loi{X^\loi\theta_s},Z^{\alpha,x,\loi{\theta}}_s)-\alpha Y^{\alpha,x,\loi{\theta}}_s\right)\d s -\int^T_t Z^{\alpha,x,\loi{\theta}}_s\d W_s.
     \end{align}
Let us remark once again that \eqref{EBSDE: Decoupled EBSDE} becomes \eqref{EBSDE: EBSDE} when we replace $x$ by a $\mathcal{G}$-measurable r.v. $\theta\in\L{2q+2}(\Omega;\R^d)$.}

Let us now give the two different sets of assumptions we are going to use to prove the main result:
\begin{hyp*}[$\H_1$]\label{EBSDE: H1}\begin{enumerate}$\vspace{0pt}$
\item (\nameref{SDE: H1}) and (\nameref{EBSD: H0}) hold,
\item $\nu >K^\sigma_x+\sqrt{2K^\sigma_x} K^f_z$,

\item $\sigma$ is invertible.
\end{enumerate}
\end{hyp*}
\begin{hyp*}[$\H_2$]\label{EBSDE: H2}\begin{enumerate}$\vspace{0pt}$
    \item (\nameref{SDE: H2-2}) and (\nameref{EBSD: H0}) hold.
\end{enumerate}
\end{hyp*}

\begin{prop}\label{EBSDE: borne Y}
    Let us assume that (\nameref{EBSDE: H1}) or (\nameref{EBSDE: H2}) hold. Then, for all $\alpha>0$, $x\in \R^d$, $\theta \in \L{2q+2}$, there
     exists a unique solution 
     $(Y^{\alpha,x,\loi\theta},Z^{\alpha,x,\loi\theta})$  to \eqref{EBSDE: Decoupled alpha BSDE} such that $Z^{\alpha,x,\loi\theta}\in\L{2}_\text{loc}(\Omega; \L{2}(\R_+; \R^d))$ and there exists $C$ that does not depend on $\alpha$, such that for all $x\in\R^d,\theta\in{ \L{2q+2} }(\Omega;\R^d)$:
    \begin{align}
    \abs*{Y^{\alpha,x,\loi\theta}_0}&\leq \frac{C}{\alpha}\left(1+\abs*{x}^{q+1}+\norm*{\theta}^{{ q+1}}_{ 2q+2}\right).\label{EBSDE: Y alpha bound decoupled}
    \end{align}
    Moreover, there exists { a measurable function} $\zeta^\alpha:\R^d\times\Pcal_{ 2q+2}(\R^d)\rightarrow\R^d$ such that $Z^{\alpha,x,\loi \theta}_t=\zeta^{\alpha}(X^{x,\loi \theta}_t,\loi{X^{\loi \theta}_t})$ a.s. for a.e. $t \in \mathbb{R}_+$, all $x \in \R^d$ and $\theta \in L^{2q+2}$.
\end{prop}

\begin{proof} 
    If we consider a solution $(Y^{\alpha,x,\loi\theta},Z^{\alpha,x,\loi\theta})$ to \eqref{EBSDE: Decoupled alpha BSDE} such that $Z^{\alpha,x,\loi\theta}\in\L{2}_\text{loc}(\Omega; \L{2}(\R_+; \R^d))$ and $Y^{\alpha,x,\loi\theta} \in \bigcup_{T>0}{\mathcal{S}_T^2}$, then we have the \emph{a priori} estimate 
    \begin{align*}
       \abs{Y_0^{\alpha,x,\loi\theta}}&\leq \E^{\Q^{\alpha}}\seg*{\int^{+\infty}_0 e^{-\alpha s}\abs*{f(X^{x,\loi\theta}_s,\loi{X^{\loi\theta}_s},0)}\d s} \leq \int^{+\infty}_0 e^{-\alpha s}C(1+\E^{\Q^{\alpha}}\seg*{|X^{x,\loi\theta}_s|^{q+\varepsilon}} + \norm*{X^{\loi\theta}_s}^{q+\varepsilon}_{ 2q+2})\d s,
    \end{align*}
    where $\Q^{\alpha}$ is the Girsanov's probability associated to the bounded process $\beta^{\alpha} $ defined for all $s\geq 0$ as
\begin{equation}\label{beta 0}
\beta^\alpha :=\frac{f(X^{x,\loi\theta}_s,\loi{X^\loi\theta_s},Z^{\alpha,x,\loi\theta}_s)-f(X^{x,\loi\theta}_s,\loi{X^\loi\theta_s},0)}{\abs*{Z^{\alpha,x,\loi\theta}_s}^2}(Z^{\alpha,x,\loi\theta})^\top\1{Z^{\alpha,x,\loi\theta}_s\neq 0}.
\end{equation}
    Let us remark that $|\beta^{\alpha}|_{\infty} \leq K^f_z$ uniformly in $\alpha$. Then, we can apply \autoref{SDE: H1: prop: Estim unif T Q} under (\nameref{EBSDE: H1}) or \autoref{SDE: H2-2: prop: estim moments} under (\nameref{EBSDE: H2}), Jensen's and Young's inequalities to get that there exists $C>0$ that does not depend on $\alpha$, such that
    \begin{equation}
        \label{estim EQ moment X}
        \sup{s \geq 0} \E^{\Q^{\alpha}}\seg*{|X^{x,\loi\theta}_s|^{q+\varepsilon}} \leq C\left(1+\abs*{x}^{q+\varepsilon}+\norm*{\theta}^{q+\varepsilon}_{q+\varepsilon}\right) \leq C\left( 1+ \abs*{x}^{q+1}+\norm*{\theta}^{q+1}_{2q+2}\right) 
    \end{equation}
    and then
    \begin{align*}
       \abs{Y_0^{\alpha,x,\loi\theta}}&\leq \frac{C}{\alpha}\left(1+\abs*{x}^{q+1}+\norm*{\theta}^{q+1}_{ 2q+2}\right).
    \end{align*}
    Thanks to this \emph{a priori} estimate, we can directly adapt the proof of \cite[Lemma 3.1]{Briand_Stability-BSDEs} which is restricted to the bounded framework in order to get the existence and uniqueness result.  Finally, the Markovian representation of $Z^{\alpha, x, \loi{\theta}}$ is classical, see e.g. the 4th step in the proof of \cite[Theorem 18]{Hu-Lemmonier}.

\end{proof}
{ Let us recall the following flow property, see \cite[Lemma 4.25 and Remark 4.26]{Delarue-Carmona_Book}, that will be useful later.} 
\begin{prop}\label{EBSDE: Flow property}For $X^\loi\theta$ the solution to \eqref{SDE: MKV SDE} starting from $\theta\in{ \L{2q+2} }(\Omega;\R^d)$ and $X^{x,\loi\theta}$ the solution to \eqref{SDE: Decoupled MKV SDE} starting from $x\in\R^d$, we have
    \begin{align*}
Y^{\alpha,x,\loi{\theta}}_t&=Y_0^{\alpha,X^{x,\loi\theta}_t,\loi{X^\loi\theta_t}}.
    \end{align*}  
\end{prop}

 The main goal of this section is to prove the following theorem.
 {
 \begin{thm}\label{EBSDE: Existence and uniqueness}
    Under (\nameref{EBSDE: H1}) or (\nameref{EBSDE: H2}) 
 the following holds:
    there exists a deterministic function $\bar u:\R^d\times\Pcal_{ 2q+2}(\R^d)\rightarrow\R$ satisfying that there exists $C>0$ such that for all $x,x'\in\R^d,\theta,\theta'\in{ \L{2q+2} }(\Omega;\R^d)$ we have
    \begin{enumerate} 
        \item $\bar{u}(0,\mu^*)=0$ where $\mu^*$ is the unique invariant measure of \eqref{SDE: MKV SDE},
        \item $\abs*{ \bar u (x,\loi\theta)}\leq C(1+\abs*{x}^{q+1}+\norm*{\theta}^{q+1}_{2q+2})$,
        \item depending on the set of assumption, $\bar u $ satisfies
            \begin{enumerate}
            \item for (\nameref{EBSDE: H1}), $$\abs*{\bar u(x,\loi\theta)-\bar u(x',\loi{\theta'})}\leq C(1+\abs*{x}^q+\abs*{x'}^q+\norm*{\theta}^q_{ 2q+2}+\norm*{\theta'}^q_{ 2q+2})\left(\abs*{x-x'}^\epsilon+ \W(\loi\theta,\loi{\theta'})^\epsilon\right),$$
            \item for (\nameref{EBSDE: H2}),
            { $$\abs*{\bar u(x,\loi\theta)-\bar u(x',\loi{\theta'})}\leq C(1+\abs*{x}^{q+1/2}+\abs*{x'}^{q+1/2}+\norm*{\theta}^{q+1/2}_{ 2q+2}+\norm*{\theta'}^{q+1/2}_{ 2q+2})\left(\abs*{x-x'}^{\epsilon/2}+ \w(\loi\theta,\loi{\theta'})^{\epsilon/2}\right).$$}
        \end{enumerate}
    \end{enumerate}
     For $x\in \R^d$ and $\theta \in {\L{ 2q+2} }(\Omega;\R^d)$, we set $Y_t=\bar u(X^{x,\loi \theta}_t,\loi{X^{\loi \theta}_t})$. Then there exists $\lambda \in \mathbb{R}$ and $Z\in\L{2}_\text{loc}(\Omega;\L{2}(\R_+;\R^d))$ such that $(Y,Z,\lambda)$ is a solution to $\eqref{EBSDE: Decoupled EBSDE}$. Moreover, there exists { a measurable function} $\bar \zeta:\R^d\times\Pcal_{ 2q+2}(\R^d)\rightarrow\R^d$ such that $Z_t=\bar\zeta(X^{x,\loi \theta}_t,\loi{X^{\loi \theta}_t})$ a.s. for a.e. $t \in \mathbb{R}_+$.\\ 
    Finally $\lambda$ is unique in the class of Markovian solutions $(Y,Z,\lambda)$ such that $Y=u(X^{x,\loi \theta},\loi{X^{\loi \theta}})$ where $u$ satisfies \textit{2.}, $Z \in \L{2}_\text{loc}(\Omega;\L{2}(\R_+;\R^d))$ and $Z = \zeta((X^{x,\loi \theta},\loi{X^{\loi \theta}})$. Moreover, if $\sigma$ is uniformly elliptic, which is already satisfied under (\nameref{EBSDE: H2}), then $(Y,Z)$ is unique in this class if we also add that $u$ satisfies \textit{1.} and \textit{3.}
\end{thm}}
\begin{rmq}
    The existence and uniqueness of $\mu^*$ is a direct consequence of \cite[Theorem 3.1]{Wang_Landau} and \autoref{SDE: H1: thm: Wp exponential contractivity} or \autoref{SDE: H2-2: thm: Wp exponential contractivity}. 
\end{rmq}
{
\begin{rmq}
    We obtain an existence and uniqueness result for \eqref{EBSDE: EBSDE} by evaluating $(Y^{\cdot,\loi\theta},Z^{\cdot,\loi\theta})$ at $\theta$ in \autoref{EBSDE: Existence and uniqueness}.
\end{rmq}}

The remainder of this section consists of proving \autoref{EBSDE: Existence and uniqueness} separately for each set of assumptions.

\subsection{Under the assumptions \texorpdfstring{{\nameref{EBSDE: H1}}}{EBSDE: H1}}
\label{subsection:strong:dissip:BSDE}

Let us assume that (\nameref{EBSDE: H1}) holds.

\begin{proof}[Proof of \autoref{EBSDE: Existence and uniqueness}]
\textbf{\underline{Existence:}}
    Let us define $u^\alpha:\R^d\times \Pcal_{ 2q+2} \to \R$, as $u^\alpha(x,\loi\theta)=Y^{\alpha, x,\loi\theta}_0$ for all $x\in\R^d,\, \theta\in{ \L{2q+2} }(\Omega;\R^d)$ and let us prove that it satisfies an equicontinuous estimate. Let $x,x'\in\R^d$ and $\theta,\theta'\in{ \L{2q+2} }(\Omega;\R^d)$. As in \cite[Lemma 2.6]{DEBUSSCHE2011407}, we apply Itô's formula to $e^{-\alpha t}(Y^{\alpha,x,\loi\theta}_t-Y^{\alpha,x',\loi{\theta'}}_t)$, in order to get
    \begin{align*}
        u^\alpha(x,\loi\theta)-u^\alpha(x',\loi{\theta'}):&=Y^{\alpha x,\loi\theta}_0-Y^{\alpha,x',\loi{\theta'}}_0=e^{-\alpha T}(Y^{\alpha,x,\loi\theta}_T-Y^{\alpha,x',\loi{\theta'}}_T) -\int^T_0 e^{-\alpha s}( Z^{\alpha,x,\loi\theta}_s-Z^{\alpha,x' ,\loi{\theta'}}_s)\d W_s
        \\&\quad+\int^T_0 e^{-\alpha s}(f(X^{x,\loi\theta}_s,\loi{X^\loi\theta_s},Z^{\alpha,x,\loi\theta}_s)-f(X^{x',\loi{\theta'}}_s,\loi{X^\loi{\theta'}_s},Z^{\alpha,x',\loi{\theta'}}_s)) \d s \\
        &\leq e^{-\alpha T}(Y^{\alpha,x,\loi\theta}_T-Y^{\alpha,x',\loi{\theta'}}_T) -\int^T_0Z^{\alpha,x,\loi\theta}_s-Z^{\alpha,x' ,\loi{\theta'}}_s(\d W_s-\beta^\alpha_s\d s)
        \\&\quad+\int^T_0 e^{-\alpha s}\left(f(X^{x,\loi\theta}_s,\loi{X^\loi\theta_s},Z^{\alpha,x',\loi{\theta'}}_s)-f(X^{x',\loi{\theta'}}_s,\loi{X^\loi{\theta'}_s},Z^{\alpha,x',\loi{\theta'}}_s)\right) \d s ,
    \end{align*}
    where $\beta^\alpha$ is a bounded progressively measurable process given by, for $s\geq 0$, \begin{equation}\label{beta alpha}
        \beta^\alpha_s:=\frac{f(X^{x,\loi\theta}_s,\loi{X^\loi\theta_s},Z^{\alpha,x,\loi\theta}_s)-f(X^{x,\loi\theta}_s,\loi{X^\loi\theta_s},Z^{\alpha,x' ,\loi{\theta'}}_s)}{\abs*{Z^{\alpha,x,\loi\theta}_s-Z^{\alpha,x' ,\loi{\theta'}}_s}^2}(Z^{\alpha,x,\loi\theta}_s-Z^{\alpha,x' ,\loi{\theta'}}_s)^\top\1{Z^{\alpha,x,\loi\theta}_s\neq Z^{\alpha,x' ,\loi{\theta'}}_s}.
    \end{equation}
    Let us set $\~\Q^\alpha$ the Girsanov's probability associated to $\beta^\alpha$. 
       { Then, by taking the expectation under $\~\Q^\alpha$, by using \autoref{EBSDE: borne Y}, (\nameref{EBSD: H0}), Cauchy-Schwarz's inequality and \autoref{SDE: H1: prop: Estim unif T Q} with $2q$
    if $q\geq 1$ or with $2$ and Jensen's inequality if $0\leq q<1$,
    the fact that $\w\leq\W$, Jensen's inequalities and Young's inequalities,}
    \begin{align*}
        \abs*{u^\alpha(x,\loi\theta)-u^\alpha(x',\loi{\theta'})}&\leq e^{-\alpha T}\~ \E^\alpha\seg*{\abs*{Y^{\alpha,x,\loi\theta}_T-Y^{\alpha,x',\loi{\theta'}}_T}} 
        \\&\quad+\int^T_0 e^{-\alpha s}\~ \E^\alpha\seg*{\abs*{f(X^{x,\loi{\theta}}_s,\loi{X^\loi\theta_s},Z^{\alpha,x',\loi{\theta'}}_s)-f(X^{x',\loi{\theta'}}_s,\loi{X^\loi{\theta'}_s},Z^{\alpha,x',\loi{\theta'}}_s)} }\d s \\
        &\leq \frac{C}{\alpha}(1+\abs*{x}^{q+1}+\abs*{x'}^{q+1}+\norm*{\theta}^{q+1}_{ 2q+2}+\norm*{\theta'}^{q+1}_{ 2q+2})e^{-\alpha T} \\
        &\quad+ C\int^T_0 e^{-\alpha s}\~\E^{\alpha}_T\left[(1+\abs*{X^{x,\loi\theta}_s}^q+\abs*{X^{x',\loi{\theta'}}_s}^q + \norm*{X^{\loi\theta}_s}^q_{ 2q+2} +\norm*{X^\loi{\theta'}_s}^q_{ 2q+2})\right.\\&\quad\left.\times\left(\abs*{X^{x,\loi{\theta}}_s-X^{x',\loi{\theta'}}_s}^\epsilon + \W(\loi{X^\loi\theta_s},\loi{X^{\loi{\theta'}}_s})^\epsilon\right)     \right]\d s\\
        &\leq  \frac{C}{\alpha}(1+\abs*{x}^{q+1}+\abs*{x'}^{q+1}+\norm*{\theta}^{q+1}_{ 2q+2}+\norm*{\theta'}^{q+1}_{ 2q+2})e^{-\alpha T}\\
        &\quad +C(1+\abs*{x}^q+\abs*{x'}^q+\norm*{\theta}^q_{ 2q+2}+\norm*{\theta'}^q_{ 2q+2})\int^T_0 e^{-\alpha s}\left( \~ \E^\alpha\seg*{\abs*{X^{x,\loi{\theta}}_s-X^{x',\loi{\theta'}}_s}^2}^{\epsilon/2} + \W(\loi{X^{\loi\theta}_s},\loi{X^{\loi{\theta'}}_s})^\epsilon \right)\d s,
    \end{align*}
    where $C$ does not depend on $\alpha$. Then, we just have to apply \autoref{SDE: H1: thm: Wp exponential contractivity} and \autoref{SDE: H1: cor: Wp exponential contractivity} to get
    \begin{align*}
        \abs*{u^\alpha(x,\loi\theta)-u^\alpha(x',\loi{\theta'})}&\leq \frac{C}{\alpha}(1+\abs*{x}^{q+1}+\abs*{x'}^{q+1}+\norm*{\theta}^{q+1}_{ 2q+2}+\norm*{\theta'}^{q+1}_{ 2q+2})e^{-\alpha T} \\&\quad+ C(1+\abs*{x}^q+\abs*{x'}^q+\norm*{\theta}^q_{ 2q+2}+\norm*{\theta'}^q_{ 2q+2})\left( \abs*{x-x'}^\epsilon +\W(\loi\theta,\loi{\theta'})^\epsilon\right)\int^T_0 e^{-\~\Lambda s} \d s
    \end{align*}
    and, by taking $T\rightarrow + \infty$, 
    $$\abs*{u^\alpha(x,\loi\theta)-u^\alpha(x',\loi{\theta'})} \leq C(1+\abs*{x}^q+\abs*{x'}^q+\norm*{\theta}^q_{ 2q+2}+\norm*{\theta'}^q_{ 2q+2})\left( \abs*{x-x'}^\epsilon +\W(\loi\theta,\loi{\theta'})^\epsilon\right).$$
    Therefore, $\bar{u}^\alpha:=u^\alpha-u^\alpha(0,\mu^*)$ satisfies also, for all $x,x'\in\R^d$, $\theta,\theta'\in{ \L{2q+2} }(\Omega;\R^d)$, $\alpha>0$,
    \begin{equation}
        \label{regularity ualpha}
    \abs*{\bar u^\alpha(x,\loi\theta)- \bar u^\alpha(x',\loi{\theta'})}\leq C(1+\abs*{x}^q+\abs*{x'}^q+\norm*{\theta}^q_{ 2q+2}+\norm*{\theta'}^q_{ 2q+2})\left( \abs*{x-x'}^\epsilon +\W(\loi\theta,\loi{\theta'})^\epsilon\right)
    \end{equation}
    {
    and 
    \begin{equation}
    \label{bound ualphabar}
    \abs*{\bar u^\alpha(x,\loi\theta)}\leq C(1+\abs*{x}^{q+1}+\norm*{\theta}^{q+1}_{ 2q+2}).
    \end{equation}}
    Consequently, by using a diagonal procedure, there exists a sequence $(\alpha_n)_n\searrow_0$ and $\bar u$ such that $\bar u^{\alpha_n}\rightarrow \bar u$ on a countable dense subset of $\R^d\times \Pcal_{ 2q+2}(\R^d)$ (which exists since both are separable, see e.g. the paragraph after \cite[Corollary 5.6]{Delarue-Carmona_Book}). We can extend this convergence on the whole space due to the equicontinuous estimate. { It is easy to see that $\bar{u}$ satisfies \textit{1.}, \textit{2.} and \textit{3.} thanks to estimates \eqref{regularity ualpha}-\eqref{bound ualphabar}. We set $\bar{Y}^{\alpha,x,\loi\theta}:=\bar{u}^\alpha(X^{x,\loi\theta},\loi{X^\loi\theta})$ and $\bar{Y}^{x,\loi\theta}:=\bar{u}(X^{x,\loi\theta},\loi{X^\loi\theta})$. Furthermore, since $\abs*{\alpha u^\alpha(0,\mu^*)}\leq C$ due to \eqref{EBSDE: Y alpha bound decoupled}, we have that there exists $\lambda$ such that $(\alpha_n u^{\alpha_n}(0,\mu^*))_n$  tends to $\lambda$ up to a subsequence that, by a slight abuse of notation, we still denote $(\alpha_n)_n$. Then, due to \eqref{bound ualphabar} and classical estimates on SDEs, we know that, for any $T>0$,
    
    $$ 
    \E\seg*{\sup{t\leq T}\sup{\alpha>0} \abs*{\bar{Y}^{\alpha,x,\loi\theta}_t}^2} \leq C\left(1 + \E\seg*{\sup{t\leq T} \abs*{{X}^{x,\loi\theta}_t}^{2q+2}}+ \norm*{\theta}^q_{ 2q+2}\right) <+\infty.
    $$
    Hence, by dominated convergence theorem, we get, for all $T>0$, 
    $$ \E\int^T_0 \abs*{\bar{Y}^{\alpha_n,x,\loi\theta}_t-\bar{Y}^{x,\loi\theta}_t}^2 \d t \longrightarrow 0 \quad \text{ and }\quad  \abs*{\bar{Y}^{\alpha_n,x,\loi\theta}_T-\bar{Y}^{x,\loi\theta}_T}^2  \longrightarrow 0.
    $$
    
     By applying Itô's formula to $\abs*{\~Y_t}^2$ where $\~Y:=\bar{Y}^{\alpha_n,x,\loi\theta}-\bar{Y}^{\alpha_m,x,\loi\theta}$ and $\~Z:={Z}^{\alpha_n,x,\loi\theta}-{Z}^{\alpha_m,x,\loi\theta}$, for $n,m\in\N$, one could get that $\left( Z^{\alpha_n,x,\loi\theta} \right)_n$ is a Cauchy sequence in $\L{2}_\text{loc}(\Omega;\L{2}(\R_+,\R^d))$ and thus that there exists $\bar{Z}^{x,\loi\theta}\in\L{2}_\text{loc}(\Omega;\L{2}(\R_+,\R^d))$ such that, for all $T>0$,
    $$\E\int^T_0 \abs*{\bar{Z}^{\alpha_n,x,\loi\theta}_t-\bar{Z}^{x,\loi\theta}_t}^2 \d t \longrightarrow 0 .$$
    By applying same arguments as in \cite{Furhman-Tessitore-Hu_EBSDE_Banach}, we can also prove the existence of a measurable function $\bar{\zeta}: \R^d \times \mathcal{P}_{2q+2}(\R^d) \to \R^d$ such that $\bar{Z}^{x,\loi\theta}=\bar\zeta(X^x,\loi{X^\loi\theta})$.
    Finally, we can pass to the limit in the BSDE
    \begin{align*}
        \bar Y^{\alpha_n,x,\loi\theta}_0&=\bar Y^{\alpha_n,x,\loi\theta}_T+\int^T_0 \left(f(X^{x,\loi\theta}_s,\loi{X^\loi\theta_s}, Z^{\alpha_n,x,\loi\theta}_s) - \alpha_n\bar Y_s^{\alpha_n,x,\loi\theta}- \alpha_n u^{\alpha_n}(0,\mu^*)\right)\d s -\int^T_0 \bar Z^{\alpha_n,x,\loi\theta}_s \d W_s,
    \end{align*}
    in order to show that $(\bar{ Y}^{x,\loi\theta},\bar{Z}^{x,\loi\theta},\lambda)$ is a solution to \eqref{EBSDE: Decoupled EBSDE}.
    }
    
    \par \textbf{\underline{Uniqueness:}}
    { 
    Let $(Y^{x,\loi\theta},Z^{x,\loi\theta},\lambda)$ and $(Y'^{x,\loi\theta},Z'^{x,\loi\theta},\lambda')$ be solutions to \eqref{EBSDE: Decoupled EBSDE} with $(u,\zeta)$ and $(u',\zeta')$ respectively associated to $(Y,Z)$ and $(Y',Z')$ with same property as $\bar u$ defined above. Denoting $\~Y:=Y^{x,\loi\theta}-Y'^{x,\loi\theta}$ and $\~Z:=Z^{x,\loi\theta}-Z'^{x,\loi\theta}$, we have 
    \begin{equation}\label{ito uniqueness}
    -\d \~Y_t = \~Z_t\beta_t\d t -(\lambda-\lambda')\d t -\~Z_t\d W_t= -(\lambda-\lambda')\d t-\~Z_t\d W^{\~\Q}_t,
    \end{equation}
where $\beta$ is the progressively measurable process given by 
\begin{equation}\label{beta}
        \beta_t:=\frac{f(X^{x,\loi\theta}_t,\loi{X^\loi\theta_t},Z^{\alpha,x,\loi\theta}_t)-f(X^{x,\loi\theta}_t,\loi{X^\loi\theta_t},Z'^{\alpha,x ,\loi{\theta}}_t)}{\abs*{Z^{\alpha,x,\loi\theta}_t-Z'^{\alpha,x ,\loi{\theta}}_t}^2}(Z^{\alpha,x,\loi\theta}_t-Z'^{\alpha,x ,\loi{\theta}}_t)^\top\1{Z^{\alpha,x,\loi\theta}_t\neq Z'^{\alpha,x ,\loi{\theta}}_t}
    \end{equation}
and $\d W^{\~\Q}_t := \d W_t -\beta_t\d t$.
    Thus, by taking the expectation w.r.t. the new Girsanov's probability ${\~\Q}$ associated to $\beta$, we obtain
    $$
    \frac{\~Y_0-\E^{\~\Q}\seg*{\~Y_T}}{T} = \lambda-\lambda'.
    $$
    Then, due to the polynomial growth of of $u,u'$ and \autoref{SDE: H1: prop: Estim unif T Q}, we obtain
    \begin{equation}\label{SDE: H1: eq: unique lambda}
        \abs*{\lambda-\lambda'}\leq \frac{C}{T}\E^{\~\Q}\seg*{1+\abs*{X^{x,\loi\theta}_T}^{q+1}+\norm*{X^\loi\theta_T}^{q+1}_{2q+1}+\abs*{x}^{q+1}+\norm*{\theta}^{q+1}_{2q+1}}\leq \frac{C}{T}(1+\abs*{x}^{q+1}+\norm*{\theta}^{q+1}).
    \end{equation}
    Thus, by taking $T \rightarrow + \infty$, we obtain that $\lambda$ is unique. 
    
    Let us show now that $u=u'$.
    {
    We can remark that $(u(X^{x,\mu^*},\mu^*),\zeta(X^{x,\mu^*},\mu^*),\lambda)$ and $(u'(X^{x,\mu^*},\mu^*),\zeta'(X^{x,\mu^*},\mu^*),\lambda)$ are solutions to the same classical,
    ergodic BSDE, i.e. a non-distribution dependent one. Since we assume that $\sigma$ is invertible, we can apply \cite[Theorem 20]{Hu-Lemmonier} in order to get that $u(.,\mu^*)=u'(.,\mu^*)$. Hence, since for all $T>0$, $x \in \R^d$ and $\theta \in { \L{2q+2} }(\Omega;\R^d)$, $u(X^{x,\loi\theta}_T,\mu^*)=u'(X^{x,\loi\theta}_T,\mu^*)$, we have, due to \eqref{bound ualphabar}, the locally Lipschitz property of $u,u'$, \autoref{SDE: H1: cor: Wp exponential contractivity} and \eqref{SDE: H1: Estim unif T decoupled Q},
    \begin{align*}
        \abs*{u(x,\theta) -u'(x,\theta)} &=  \abs*{\E^{\~\Q}\seg*{u(X^{x,\loi\theta}_T,\loi{X^\loi\theta_T})-u'(X^{x,\loi\theta}_T,\loi{X^\loi\theta_T})}}\\
        &= \abs*{\E^{\~\Q}\seg*{u(X^{x,\loi\theta}_T,\loi{X^\loi\theta_T})-u(X^{x,\loi\theta}_T,\mu^*) + u'(X^{x,\loi\theta}_T,\mu^*)-u'(X^{x,\loi\theta}_T,\loi{X^\loi\theta_T})}}\\
        &\leq \abs*{\E^{\~\Q}\seg*{u(X^{x,\loi\theta}_T,\loi{X^\loi\theta_T})-u(X^{x,\loi\theta}_T,\mu^*)}} +  \abs*{\E^{\~\Q}\seg*{u'(X^{x,\loi\theta}_T,\mu^*)-u'(X^{x,\loi\theta}_T,\loi{X^\loi\theta_T})}}\\
        &= \abs*{\E^{\~\Q}\seg*{u(X^{x,\loi\theta}_T,\loi{X^\loi\theta_T})-u(X^{x,\loi\theta}_T,\loi{X^{\mu^*}_T})}} +  \abs*{\E^{\~\Q}\seg*{u'(X^{x,\loi\theta}_T,\loi{X^{\mu^*}_T})-u'(X^{x,\loi\theta}_T,\loi{X^\loi\theta_T})}}\\
        &\leq \E^{\~\Q}\seg*{\abs*{u(X^{x,\loi\theta}_T,\loi{X^\loi\theta_T})-u(X^{x,\loi\theta}_T,\loi{X^{\mu^*}_T})}}+\E^{\~\Q}\seg*{\abs*{u'(X^{x,\loi\theta}_T,\loi{X^{\mu^*}_T})-u'(X^{x,\loi\theta}_T,\loi{X^\loi\theta_T})}}\\
        &\leq C\left(1+\abs*{x}^q+\norm*{\theta}^q_{2q+2}+\WW_{2q+2}^q(\mu^*,0)\right)\W(X^{\loi\theta}_T,\loi{X^{\mu^*}_T})^{\epsilon}\\
        &\leq C\left(1+\abs*{x}^q+\norm*{\theta}^q_{2q+2}+\WW_{2q+2}^q(\mu^*,0)\right)\W(\loi{\theta},\mu^*)^\epsilon e^{-\gamma\epsilon T}\overset{T\rightarrow\infty}{ \longrightarrow }0.
    \end{align*}
    Uniqueness of $Z$ is a direct consequence of It\^o's formula applied to \eqref{ito uniqueness}.
    }
}\end{proof}

\subsection{Under the assumption \texorpdfstring{\nameref{EBSDE: H2}}{EBSDE: H2}}\label{EBSDE-subsection 2}

\label{subsection:weak:dissip:BSDE}

Let us assume that (\nameref{EBSDE: H2}) holds.
\begin{proof}[Proof of \autoref{EBSDE: Existence and uniqueness}]
\underline{\textbf{Existence:}} As in the proof of Section \ref{subsection:strong:dissip:BSDE} we will prove that $u^\alpha$, defined for all $x\in\R^d,\, \theta\in{ \L{2q+2} }(\Omega;\R^d)$ as $u^\alpha(x,\loi\theta):=Y^{\alpha, x,\loi\theta}_0$, satisfies an equicontinuous estimate, firstly with respect to $x$ and then with respect to $\loi \theta$. 
Let $x\in\R^d$ and $\theta\in{\L{2q+2} }(\Omega;\R^d)$. We have, by applying It\^o's formula to $e^{-\alpha t}Y^{\alpha,x,\loi\theta}_t$,
    \begin{align}
    \nonumber
        u^\alpha(x,\loi\theta)&=Y^{\alpha,x,\loi\theta}_0 =e^{-\alpha T}Y^{\alpha,x,\loi\theta}_T +\int^T_0 e^{-\alpha s}f(X^{x,\loi\theta}_s,\loi{X^\loi\theta_s},Z^{\alpha,x,\loi\theta}_s)\d s-\int^T_0Z^{\alpha,x,\loi\theta}_s\d W_s\\
        &=e^{-\alpha T}Y^{\alpha,x,\loi\theta}_T + \int^T_0 e^{-\alpha s}f(X^{x,\loi\theta}_s,\loi{X^\loi\theta_s},0)\d s \int^T_0 Z^{\alpha,x,\loi\theta}_s(\d W_s -\beta^\alpha_0(X^{x,\loi\theta}_s,\loi{X^\loi\theta_s})\d s), \label{eq:rep:ualpha:beta0}
    \end{align}
with 

\begin{equation}\label{beta 0 0}
\beta_0^\alpha(\rm x,\mu) :=\frac{f(\rm x,\mu,\zeta^{\alpha}(\rm x,\mu))-f(\rm x,\mu,0)}{\abs*{\zeta^{\alpha}(\rm x,\mu)}^2}\zeta^{\alpha}(\rm x,\mu)^\top\1{\zeta^{\alpha}(\rm x,\mu)\neq 0}, \quad \forall~(\rm x,\mu) \in \R^d\times \mathcal{P}_{2q+2}.
\end{equation}

    We already know that $\beta^\alpha_0$ is bounded by $K^f_z$, so we can apply Girsanov's theorem. However, $\beta^\alpha_0$ is not necessarily Lipschitz but we can approximate it thanks to the following lemma.

    \begin{lem}
    \label{lem: approx beta}
        There exists a uniformly bounded sequence of Lipschitz function $(\beta^\alpha_{0,n})_n$ such that, $\sup{n \in \N} \norm{\beta_{0,n}^\alpha}_{\infty}<+\infty$,
        \begin{equation*}
            |\beta^\alpha_{0,n}(x,\loi{\theta})-\beta^\alpha_{0,n}(x',\loi{\theta'})|\leq C_n|x-x'|+C_n \WW_2(\loi{\theta},\loi{\theta'}),\quad\quad \forall (x,\theta),(x',\theta')\in\R^d\times{ \L{2q+2} }(\Omega;\R^d),
        \end{equation*}
        and
        \begin{equation*}
            \lim{n \to +\infty} \beta^\alpha_{0,n}(x,\loi{\theta})= \beta^\alpha_0(x,\loi{\theta}),\quad\quad \forall (x,\theta)\in\R^d\times{ \L{2q+2} }(\Omega;\R^d).
        \end{equation*}
    \end{lem}
    \begin{proof}
    By lifting $\loi\theta$ on the Hilbert space ${ \L{2} (\Omega;\R^d) \supset \L{2q+2}(\Omega;\R^d)}$ we get a function $\v\beta^\alpha:\R^d\times{\L{2} }(\Omega;\R^d)$ such that for all $\v\theta \in \L{2q+2}(\Omega;\R^d)\sim\loi{\theta}$ and $x \in \R^d$, $\v\beta^\alpha(x,\v\theta)=\beta^\alpha(x,\loi\theta)$.
    Then, we can apply \cite[Lemma 3.5]{DEBUSSCHE2011407} in order to get the existence of a uniformly bounded sequence of Lipschitz lifted functions $(\v\beta^\alpha_{0,n})_n$ such that $\v\beta^\alpha_{0,n}(x,\v\theta)\rightarrow\v\beta^\alpha(x,\v\theta)$ when $n \to + \infty$, which implies that
    \begin{equation*}
    \beta^\alpha_{0,n}(x,\loi\theta):=\v\beta^\alpha_{0,n}(x,\v\theta) \longrightarrow \v\beta^\alpha(x,\v\theta)=\beta^\alpha(x,\loi\theta),\quad \forall (x,\theta)\in\R^d\times{ \L{2q+2} }(\Omega;\R^d).
    \end{equation*}
    Moreover, for all $(x,\v\theta),(x',\v\theta')\in\R^d\times{ \L{2q+2} }(\Omega;\R^d)$ with $\loi{\v \theta} = \loi{\theta}$ and $\loi{\v\theta'} = \loi{\theta'}$, we have
    $$|\beta^\alpha_{0,n}(x,\loi\theta)-\beta^\alpha_{0,n}(x,\loi{\theta'
    })|=|\v\beta^\alpha_{0,n}(x,\v\theta)-\v\beta^\alpha_{0,n}(x',\v\theta')|\leq C_n|x-x'|+C_n \E[|\v\theta-\v\theta'|^2]^{1/2}.$$
    Then, it just remains to take for $(\v\theta,\v \theta')$ an optimal coupling of $(\loi{\theta},\loi{\theta'})$ for the $\WW_2-$Wasserstein distance in order to conclude.
    \end{proof}
    
    Let us denote $(Y^{\alpha,n,x,\loi \theta},Z^{\alpha,n,x,\loi \theta})$ the solution of the infinite horizon BSDE \eqref{eq:rep:ualpha:beta0} where $\beta^{\alpha}_{0}$ is replaced by $\beta^{\alpha}_{0,n}$ and $u^{\alpha,n}(x,\loi \theta):= Y_0^{\alpha,n,x, \loi \theta}$. Then using Girsanov's theorem we get that there exists $\Q^\alpha_{0,n}$ for which $W^{\Q^\alpha_{0,n}}_t=W_t-\int^t_0 \beta^\alpha_{0,n}(X^{x,\loi\theta}_s,\loi{X^\loi\theta_s})\d s$ is a $\Q^\alpha_{0,n}-$Brownian motion and we obtain, for all $x\in \R^d,\theta \in \L{2q+2}(\Omega,\R^d)$,

    \begin{align}
        u^{\alpha,n}(x,\loi \theta) &= e^{-\alpha T} \E^{\Q^\alpha_{0,n}}\seg*{u^{\alpha,n}(X_T^{x,\loi \theta},\loi{X_T^{\loi \theta}})}+\int_0^T e^{-\alpha s} \E^{\Q^\alpha_{0,n}}\seg*{f(X_s^{x,\loi \theta},\loi{X_s^{\loi \theta}},0)}\d s\notag\\
        &= e^{-\alpha T}\Pcal_T^{\alpha,n}[u^{\alpha,n}](x,\loi \theta)+ \int_0^T e^{-\alpha s} \Pcal_s^{\alpha,n}[f(\cdot,\cdot,0)](x,\loi\theta)\d s  ,      \label{eq:ualphan}
    \end{align}
    where $(\Pcal_t^{\alpha,n})_{t\geq 0}$ is the semigroup associated to the following family of SDEs
    \begin{equation}\label{EBSDE: H2: eq: X alpha}\left\{
    \begin{array}{ll}
       \d X^{\alpha,n,x,\loi\theta}_t&=b(X^{\alpha,n,x,\loi\theta}_t,\loi{X^{\loi\theta}_t})\d t + \sigma(X^{\alpha,n,x,\loi\theta}_t)\beta^\alpha_{0,n}(X^{\alpha,n,x,\loi\theta}_s,\loi{X^\loi\theta_s})\d t +\sigma(X^{\alpha,n,x,\loi\theta}_t)\d W_t,\quad t \geq 0,    \\
        
        \quad  X^{\alpha,n,x,\loi\theta}_0&=x. 
    \end{array}
        \right.
    \end{equation}

Let us set $\bar u^{\alpha,n}:=u^{\alpha,n}-u^{\alpha,n}(0,\mu^*)$. Then, for all $x,x'\in \R^d,\theta \in \L{2q+2}(\Omega,\R^d)$, \eqref{eq:ualphan} gives us
    \begin{eqnarray}
    \nonumber
    \bar u^{\alpha,n}(x,\loi \theta)- \bar u^{\alpha,n}(x',\loi \theta) &=& e^{-\alpha T}\left(\Pcal_T^{\alpha,n}[u^{\alpha,n}](x,\loi \theta)-\Pcal_T^{\alpha,n}[u^{\alpha,n}](x',\loi \theta)\right)\\
    &&+ \int_0^T e^{-\alpha s} \left(\Pcal_s^{\alpha,n}[f(\cdot,\cdot,0)](x,\loi\theta)-\Pcal_s^{\alpha,n}[f(\cdot,\cdot,0)](x',\loi\theta)\right)\d s.
    \label{eq:ualphan:delta}
    \end{eqnarray}
 By using \autoref{EBSDE: borne Y}, \autoref{SDE: H2: prop: Estim unif T} and \autoref{SDE: H2-2: thm: coupling estimates decoupled} in \eqref{eq:ualphan:delta}, we obtain that there exists $C$ that does not depend on $n$ {and $\alpha$ }such that
        \begin{eqnarray}
            \abs*{\bar u^{\alpha,n}(x,\loi\theta)-\bar u^{\alpha,n}(x',\loi\theta)}&\leq &e^{-\alpha T}\frac{C}{\alpha}(1+\abs*{x}^{q+1}+\abs*{x'}^{q+1}+\norm*{\theta}^{q+1}_{2q+2}) \nonumber \\& &+ C(1+\abs*{x}^{q+1/2} + \abs*{x'}^{q+1/2}+\norm*{\theta}^{q+1/2}_{2q+2})\abs*{x-x'}^{\frac{\epsilon}{2}}\int^T_0 e^{-(\alpha +\frac{\hat\eta \varepsilon}{2})s}\d s\notag\\
            &&\overset{T\rightarrow\infty}{\longrightarrow} C(1+\abs*{x}^{q+1/2} + \abs*{x'}^{q+1/2}+\norm*{\theta}^{q+1/2}_{2q+2})\abs*{x-x'}^{\frac{\epsilon}{2}}.\label{EBSDE: estim equicont intermediaire 0}
        \end{eqnarray}
Then, the same diagonal argument as in the second step of the proof of \cite[Proposition 16]{Hu-Lemmonier} gives us that $\bar u^{\alpha,n} \to \bar u^{\alpha}:=u^{\alpha}-u^\alpha (0,\mu^*)$ when $n \to + \infty$, up to a subsequence. It easily implies, thanks to \eqref{EBSDE: estim equicont intermediaire 0}, that, {there exists $C$ that does not depend on $\alpha$ such that} for all $x,x'\in \R^d,\theta \in \L{2q+2}(\Omega,\R^d)$,
\begin{equation}
    \label{EBSDE: estim equicont intermediaire 1}
    \abs*{\bar u^{\alpha}(x,\loi\theta)-\bar u^{\alpha}(x',\loi\theta)}\leq  C(1+\abs*{x}^{q+1/2} + \abs*{x'}^{q+1/2}+\norm*{\theta}^{q+1/2}_{ 2q+2})\abs*{x-x'}^{\frac{\epsilon}{2}}.
\end{equation}

        It remains to deal with $\bar u^\alpha(x',\loi\theta)-\bar u^\alpha(x',\loi{\theta'})$, for all $x'\in \R^d,\theta,\theta' \in \L{2q+2}(\Omega,\R^d)$. 
        To that end, we will not use the same probability change as above. Indeed, we write

        \begin{align*} 
        \bar u^\alpha(x',\loi\theta)-\bar u^\alpha(x',\loi{\theta'})&=e^{-\alpha T}\left(Y^{\alpha,x',\loi\theta}_T-Y^{\alpha,x',\loi{\theta'}}_T\right) -\int^T_0 (Z^{\alpha,x',\loi\theta}_s-Z^{\alpha,x',\loi{\theta'}}_s)(\d W_s -\beta^\alpha_s \d s)\\&\quad+ \int^T_0 e^{-\alpha s}\left(f(X^{x',\loi\theta}_s,\loi{X^\loi\theta_s},Z^{\alpha,x',\loi{\theta'}}_s)-f(X^{x',\loi{\theta'}}_s,\loi{X^ \loi{\theta'}_s},Z^{\alpha,x',\loi{\theta'}}_s)\right)\d s ,
        \end{align*}
        for $\beta^\alpha_s$ given in \eqref{beta alpha} where $x$ is replaced by $x'$.
        We have, due to the boundedness of $\beta^\alpha$ and Girsanov's theorem, that there exists $\Q^\alpha$ under which $\d W^{\Q^\alpha}:=\d W_s -\beta^\alpha_s \d s$ is a $\Q^\alpha$-Brownian motion. Hence, we have 
        \begin{align} \nonumber
            \bar u^\alpha(x',\loi\theta)-\bar u^\alpha(x',\loi{\theta'})&=e^{-\alpha T} \E^\alpha\seg*{\bar u^\alpha(X^{x',\loi\theta}_T,\loi{X^\loi\theta_T})-\bar u^\alpha(X^{x',\loi{\theta'}}_T,\loi{X^\loi{\theta'}_T})}\\ &\quad  + \E^\alpha\seg*{\int^T_0 e^{-\alpha s}\left(f(X^{x',\loi\theta}_s,\loi{X^\loi\theta_s},Z^{\alpha,x',\loi{\theta'}}_s)-f(X^{x',\loi{\theta'}}_s,\loi{X^ \loi{\theta'}_s},Z^{\alpha,x',\loi{\theta'}}_s) \right)\d s}.  \label{eq:rep:delta:ubaralpha}
            \end{align}

        Thanks to \eqref{eq:rep:delta:ubaralpha}, let us prove now by induction that $\bar u^\alpha$ satisfies, for all $n\geq 0$, $x'\in \R^d$, $\theta,\theta' \in \L{2q+2}(\Omega,\R^d)$,
    
    \begin{align} \nonumber
            \abs*{\bar u^\alpha(x',\loi\theta)-\bar u^\alpha(x',\loi{\theta'})} &\leq A_n(1+\abs*{x'}^{q+1/2}+\norm*{\theta}^{q+1/2}_{2q+2}+\norm*{\theta'}^{q+1/2}_{2q+2})\w(\loi\theta,\loi{\theta'})^{\epsilon/2}\\
            & \quad  + C_{\alpha}(Ce^{-\alpha T})^n(1+\abs*{x'}^{q+1}+\norm*{\theta}^{q+1}_{2q+2}+\norm*{\theta'}^{q+1}_{2q+2}),\tag{$HR_n$}\label{EBSDE: HRn}
        \end{align}    
        with 
         $(A_n)_n$ that satisfies the following induction relation: 
        \begin{equation}
         A_{n+1} =C_T +A_n Ce^{-{\frac{\hat\eta\epsilon}{2}} T}, 
         \quad n \in \mathbb{N}, \quad A_0=0,
        \end{equation}
        where $C_{\alpha}$, $C_T$, $C$ do not depend on $n$, $C_{\alpha}$ and $C$ do not depend on $T$, $C_T$ and $C$ do not depend on $\alpha$. 
First, due to \autoref{EBSDE: borne Y}, we have 
$$\abs{u^\alpha(x',\loi\theta)-u^\alpha(x',\loi{\theta'})}\leq \frac{C}{\alpha}(1+\abs*{x'}^{q+1}+\norm*{\theta}^{q+1}_{2q+2}+\norm*{\theta'}^{q+1}_{2q+2}),
$$
which proves the initial case.
        Let us now move on to the inductive step. First, due to the Cauchy-Schwarz's and Jensen's inequalities, \eqref{EBSDE: HRn} and \eqref{EBSDE: estim equicont intermediaire 1}, we have

        \begin{align*}
            &\abs*{\bar u^\alpha(x',\loi\theta)-\bar u^\alpha(x',\loi{\theta'})}\\
            \leq &e^{-\alpha T} \left( \E^\alpha\seg*{\abs*{\bar u^\alpha(X^{x',\loi{\theta}}_T,\loi{X^\loi{\theta}_T})-\bar u^\alpha(X^{x',\loi{\theta'}}_T,\loi{X^{\loi{\theta}}_T})}}+ \E^\alpha\seg*{\abs*{\bar u^\alpha(X^{x',\loi{\theta'}}_T,\loi{X^\loi\theta_T})-\bar u^\alpha(X^{x',\loi{\theta'}}_T,\loi{X^\loi{\theta'}_T})}}\right) 
                \\&+ \int^T_0 e^{-\alpha s}  \E^\alpha\seg*{\abs*{f(X^{x',\loi{\theta}}_s,\loi{X^\loi\theta_s},Z^{\alpha,x',\loi{\theta'}}_s)-f(X^{x',\loi{\theta'}}_s,\loi{X^ \loi{\theta}_s},Z^{\alpha,x',\loi{\theta'}}_s)} } \d s\\
                &+  \int^T_0 e^{-\alpha s} \E^\alpha\seg*{\abs*{f(X^{x',\loi{\theta'}}_s,\loi{X^\loi\theta_s},Z^{\alpha,x',\loi{\theta'}}_s)-f(X^{x',\loi{\theta'}}_s,\loi{X^ \loi{\theta'}_s},Z^{\alpha,x',\loi{\theta'}}_s)} }  \d s \\
            \leq&  Ce^{-\alpha T} \E^\alpha\seg*{(1+ \abs*{X^{x',\loi{\theta}}_T}^{2q+1}+\abs*{X^{x',\loi{\theta'}}_T}^{2q+1}+ \norm*{X^{\loi{\theta}}_T}_{2q+2}^{2q+1})}^{1/2}\E^\alpha\seg*{\abs*{X^{x',\loi{\theta}}_T-X^{x',\loi{\theta'}}_T}^2}^{\epsilon/4}
                \\&+ e^{-\alpha T} A_n\E^\alpha\seg*{\left( 1+\abs*{X^{x',\loi{\theta'}}_T}^{q+1/2} +\norm*{X^\loi\theta_T}^{q+1/2}_{2q+2}+\norm*{X^\loi{\theta'}_T}^{q+1/2}_{2q+2} \right)}\w(\loi{X^\loi\theta_T},\loi{X^\loi{\theta'}_T})^{\epsilon/2}\\
                & +e^{-\alpha T} C_{\alpha}(Ce^{-\alpha T})^n\E^\alpha\seg*{\left( 1+\abs*{X^{x',\loi{\theta'}}_T}^{q+1} +\norm*{X^\loi\theta_T}^{q+1}_{2q+2}+\norm*{X^\loi{\theta'}_T}^{q+1}_{2q+2} \right)}\\
                &+ C\int^T_0e^{-\alpha s} \E^\alpha\seg*{(1+ \abs*{X^{x',\loi{\theta}}_s}^{2q+1}+\abs*{X^{x',\loi{\theta'}}_s}^{2q+1}+ \norm*{X^{\loi{\theta}}_s}_{2q+2}^{2q+1})}^{1/2}\E^\alpha\seg*{\abs*{X^{x',\loi{\theta}}_s-X^{x',\loi{\theta'}}_s}^2}^{\epsilon/4} \d s \\
                &+ C\int^T_0e^{-\alpha s} \E^\alpha\seg*{(1+ \abs*{X^{x',\loi{\theta'}}_s}^{q}+ \norm*{X^{\loi{\theta}}_s}_{2q+2}^{q}+ \norm*{X^{\loi{\theta'}}_s}_{2q+2}^{q})}\w(\loi{X^\loi\theta_s},\loi{X^\loi{\theta'}_s})^\epsilon\d s.
                \end{align*}
    
            Consequently, by using \autoref{SDE: H2-2: prop: estim moments} with $p=2q+1$ or $2q$ if $q\geq 1$ and, using Cauchy-Schwarz's and Jensen's inequality, with $p=2$ otherwise, we obtain 
            \begin{align} 
            \nonumber
            \abs*{\bar u^\alpha(x',\loi\theta)-\bar u^\alpha(x',\loi{\theta'})}
            &\leq Ce^{-\alpha T}(1+ \abs*{x'}^{q+1/2}+\norm*{\theta}_{2q+2}^{q+1/2}+\norm*{\theta'}_{2q+2}^{q+1/2})\E^\alpha\seg*{\abs*{X^{x',\loi{\theta}}_T-X^{x',\loi{\theta'}}_T}^2}^{\epsilon/4}
                 \\ \nonumber &\quad+  Ce^{-\alpha T}(1+ \abs*{x'}^{q+1/2}+\norm*{\theta}_{2q+2}^{q+1/2}+\norm*{\theta'}_{2q+2}^{q+1/2}) A_n\w(\loi{X^\loi\theta_T},\loi{X^\loi{\theta'}_T})^{\epsilon/2}
                 \\ \nonumber &\quad +  C_{\alpha}(Ce^{-\alpha T})^{n+1}{\left( 1+\abs*{x'}^{q+1} +\norm*{\theta}^{q+1}_{2q+2}+\norm*{\theta'}^{q+1}_{2q+2} \right)}
                 \\ \nonumber &\quad+C(1+ \abs*{x'}^{q+1/2}+\norm*{\theta}_{2q+2}^{q+1/2}+\norm*{\theta'}_{2q+2}^{q+1/2})\int^T_0e^{-\alpha s}\E^\alpha\seg*{\abs*{X^{x',\loi{\theta}}_s-X^{x',\loi{\theta'}}_s}^2}^{\epsilon/4}\d s
                 \\ \label{eq:delta:ubaralpha:theta} &\quad+C(1+ \abs*{x'}^{q+1/2}+\norm*{\theta}_{2q+2}^{q+1/2}+\norm*{\theta'}_{2q+2}^{q+1/2})\int^T_0e^{-\alpha s}\w(\loi{X^{\loi\theta}_s},\loi{X^{\loi{\theta'}}_s})^{\epsilon/2}\d s.
        \end{align}        
            Now we want to obtain a suitable bound on $\E^\alpha\seg*{\abs*{X^{x',\loi{\theta}}_t-X^{x',\loi{\theta'}}_t}^2}$, for all $t\geq 0$. By applying Itô's formula to $\abs*{X^{x',\loi{\theta}}_t-X^{x',\loi{\theta'}}_t}^2$ under $\Q^\alpha$, we get by the Lipschitz property of $b$ and $\sigma$, and the boundedness of $\beta$,
            \begin{align*}
                \d \abs*{\~X_t}^2:=\d \abs*{X^{x',\loi{\theta}}_t-X^{x',\loi{\theta'}}_t}^2&=2\braket*{\~X_t}{b(X^{x',\loi\theta}_t,\loi{X^\loi\theta_t})-b(X^{x',\loi{\theta'}}_t,\loi{X^\loi{\theta'}_t})}\d t+{2}\braket*{\~X_t}{(\sigma(X^{x',\loi\theta}_t)-\sigma(X^{x',\loi{\theta'}}_t))\beta^\alpha_t}\d t\\
                &\quad+\norm*{\sigma(X^{x',\loi\theta}_t)-\sigma(X^{x',\loi{\theta'}}_t)}^2\d t + 2\braket*{\~X_t}{(\sigma(X^{x',\loi\theta}_t)-\sigma(X^{x',\loi{\theta'}}_t))\d W^{\Q^\alpha}_t}\\
            &\leq 2K^b_x\abs*{\~X_t}^2\d t + 2 K^b_\LL \abs*{\~X_t}\w(\loi{X^\loi\theta_t},\loi{X^\loi{\theta'}_t})\d t +{2}\sqrt{2K^\sigma_x}\abs*{\beta^\alpha}_\infty \abs*{\~X_t}^2\d t + 2K^\sigma_x\abs*{\~X_t}^2\d t+ \d M_t,
            \end{align*}
            with the martingale $\d M_t=2\braket*{\~X_t}{(\sigma(X^{x',\loi\theta}_t)-\sigma(X^{x',\loi{\theta'}}_t))\d W^{\Q^\alpha}_t}$.
            Hence, by using Young's inequality, \autoref{SDE: H2-2: thm: Wp exponential contractivity} and Grönwall's lemma, we obtain
            \begin{align}
                \E^{\alpha}\seg*{\abs*{\~X_t}^2}&\leq \left( 2K^b_x +K^b_\LL+2K^\sigma_x + {2}\sqrt{2K^\sigma_x}\abs*{\beta^\alpha}_\infty\right)\int^t_0\E^{\alpha}\seg*{\abs*{\~X_s}^2} \d s + K^b_\LL\int^t_0 \w(\loi{X^{\loi\theta}_s},\loi{X^\loi{\theta'}_s})^2\d s\notag\\
                &\leq \left( 2K^b_x +K^b_\LL+2K^\sigma_x + {2}\sqrt{2K^\sigma_x}\abs*{\beta^\alpha}_\infty\right)\int^t_0\E^{\alpha}\seg*{\abs*{\~X_s}^2} \d s + K^b_\LL\int^t_0 e^{-2\~\eta s }\w(\loi{\theta},\loi{\theta'})^2\d s\notag\\
                &\leq \frac{K^b_\LL}{2\hat\eta}\w(\loi{\theta},\loi{\theta'})^2e^{ \left( 2K^b_x +K^b_\LL+2K^\sigma_x + {2}\sqrt{2K^\sigma_x}\abs*{\beta^\alpha}_\infty\right)t}=:C_t\w(\loi{\theta},\loi{\theta'})^2.\label{EBSDE: H2: borne moment 2}
            \end{align}

    Then, using \eqref{EBSDE: H2: borne moment 2}, \autoref{SDE: H2-2: thm: Wp exponential contractivity}, \autoref{SDE: H2-2: thm: Wp exponential estimate decoupled} and the fact that $e^{-\alpha T}\leq 1$ in \eqref{eq:delta:ubaralpha:theta}, we get  
        \begin{align*}
            \abs*{\bar u^\alpha(x',\loi\theta)-\bar u^\alpha(x',\loi{\theta'})}
            \leq& (Ce^{-\frac{\hat\eta\epsilon}{2} T} A_n+C_T)(1 + \abs*{x'}^{q+1/2}+\norm*{\theta}_{2q+2}^{q+1/2}+\norm*{\theta'}^{q+1/2}_{2q+2})\w(\loi\theta,\loi{\theta'})^{\epsilon/2}\\
            & +C_{\alpha}(Ce^{-\alpha T})^{n+1}{\left( 1+\abs*{x'}^{q+1} +\norm*{\theta}^{q+1}_{2q+2}+\norm*{\theta'}^{q+1}_{2q+2} \right)},
        \end{align*}
        which gives that \eqref{EBSDE: HRn} holds for all $n\geq 0$.
        Now, for a given $\alpha>0$ we set $T_{\alpha}$ large enough in order to have $Ce^{-\alpha T_{\alpha}}<1$ and $Ce^{-{\frac{\hat\eta\epsilon}{2}} T_{\alpha}}<1$ in \eqref{EBSDE: HRn}. Thus $(A_n)_n$ is a contractive sequence and hence admits a limit $A_{\infty}$ that depends on 
        $\alpha$. By taking $n \to +\infty$,  \eqref{EBSDE: HRn} gives us 
        \begin{equation*}
            \abs*{\bar u^\alpha(x',\loi\theta)-\bar u^\alpha(x',\loi{\theta'})}\leq A_\infty(1+ \abs*{x'}^{q+1/2}+\norm*{\theta}_{2q+2}^{q+1/2}+\norm*{\theta'}^{q+1/2}_{2q+2})\w(\loi\theta,\loi{\theta'})^{\epsilon/2}.
        \end{equation*}
        But, by doing the same computations as for proving \eqref{EBSDE: HRn}, we also get
          \begin{equation} 
            \abs*{\bar u^\alpha(x',\loi\theta)-\bar u^\alpha(x',\loi{\theta'})} \leq \tilde{A}_n(1+\abs*{x'}^{q+1/2}+\norm*{\theta}^{q+1/2}_{2q+2}+\norm*{\theta'}^{q+1/2}_{2q+2})\w(\loi\theta,\loi{\theta'})^{\epsilon/2},
            \tag{$HR'_n$}\label{EBSDE: HRn: p}
        \end{equation}        
        with 
         $(\tilde{A}_n)_n$ that satisfies the following induction relation: 
        \begin{equation}
        \label{suite rec An Bn}
         \tilde{A}_{n+1} =C_T +\tilde{A}_n Ce^{-{\frac{\hat\eta\epsilon}{2}} T}, 
         \quad n \in \mathbb{N}, \quad \tilde{A}_0=A_{\infty},
        \end{equation}
        where $C_T$ and $C$ do not depend on $n$ and $\alpha$, $C$ does not depend on $T$. Now we can set $T$ (independently of $\alpha$) such that $Ce^{-{\frac{\hat\eta\epsilon}{2}} T}<1$, which gives us that $(\tilde{A}_n)_n$ is a contractive sequence and hence admits a limit $\tilde{A}_{\infty}$ that does not depend on $\alpha$. 
        Thus, letting $n\longrightarrow \infty$ in the previous inequality gives the following estimate, uniformly in $\alpha$,
        \begin{equation}\label{eq temp rmq}
            \abs*{\bar u^\alpha(x',\loi\theta)-\bar u^\alpha(x',\loi{\theta'})}\leq \tilde{A}_\infty(1+ \abs*{x'}^{q+1/2}+\norm*{\theta}_{2q+2}^{q+1/2}+\norm*{\theta'}^{q+1/2}_{2q+2})\w(\loi\theta,\loi{\theta'})^{\epsilon/2}.
        \end{equation}

        
    
        Then, combining it with \eqref{EBSDE: estim equicont intermediaire 1} gives us, for all $x,x'\in \R^d$, $\theta,\theta' \in \L{2q+2}(\Omega,\R^d)$,
        $$
        \abs*{ \bar u^\alpha(x,\loi\theta)-\bar u ^\alpha(x',\loi{\theta'})}\leq C(1+\abs*{x}^{q+1/2} + \abs*{x'}^{q+1/2}+\norm*{\theta}^{q+1/2}_{2q+2}+\norm*{\theta'}^{q+1/2}_{2q+2})\left(\w(\loi\theta,\loi{\theta'})^{\epsilon/2} + \abs*{x-x'}^{\epsilon/2}\right),
        $$
        which is the wanted equicontinuous estimate. Then we just have to follow the end of the proof of Section \ref{subsection:strong:dissip:BSDE} in order to prove the existence of a solution to \eqref{EBSDE: Decoupled EBSDE}.
        \hfill\vspace{\baselineskip}
        
        \underline{\textbf{Uniqueness:}} 
        The proof of the uniqueness is the same as in Section \ref{subsection:strong:dissip:BSDE}, up to the $\W$-Wasserstein distance that is replaced by the $\w$ one.

    \end{proof} 
    \begin{rmq}\label{EBSDE: rmq global}
    If $f$ is globally $\epsilon$-Hölder, i.e. $q=0$ in \eqref{hyp-f}, then $\bar u$ is globally $\epsilon$-Hölder too. Namely,
    $$  \abs*{\bar u(x,\loi\theta)-\bar u(x',\loi{\theta'})}\leq C\left(\abs*{x-x'}^\epsilon+\w(\loi\theta,\loi{\theta'})^\epsilon \right).
    $$ 
    Indeed, for the $x$ part, it follows from \eqref{eq:ualphan:delta} and \autoref{SDE: H2-2: rmq global lip}. For the distribution part, we instead prove, by using the same computation and the new globally $\epsilon-$Hölder estimate w.r.t $x$, the following inductive relation
    $$ \abs*{\bar u^\alpha(x',\loi\theta)-\bar u^\alpha(x',\loi{\theta'})}\leq A_n\w(\loi\theta,\loi{\theta'})^{\epsilon}+C_{\alpha}(Ce^{-\alpha T})^n(1+\abs*{x'}+\norm*{\theta}_{2}+\norm*{\theta'}_{2}).
    $$ 
    Hence, by following computations, up to Cauchy-Schwarz's inequalities, leading to \eqref{eq:delta:ubaralpha:theta}, we obtain 
\begin{align*}
    \abs*{\bar u^\alpha(x',\loi\theta)-\bar u^\alpha(x',\loi{\theta'})}&\leq  Ce^{-\alpha T}\E^\alpha\seg*{\abs*{X^{x',\loi{\theta}}_T-X^{x',\loi{\theta'}}_T}}^{\epsilon} + Ce^{-\alpha T}A_n\w(\loi{X^\loi\theta_T},\loi{X^\loi{\theta'}_T})^{\epsilon}\\
    &\quad + C_{\alpha}(Ce^{-\alpha T})^{n+1}(1+\abs*{x'}+\norm*{\theta}_{2}+\norm*{\theta'}_{2})\\
    &\quad + C\int^T_0e^{-\alpha s}\left(\E^\alpha\seg*{\abs*{X^{x',\loi{\theta}}_s-X^{x',\loi{\theta'}}_s}}^{\epsilon}+\w(\loi{X^\loi\theta_s},\loi{X^\loi{\theta'}_s})^\epsilon \right)\d s 
\end{align*}
and we bound $\E^\alpha\seg*{\abs*{X^{x',\loi{\theta}}-X^{x',\loi{\theta'}}}}^{\epsilon}\leq \E^\alpha\seg*{\abs*{X^{x',\loi{\theta}}-X^{x',\loi{\theta'}}}^2}^{\epsilon/2} $ as before. It is then sufficient to apply the same arguments to reach a conclusion.
\end{rmq}

\section{Long-time behaviour of McKean-Vlasov BSDEs}\label{LTB}
The main goal of this section is to establish long-time behaviour of the solution of the (decoupled) finite horizon BSDE towards the solution to the (decoupled) McKean-Vlasov infinite horizon EBSDE. Let us denote $(Y^{T,x,\loi\theta},Z^{T,x,\loi\theta})$ the solution of the decoupled McKean-Vlasov finite horizon BSDE, $(Y^{T,\loi\theta},Z^{T,\loi\theta})$ the solution of the coupled finite horizon BSDE and $(Y^{x,\loi\theta},Z^{x,\loi\theta},\lambda)$ the solution to the decoupled EBSDE. Namely, we set:
\begin{align}&Y^{T,x,\loi{\theta}}_t = g(X^{x,\loi\theta}_T,\loi{X^{\loi\theta}_T}) +\int^T_t f(X^{x,\loi\theta}_s,\loi{X^{\loi\theta}_s},Z^{T,x,\loi\theta}_s)\d s -\int^T_t Z^{T,x,\loi{\theta}}_s\d W_s, &\forall t\in\seg*{0,T}, \label{LTB: Decoupled BSDE}\tag{Decoupled BSDE}\\
    &Y^{x,\loi{\theta}}_t = Y^{x,\loi\theta}_T +\int^T_t \left(f(X^{x,\loi\theta}_s,\loi{X^{\loi\theta}_s},Z^{x,\loi{\theta}}_s)-\lambda\right)\d s -\int^T_t Z^{x,\loi\theta}_s\d W_s, &\forall 0\leq t\leq T<\infty ,
\label{LTB: Decoupled EBSDE}\tag{Decoupled EBSDE}
\end{align}
where we recall that $X^{x,\loi\theta}$ and $X^\loi\theta$ are, respectively, the solutions of \eqref{SDE: Decoupled MKV SDE} and \eqref{SDE: MKV SDE} and for $g:\R^d\times\Pcal_{2q+2}\rightarrow\R^d$  being such that
\begin{hyp*}[$\H_g$]\label{LTB: hyp g}
There exist $q \geq 0$, $c>0$ and $\epsilon\in (0,1]$ such that, for all $x,x' \in \R^d$, $\theta, \theta' \in \L{2q+2}(\Omega,\R^d)$,
\begin{enumerate}
            \item $\abs*{g(x,\loi{\theta})}\leq c(1+\abs*{x}^{q+1}+\norm*{\theta}_{2q+2}^{q+1}),$
            \item $\abs*{g(x,\loi{\theta})-g(x',\loi{\theta'})}\leq c(1+\abs*{x}^q+\abs*{x'}^q+\norm*{\theta}_{2q+2}^q+\norm*{\theta'}_{2q+2}^q)\left(   \abs*{x-x'}^\epsilon+\w(\loi{\theta},\loi{\theta'})^\epsilon  \right)$.
        \end{enumerate}
\end{hyp*}
Without loss of generality, we can take the same $q$ in (\nameref{LTB: hyp g}) and (\nameref{EBSD: H0}). We will assume that  (\nameref{EBSDE: H1}) or that (\nameref{EBSDE: H2}) hold: Thus, we can apply \autoref{EBSDE: Existence and uniqueness} in order to get the existence of a triplet $(u,\zeta,\lambda)$ such that $(u(X_t^{x,\loi \theta},\loi{X_t^{\loi \theta}}),\zeta(X_t^{x,\loi \theta},\loi{X_t^{\loi \theta}}),\lambda)$ is a solution to \eqref{LTB: Decoupled EBSDE} and $u$ has growth and smoothness properties specified in \autoref{EBSDE: Existence and uniqueness}. In the remaining of this section, we will always consider this solution of \eqref{LTB: Decoupled EBSDE}.
We will now establish three different long-time behaviour for \eqref{LTB: Decoupled BSDE}.
\begin{thm}[LTB 1]\label{LTB: LTB 1}Assume that (\nameref{EBSDE: H1})-(\nameref{LTB: hyp g}) or that (\nameref{EBSDE: H2})-(\nameref{LTB: hyp g}) hold. Then,
    there exists $C>0$ such that, for all $x\in\R^d,\theta\in\L{2q+2}(\Omega,\R^d)$ and all $T>0$,\begin{equation}\label{LTB: eq: LTB 1}
        \abs*{\frac{Y^{T,x,\loi\theta}_0}{T}-\lambda}\leq  C\frac{1+\abs*{x}^{q+1}+\norm*{\theta}_{2q+2}^{q+1}}{T}.
    \end{equation}
\end{thm}
\begin{proof}
    Let $x\in\R^d$, $\theta\in\L{2q+2}(\Omega,\R^d)$ and $T>0$. We have
    \begin{equation*}
        \abs*{\frac{Y^{T,x,{\loi\theta}}_0}{T}-\lambda}\leq \abs*{\frac{Y^{T,x,\loi\theta}_0-Y^{x,\loi\theta}_0 -\lambda T}{T}}+\abs*{\frac{Y^{x,\loi\theta}_0}{T}}.
    \end{equation*}
    First, by the growth property of $u$, $\abs*{Y^{x,\loi\theta}_0}=\abs*{u(x,\loi{\theta})}\leq C(1+\abs*{x}^{q+1}+\norm*{\theta}^{q+1}_{2q+2})$. Furthermore, we have, for $\~Z=Z^{T,x,\loi\theta}-Z^{x,\loi\theta}$,
    \begin{align*}
        Y^{T,x,\loi\theta}_0-Y^{x,\loi\theta}_0 -\lambda T&=g(X^{x,\loi\theta}_T,\loi{X^{\loi\theta}_T}) -u(X^{x,\loi\theta}_T,\loi{X^{\loi\theta}_T})+\int^T_0 f(X^{x,\loi\theta}_s,\loi{X^{\loi\theta}_s},Z^{T,x,\loi\theta}_s)-f(X^{x,\loi\theta}_s,\loi{X^{\loi\theta}_s},Z^{x,\loi\theta}_s) \d s \\&\quad-\int^T_0 \~Z_s \d W_s\\
        &= g(X^{x,\loi\theta}_T,\loi{X^{\loi\theta}_T}) -u(X^{x,\loi\theta}_T,\loi{X^{\loi\theta}_T})-\int^T_0 \~Z_s \left(  \d W_s -\beta_s^T\d s  \right),
    \end{align*}
    where \begin{equation*}
        \beta^T_s=\frac{f(X^{x,\loi\theta}_s,\loi{X^{\loi\theta}_s},Z^{T,x,\loi\theta}_s)-f(X^{x,\loi\theta}_s,\loi{X^{\loi\theta}_s},Z^{x,\loi\theta}_s)}{\abs*{\~Z_s}^2}\~Z_s^\top\1{Z^{T,x,\loi\theta}_s\neq Z^{x,\loi\theta}_s}.
    \end{equation*}
By assumptions on $f$,  $\beta^T$ is bounded, thus by Girsanov's theorem, there exists $\Q_T$ s.t $W^T_t =  W_t -\int^t_0 \beta_s^T\d s  $ is a $\Q_T$-Brownian motion. Then, we get from the polynomial growth of $u$ and the assumption on $g$
\begin{align}
    \abs*{Y^{T,x,\loi\theta}_0-Y^{x,\loi\theta}_0 -\lambda T}&=\abs*{\E^{\Q_T}\seg*{g(X^{x,\loi\theta}_T,\loi{X^{\loi\theta}_T})-u(X^{x,\loi\theta}_T,\loi{X^{\loi\theta}_T})}}\label{LTB: proof: P_t}\\
    &\leq \E^{\Q_T}\seg*{\abs*{g(X^{x,\loi\theta}_T,\loi{X^{\loi\theta}_T})}}+\E^{\Q_T}\seg*{\abs*{u(X^{x,\loi\theta}_T,\loi{X^{\loi\theta}_T})}}\notag\\
    &\leq C(1+\E^{\Q_T}\seg*{\abs*{X^{x,\loi\theta}_T}^{q+1}}+\norm*{X^{\loi\theta}_T}_{2q+2}^{q+1}).\notag
\end{align}
Using Jensen's inequality and  \autoref{SDE: H1: prop: Estim unif T Q} or \autoref{SDE: H2-2: prop: estim moments} (depending on the assumptions set), we get
\begin{equation}\label{LTB: eq: proof LTB 1}
        \abs*{Y^{T,x,\loi\theta}_0-Y^{x,\loi\theta}_0 -\lambda T}\leq C(1+\abs*{x}^{q+1}+\norm*{\theta}_{2q+2}^{q+1}),
\end{equation}
which gives the wanted result.
\end{proof}

\begin{rmq}
    This previous results would still holds if we drop the Markovian assumption on the terminal condition as long as it satisfies the suitable growth condition. Namely, it is enough to consider a terminal condition $\xi^T$ such that $\abs*{\xi^T}\leq C(1+\abs*{X^{x,\loi\theta}_T}^{q+1}+\norm*{{X^\loi\theta_T}}_{2q+2}^{q+1})$, see for instance \cite[Theorem 21]{Hu-Lemmonier} in the non-McKean-Vlasov framework. Moreover, it should also be possible to deal with an unbounded path-dependent terminal condition by adapting the proof of \cite[Theorem 4.1]{Hu-Madec-Richou}.
\end{rmq}
\begin{thm}[LTB 2]\label{LTB: LTB 2}Assume that (\nameref{EBSDE: H1})-(\nameref{LTB: hyp g}) or that (\nameref{EBSDE: H2})-(\nameref{LTB: hyp g}) hold. Let us also enrich (\nameref{EBSDE: H1}) by assuming that $\sigma$ is uniformly elliptic. 
 Then, there exists $\ell \in \R$, $C>0$ and $\bm\eta>0$ such that, for all $x\in \R^d$, $\theta \in \L{2q+2}(\Omega;\R^d)$, $T >0$, we have
\begin{equation}\label{LTB: eq: LTB 2 H1}
            \abs*{Y^{T,x,\loi{\theta}}_0-\lambda T-Y^{x,\loi\theta}_0-\ell}\leq C\left( 1+\abs*{x}^{q+1}+\norm*{\theta}_{2q+2}^{q+1} \right)e^{-{\bm\eta}T}.
        \end{equation}
        
\end{thm}
\begin{proof}
    The idea of the proof follows the same as the one in \cite[Theorem 4.4]{Hu-Madec-Richou}. For all $T>0$, we introduce $w^T : [0,T] \times \R^d \times \Pcal_{2q+2}\rightarrow\R$, defined for all $t\in [0,T]$, $x\in\R^d,\, \theta\in\L{2q+2}(\Omega;\R^d)$ as $w^T(t,x,\loi\theta):=u^T(t,x,\loi\theta) - \lambda (T-t) -u(x,\loi\theta)$. 
     The equicontinuous property of $u$ allows us to obtain the following equicontinuous property on $w^T$ for all $T>0$.
     
     \begin{lem}\label{LTB: lemme} Assume that assumptions of \autoref{LTB: LTB 2} hold. There exists $C>0$ and $\bm\eta>0$ such that, for all $T>0$, $x,\,x'\in\R^d,$ $\theta,\,\theta'\in\L{2q+2}(\Omega;\R^d),$
\begin{enumerate}
    \item for (\nameref{EBSDE: H1}),
\begin{align}\nonumber
   {\abs*{w^T(0,x,\loi\theta)-w^T(0,x',\loi{\theta'})}}\leq& C(1+\abs*{x}^{q+1/2} + \abs*{x'}^{q+1/2}+\norm*{\theta}_{2q+2}^{q+1/2} + \norm*{\theta'}_{2q+2}^{q+1/2})\abs*{x-x'}^{\epsilon/2} e^{-\bm\eta T}\\
   & + C(1+\abs*{x}^{q} + \abs*{x'}^{q}+\norm*{\theta}_{2q+2}^{q} + \norm*{\theta'}_{2q+2}^{q})\W(\loi\theta,\loi{\theta'})^\epsilon,
   \label{LTB: estim wT H1}
\end{align}
    \item for (\nameref{EBSDE: H2}),
   \begin{align}
   \abs*{w^T(0,x,\loi\theta)-w^T(0,x',\loi{\theta'})}&\leq C(1+\abs*{x}^{q+1/2} + \abs*{x'}^{q+1/2}+\norm*{\theta}_{2q+2}^{q+1/2} + \norm*{\theta'}_{2q+2}^{q+1/2})\left(\abs*{x-x'}^{\epsilon/2}e^{-\bm\eta  T}+\w(\loi\theta,\loi{\theta'})^{\epsilon/2}\right).\label{LTB: estim wT H2}
\end{align}
\end{enumerate}
     \end{lem}
    
     \begin{proof}[Proof of \autoref{LTB: lemme}]
         1. We start by proving the result under (\nameref{EBSDE: H1}).
         
         1.1. First, let us fix $\theta$ and look only for the estimate w.r.t $x$, that is to say $\abs*{w^T(0,x,\loi\theta)-w^T(0,x',\loi\theta)}$.
         Let $\beta^T$ by given, for all $s\in\R_+,\, x\in\R^d$, by 
         $$ \beta^T(s,x)=\frac{f(x,\loi{X^{\loi\theta}_s},\zeta^T(x,\loi{X^\loi\theta_s}))-f(x,\loi{X^{\loi\theta}_s},\zeta(x,\loi{X^\loi\theta_s}))}{\abs*{(\zeta^T-\zeta)(x,\loi{X^\loi\theta_s})}^2}((\zeta^T-\zeta)(x,\loi{X^\loi\theta_s}))^\top\1{(\zeta^T-\zeta)(x,\loi{X^\loi\theta_s})\neq 0}.
         $$ Due to the Lipschitz property of $f$ w.r.t $z$, 
         $\beta^T$ is bounded by $K^f_z$. Consequently, let us set $\Q^T$ the Girsanov's probability associated to $\beta^T$ under which $\d W^T:=\d W_t -\beta^T(s,x)\d s$ is a Brownian motion. Then, 
          $w^T(0,x,\loi\theta)$ writes under $\Q^T$
         $$w^T(0,x,\loi\theta)=\E^{\Q^T}\seg*{g(X^{x,\loi\theta}_T,\loi{X^\loi\theta_T})-u(X^{x,\loi\theta}_T,\loi{X^\loi\theta_T})}.
         $$
         By using \autoref{lem: approx beta}, we denote $(\beta^T_n)$ a sequence of Lipschitz functions, uniformly bounded with respect to $n$, approximating $\beta^T_{\theta}$, $(\Q^T_n)$ the associated probabilities coming from Girsanov's Theorem, and, for all $n \geq 0$,
         $$w^{T,n}(0,x,\loi\theta):=\E^{\Q^T_n}\seg*{g(X^{x,\loi\theta}_T,\loi{X^\loi\theta_T})-u(X^{x,\loi\theta}_T,\loi{X^\loi\theta_T})}.$$ 
         Then, for all $n\geq 0$, we can introduce $\Pcal^n$ the semi-group associated to the following standard SDE
         $$\left\{\begin{array}{ll}
             \d X^{\beta^T_{n},x,\loi\theta}_t=\left( b(X^{\beta^T_{n},x,\loi\theta}_t,\loi{X^\loi\theta_t})+\beta^{T}_n (t,X^{\beta^T_{n},x,\loi\theta}_t)\sigma(X^{\beta^T_{n},x,\loi\theta}_t,\loi{X^\loi\theta_t})\right)\d t + \sigma(X^{\beta^T_{n},x,\loi\theta}_t,\loi{X^\loi\theta_t})\d W_t, & \quad t \in [0,T],  \\
             X^{\beta^T_{n},x,\loi\theta}_0=x.  & 
         \end{array}\right.
         $$ 
        
         Hence, we can write $w^{T,n}(0,x,\loi\theta)=\mathcal{P}^n_T\seg*{g-u}(x,\loi\theta)$ and we obtain
         \begin{align}\label{LTB: formulation semigroup x}
             \abs*{w^{T,n}(0,x,\loi\theta)-w^{T,n}(0,x',\loi\theta)}=\abs*{\Pcal^n_T\seg*{g-u}(x,\loi \theta)-\Pcal^n_T\seg*{g-u}(x',\loi \theta)}.
         \end{align}
         Let us prove that 
         $$\abs*{\Pcal^n_T\seg*{g-u}(x,\loi \theta)-\Pcal^n_T\seg*{g-u}(x',\loi \theta)}\leq C(1+\abs*{x}^{q+1/2} + \abs*{x'}^{q+1/2}+\norm*{\theta}_{2q+2}^{q+1/2})e^{-\gamma\epsilon T/2}\abs*{x-x'}^{\epsilon/2}.
         $$
         Let $(\UU,\UU')$ be an optimal coupling of $(X^{\beta^T_{n},x,\loi\theta},X^{\beta^T_{n},x',\loi\theta})$ w.r.t the $\w$-Wasserstein distance. By using the locally Hölder property of $g$ and $u$, Cauchy-Schwarz's inequality and \autoref{SDE: H1: prop: Estim unif T Q}, we obtain 
         \begin{align}
         &\abs*{\Pcal^n_T\seg*{g-u}(x,\loi\theta)-\Pcal^n_T\seg*{g-u}(x',\loi\theta)}\notag \\ \notag
             =&\abs*{\E\seg*{g(X^{\beta^T_{n},x,\loi\theta}_T,\loi{X^\loi\theta_T})-u(X^{\beta^T_{n},x,\loi\theta}_T,\loi{X^\loi\theta_T})}-\E\seg*{g(X^{\beta^T_{n},x',\loi\theta}_T,\loi{X^\loi\theta_T})-u(X^{\beta^T_{n},x',\loi\theta}_T,\loi{X^\loi\theta_T})}}\notag\\
             =&\abs*{\E\seg*{g(\UU_T,\loi{X^\loi\theta_T})-u(\UU_T,\loi{X^\loi\theta_T})}-\E\seg*{g(\UU'_T,\loi{X^\loi\theta_T})-u(\UU'_T,\loi{X^\loi\theta_T})}}\notag\\
             \leq& \E\seg*{\abs*{g(\UU_T,\loi{X^\loi\theta_T})-g(\UU'_T,\loi{X^\loi\theta_T})}}+\E\seg*{\abs*{u(\UU_T,\loi{X^\loi\theta_T})-u(\UU'_T,\loi{X^\loi\theta_T})}}\notag\\
             \leq& C\E\seg*{ (1+ \abs*{\UU_T}^q + \abs*{\UU'_T}^q+\norm*{X^\loi\theta_T}_{2q+2}^q)\abs*{\UU_T-\UU'_T}^\epsilon  }\label{LTB: eq: CS}\\
             \leq&  (1+ \E\seg*{\abs*{\UU_T}^{2q+1}}^{1/2} + \E\seg*{\abs*{\UU'_T}^{2q+1}}^{1/2}+\norm*{X^\loi\theta_T}_{2q+2}^{q+1/2})\E\seg*{\abs*{\UU_T-\UU'_T}}^{\epsilon/2}\notag\\
             \leq&(1+ \E^{\Q_n^T}\seg*{\abs*{X^{x,\loi\theta}_T}^{2q+1}}^{1/2} + \E^{\Q_n^T}\seg*{\abs*{X^{x',\loi\theta}_T}^{2q+1}}^{1/2}+\norm*{X^\loi\theta_T}_{2q+2}^{q+1/2})\w(\loi{X^{\beta^T_{n},x,\loi\theta}_T},\loi{X^{\beta^T_{n},x',\loi\theta}_T})^{\epsilon/2}\notag\\
             \leq& C(1+\abs*{x}^{q+1/2} + \abs*{x'}^{q+1/2}+\norm*{\theta}_{2q+2}^{q+1/2})\w(\loi{X^{\beta^T_{n},x,\loi\theta}_T},\loi{X^{\beta^T_{n},x',\loi\theta}_T})^{\epsilon/2}.\notag
         \end{align}
        Now, we want to use \autoref{Appendix: thm general} with $\b(t,\cdot)=\b'(t,\cdot)=b(\cdot,\loi{X^\loi\theta_t})+\sigma(\cdot,\loi{X^\loi\theta_t})\beta^T_n(t,\cdot)$. Firstly, we can check that $\b$ satisfies \eqref{Appendix: hyp: dissipatif} uniformly in $n$: for all $x,x'\in\R^d$, $ t\in\R_+$, $R'>0$,
         \begin{align*}
             &\braket*{x-x'}{b(x,\loi{X^{\loi\theta}_t})-b(x',\loi{X^\loi\theta_t})} + \braket*{x-x'}{\sigma(x,\loi{X^{\loi\theta}_t})\beta^T_n(t,x)-\sigma(x',\loi{X^\loi\theta_t})\beta^T_n (t,x)}\\
              \leq & -\eta \abs*{x-x'}^2 +2\norm*{\sigma}_\infty\abs*{\beta^T_n}_\infty\abs*{x-x'} \\
             \leq & -\eta\abs*{x-x'}^2\1{\abs*{x-x'}> 2}+2\norm*{\sigma}_\infty K^f_z\abs*{x-x'}(\1{\abs*{x-x'}> 2} + \1{\abs*{x-x'}\leq 2})\\
             \leq & -(\eta-\frac{\norm*{\sigma}_\infty K^f_z}{R'})\abs*{x-x'}^2\1{\abs*{x-x'}>R'}+2\norm*{\sigma}_\infty K^f_z\1{\abs*{x-x'}\leq R}.
         \end{align*}
         Hence, $\b$ satisfies \autoref{Appendix: hyp} with $\bm\eta=\eta-\frac{\norm*{\sigma}_\infty K^f_z}{R'}>K_x^{\sigma}$ considering $R'$ to be large enough. Consequently, we can apply \autoref{Appendix: thm general} with $\EE_t=\bm{c}=0$ to obtain 
         \begin{equation}\label{LTB: Wassertein Q H1}
             \w\left( \loi{X^{\beta^T_{n},x,\loi\theta}_t},\loi{X^{\beta^T_{n},x',\loi{\theta}}_t}  \right) \leq Ce^{-\hat{\eta}t}\abs*{x-x'}
         \end{equation}
         and we get, uniformly in $n$,
         \begin{equation}
         \label{regularity-wTn}
         \abs*{w^{T,n}(0,x,\loi\theta)-w^{T,n}(0,x',\loi\theta)}\leq C(1+\abs*{x}^{q+1/2} + \abs*{x'}^{q+1/2}+\norm*{\theta}_{2q+2}^{q+1/2})\abs*{x-x'}^{\epsilon/2}e^{-\hat\eta\epsilon T/2}.
         \end{equation}
        Now we want to take $n \rightarrow + \infty$ in order to conclude this first step. Let us remark that $(w^T(t,X^{x,\loi \theta}_t,\loi{X_t^{\loi \theta}}),\delta Z_t)_{t \in [0,T]}$,  where $\delta Z_t := Z^{T,x ,\loi \theta}_t - Z^{x,\loi \theta}_t$, is solution of the following BSDE
        $$w^{T}(t,X^{x,\loi \theta}_t,\loi{X_t^{\loi \theta}}) = g(X^{x,\loi \theta}_T,\loi{X_T^{\loi \theta}}) - u(X^{x,\loi \theta}_T,\loi{X_T^{\loi \theta}}) + \int_t^T \delta Z_s \beta^T (s,X^{x,\loi \theta}_s) \d s - \int_t^T \delta Z_s \d W_s, \quad t \in [0,T].$$
        Moreover, by the existence and uniqueness of Lipschitz driven BSDE result, see \cite{Pardoux-Peng}, there exists a progressively measurable process $\delta Z^n$ such that 
        $(w^{T,n}(t,X^{x,\loi \theta}_t,\loi{X_t^{\loi \theta}}),\delta Z^n_t)_{t \in [0,T]}$ is solution of the following BSDE
        $$w^{T,n}(t,X^{x,\loi \theta}_t,\loi{X_t^{\loi \theta}}) = g(X^{x,\loi \theta}_T,\loi{X_T^{\loi \theta}}) - u(X^{x,\loi \theta}_T,\loi{X_T^{\loi \theta}}) + \int_t^T \delta Z^n_s \beta_n^T (s,X^{x,\loi \theta}_s) \d s - \int_t^T \delta Z^n_s \d W_s, \quad t \in [0,T].$$
        Then a classical $\L{2}$-stability result for Lipschitz BSDEs gives us 
        $$\abs{w^T(0,x,\loi \theta)-w^{T,n}(0,x,\loi \theta)}^2\leq C \int_0^T \mathbb{E}[ \abs{\beta^T (s,X^{x,\loi \theta}_s)-\beta_n^T (s,X^{x,\loi \theta}_s)}^2 \d s$$
        and, applying Lebesgue's convergence theorem, we obtain
        $\lim{n \to + \infty} w^{T,n}(0,x,\loi \theta) = w^T(0,x,\loi \theta)$.
        Thus, \eqref{regularity-wTn} allows to get
         $$\abs*{w^{T}(0,x,\loi\theta)-w^{T}(0,x',\loi\theta)}\leq C(1+\abs*{x}^{q+1/2} + \abs*{x'}^{q+1/2}+\norm*{\theta}_{2q+2}^{q+1/2})\abs*{x-x'}^{\epsilon/2}e^{-\hat\eta\epsilon T/2},$$
         which is the wanted estimate.
         
 
         1.2. For the $\theta$ part, let us fix $x'\in\R^d$. We recall that, for all $\theta,\theta'\in\L{2q+2}(\Omega;\R^d)$,
         \begin{equation}
         \label{eq:decomposition:wT}
         w^T(0,x',\loi\theta)-w^T(0,x',\loi{\theta'})=u^T(0,x',\loi\theta)-u^T(0,x',\loi{\theta'}) - (u(x',\loi\theta)-u(x',\loi{\theta'})).
         \end{equation}
         Since $u$ is $\epsilon-$locally Hölder w.r.t $\loi\theta$, see \autoref{EBSDE: Existence and uniqueness}, it suffices to prove that $u^T$ is too. By introducing ${\beta'}^T$ given by 
         \begin{equation}\label{LTB: beta Z^T}
    {\beta'_s}^T =\frac{f(X^{x',\loi\theta}_s,\loi{X^{\loi\theta}_s},Z^{T,x',\loi\theta}_s)-f(X^{x',\loi{\theta}}_s,\loi{X^{\loi{\theta}}_s},Z^{T,x',\loi{\theta'}}_s)}{\abs*{Z^{T,x',\loi\theta}_s-Z^{T,x',\loi{\theta'}}_s}^2}\left( Z^{T,x',\loi\theta}_s-Z^{T,x',\loi{\theta'}}_s\right)^\top\1{Z^{T,x',\loi\theta}_s\neq Z^{T,x',\loi{\theta'}}_s},
         \end{equation}
         and ${\Q'}^T$ the associated probability coming from Girsanov's theorem,
         we have
         \begin{align*}
             u^T(0,x',\loi\theta)-u^T(0,x',\loi{\theta'})&=g(X^{x',\loi\theta}_T,\loi{X^\loi\theta_T})-g(X^{x',\loi{\theta'}}_T,\loi{X^\loi{\theta'}_T})\\
             &\quad+ \int^T_0 \left(f(X^{x',\loi\theta}_s,\loi{X^{\loi\theta}_s},Z^{T,x',\loi\theta}_s)-f(X^{x',\loi{\theta'}}_s,\loi{X^{\loi{\theta'}}_s},Z^{T,x',\loi{\theta'}}_s)\right)\d s\\
             &\quad-\int^T_0 \left(Z^{T,x',\loi\theta}_s-Z^{T,x',\loi{\theta'}}_s\right)\d W_s \\
             &=g(X^{x',\loi\theta}_T,\loi{X^\loi\theta_T})-g(X^{x',\loi{\theta'}}_T,\loi{X^\loi{\theta'}_T})\\
             &\quad + \int^T_0 \left(f(X^{x',\loi\theta}_s,\loi{X^{\loi\theta}_s},Z^{T,x',\loi\theta}_s)-f(X^{x',\loi{\theta'}}_s,\loi{X^{\loi{\theta'}}_s},Z^{T,x',\loi{\theta}}_s)\right)\d s \\
             &\quad -\int^T_0 \left(Z^{T,x',\loi\theta}_s-Z^{T,x',\loi{\theta'}}_s\right)(\d W_s-{\beta'_s}^T \d s)\\
             &=\E^{{\Q'}^T}\seg*{g(X^{x',\loi\theta}_T,\loi{X^\loi\theta_T})-g(X^{x',\loi{\theta'}}_T,\loi{X^\loi{\theta'}_T})}\\
             &\quad+\E^{{\Q'}^T}\seg*{\int^T_0 \left(f(X^{x',\loi\theta}_s,\loi{X^{\loi\theta}_s},Z^{T,x',\loi\theta}_s)-f(X^{x',\loi{\theta'}}_s,\loi{X^{\loi{\theta'}}_s},Z^{T,x',\loi{\theta}}_s)\right)\d s}.
             \end{align*}
             Then, by using the locally Hölder property of $g$ and $f$, Cauchy-Schwarz's inequality, \autoref{SDE: H1: prop: Estim unif T Q}, \autoref{SDE: H1: thm: Wp exponential contractivity} and \autoref{SDE: H1: cor: Wp exponential contractivity}, we obtain 
             \begin{align*}
             &\abs*{u^T(0,x',\loi\theta)-u^T(0,x',\loi{\theta'})}\\
             \leq& \E^{{\Q'}^T}\seg*{\abs*{g(X^{x',\loi\theta}_T,\loi{X^\loi\theta_T})-g(X^{x',\loi{\theta'}}_T,\loi{X^\loi{\theta'}_T})}}\\
             \quad&+\int^T_0\E^{{\Q'}^T}\seg*{\abs*{f(X^{x',\loi\theta}_s,\loi{X^{\loi\theta}_s},Z^{T,x',\loi\theta}_s)-f(X^{x',\loi{\theta'}}_s,\loi{X^{\loi{\theta'}}_s},Z^{T,x',\loi{\theta}}_s)}}\d s\\
             \leq& \E^{{\Q'}^T}\seg*{C(1+\abs*{X^{x',\loi\theta}_T}^q+\abs*{X^{x',\loi{\theta'}}_T}^q+\norm*{X^{\loi\theta}_T}^q_{2q+2}+\norm*{X^{\loi{\theta'}}_T}^q_{2q+2})\left(\abs*{X^{x',\loi{\theta}}_T-X^{x',\loi{\theta'}}_T}^{\epsilon}+\W(\loi{X^\loi\theta_T},\loi{X^{\loi{\theta'}}_T})^{\epsilon}\right)}\\
             \quad&+\int^T_0 \E^{{\Q'}^T}\seg*{C(1+\abs*{X^{x',\loi\theta}_s}^q+\abs*{X^{x',\loi{\theta'}}_s}^q+\norm*{X^{\loi\theta}_s}^q_{2q+2}+\norm*{X^{\loi{\theta'}}_s}^q_{2q+2})\left(\abs*{X^{x',\loi{\theta}}_s-X^{x',\loi{\theta'}}_s}^{\epsilon}+\W(\loi{X^\loi\theta_s},\loi{X^{\loi{\theta'}}_s})^{\epsilon}\right)}\\
             \leq&  C(1+\abs*{x'}^{q}+\norm*{\theta}_{2q+2}^q+\norm*{\theta'}_{2q+2}^q)\left(\E^{{\Q'}^T}\seg*{\abs*{X^{x',\loi{\theta}}_T-X^{x',\loi{\theta'}}_T}^2}^{\epsilon/2}+\W(\loi{X^\loi\theta_T},\loi{X^{\loi{\theta'}}_T})^{\epsilon}\right)\\
             \quad&+C(1+\abs*{x'}^{q}+\norm*{\theta}_{2q+2}^q+\norm*{\theta'}_{2q+2}^q)\int^T_0 \left(\E^{{\Q'}^T}\seg*{\abs*{X^{x',\loi{\theta}}_s-X^{x',\loi{\theta'}}_s}^2}^{\epsilon/2}+\W(\loi{X^\loi\theta_s},\loi{X^{\loi{\theta'}}_s})^{\epsilon}\right)\d s\\
             \leq& C(1+\abs*{x'}^{q}+\norm*{\theta}_{2q+2}^q+\norm*{\theta'}_{2q+2}^q)(e^{-\gamma \epsilon T}+e^{-\Lambda \epsilon T})\W(\loi\theta,\loi{\theta'})\\
             \quad& +C(1+\abs*{x'}^{q}+\norm*{\theta}_{2q+2}^q+\norm*{\theta'}_{2q+2}^q)\int^T_0 (e^{-\gamma \epsilon s}+e^{-\Lambda \epsilon s})\W(\loi\theta,\loi{\theta'})\d s\\
             \leq &C(1+\abs*{x'}^{q}+\norm*{\theta}_{2q+2}^q+\norm*{\theta'}_{2q+2}^q)\W(\loi{\theta},\loi{\theta'})^\epsilon,
         \end{align*}
         which concludes the proof of \autoref{LTB: lemme} under (\nameref{EBSDE: H1}).
         
         2. Under (\nameref{EBSDE: H2}), the local regularity of $w^T$ w.r.t. $x$, that is to say \eqref{LTB: estim wT H2} when $\theta=\theta'$, follows by applying \eqref{SDE: H2-2: Coupling estimate decoupled Q} to \eqref{LTB: formulation semigroup x}. It remains to prove the local regularity of $w^T$ w.r.t $\loi\theta$.
         Using once again the decomposition \eqref{eq:decomposition:wT} and since $u$ is $\frac{\epsilon}{2}-$locally Hölder w.r.t. $\loi\theta$, see \autoref{EBSDE: Existence and uniqueness}, it suffices to prove that $u^T$ is too.
         To that end, we will prove the result using a recursive argument. Let us consider $\~T>0$ whose value will be precised later and $N:=\lfloor {T}/{\~T} \rfloor$. By denoting $(Y^{T,s,x,\loi\theta},Z^{T,s,x,\loi\theta})$ the solution to the finite horizon BSDE \eqref{LTB: Decoupled BSDE} where $(X^{x,\loi\theta},\loi{X^\loi\theta})$ is replaced by $(X^{s,x,\loi\theta},\loi{X^{s,\loi\theta}})$ We have, for all $x \in \R^d$, $\theta \in  \L{2q+2}(\Omega;\R^d)$,
         \begin{align*}
         u^T\left(N\~T,x,\loi\theta\right)& =g(X^{{N\~T},x,\loi\theta}_{T},\loi{X^{{N\~T},\loi\theta}_{T}}) +\int^{T}_{N\~T} f(X^{{N\~T},x,\loi\theta}_s,\loi{X^{{N\~T},\loi\theta}_s},Z^{T,N\~T,x,\loi\theta}_s)\d s -\int^{T}_{N\~T} Z^{T,N\~T,x,\loi\theta}_s\d W_s.
         \end{align*}
         Let us show that $u^T(N\~T,x,\cdot)$ satisfies a locally Hölder estimate: we are looking for $A_0\geq0$ that does not depend on $N$ and $T$ such that
         \begin{equation*}
             \abs*{u^T\left(N\~T,x,\loi\theta\right)-u^T\left(N\~T,x,\loi{\theta'}\right)}\leq A_0(1+\abs*{x}^q+\norm*{\theta}^q_{2q+2}+\norm*{\theta'}^q_{2q+2})\w(\loi\theta,\loi{\theta'})^\epsilon.
         \end{equation*}
         
         First, by a Girsanov's argument with $\Q^{N\~T}$ the probability associated to $\beta^{N\~T}_s$ given by
         $$ {\beta}^{N\~T} _s=\frac{f(X^{N\~T,x,\loi\theta}_s,\loi{X^{N\~T,\loi\theta}_s},Z^{T,N\~T,x,\loi\theta}_s)-f(X^{N\~T,x,\loi{\theta}}_s,\loi{X^{N\~T,\loi{\theta}}_s},Z^{T,N\~T,x,\loi{\theta'}}_s)}{\abs*{Z^{T,x,\loi\theta}_s-Z^{T,N\~T,x,\loi{\theta'}}_s}^2}\left( Z^{T,N\~T,x,\loi\theta}_s-Z^{T,N\~T,x,\loi{\theta'}}_s\right)^\top\1{Z^{T,N\~T,x,\loi\theta}_s\neq Z^{T,N\~T,x,\loi{\theta'}}_s},
         $$
         and by using the locally $\epsilon$-Hölder property of $g$ and $f$, \eqref{EBSDE: H2: borne moment 2}, \autoref{SDE: H2-2: thm: Wp exponential estimate decoupled} on $\seg*{N\~T,T}$ and the facts that $T-N\~T\leq \~T$ and that $C_t$ in \eqref{EBSDE: H2: borne moment 2} is increasing in $t$, we obtain

         \begin{align}
             &\abs*{u^T\left(N\~T,x,\loi\theta\right)-u^T\left(N\~T,x,\loi{\theta'}\right)}\notag\\
             &\leq\E^{{\Q}^{N\~T}}\seg*{\abs*{g(X^{N\~T,x,\loi\theta}_{T},\loi{X^{N\~T,\loi\theta}_{T}})-g(X^{N\~T,x,\loi{\theta'}}_{T},\loi{X^{N\~T,\loi{\theta'}}_{T}})}}\\
             &\quad+\int^{T}_{N\~T}\E^{{\Q}^{N\~T}}\seg*{\abs*{f(X^{N\~T,x,\loi\theta}_s,\loi{X^{N\~T,\loi\theta}_s},Z^{T,N\~T,x,\loi\theta}_s)-f(X^{N\~T,x,\loi{\theta'}}_s,\loi{X^{N\~T,\loi{\theta'}}_s},Z^{T,N\~T,x,\loi{\theta}}_s)}}\d s\notag\\
             &\leq C(1+\abs*{x}^{q}+\norm*{\theta}_{2q+2}^q+\norm*{\theta'}_{2q+2}^q)\left(\E^{{\Q}^{N\~T}}\seg*{\abs*{X^{N\~T,x,\loi{\theta}}_{T}-X^{N\~T,x,\loi{\theta'}}_{T}}^2}^{\epsilon/2}+\w(\loi{X^{N\~T,\loi\theta}_{T}},\loi{X^{N\~T,\loi{\theta'}}_{T}})^{\epsilon}\right)\notag\\
             &\quad+C(1+\abs*{x}^{q}+\norm*{\theta}_{2q+2}^q+\norm*{\theta'}_{2q+2}^q)\int^{T}_{N\~T} \left(\E^{{\Q}^{N\~T}}\seg*{\abs*{X^{N\~T,x,\loi{\theta}}_s-X^{N\~T,x,\loi{\theta'}}_s}^2}^{\epsilon/2}+\w(\loi{X^{N\~T,\loi\theta}_s},\loi{X^{N\~T,\loi{\theta'}}_s})^{\epsilon}\right)\d s\label{LTB: recup induction}\\
             &\leq C(1+\abs*{x}^{q}+\norm*{\theta}_{2q+2}^q+\norm*{\theta'}_{2q+2}^q)\left( C_{T-N\~T} + Ce^{-\hat\eta\epsilon ({T-N\~T})}+C_{T-N\~T}(T-N\~T)+C\right)\w(\loi\theta,\loi{\theta'})^\epsilon\notag\\
             &\leq C(1+\abs*{x}^{q}+\norm*{\theta}_{2q+2}^q+\norm*{\theta'}_{2q+2}^q)\left( C_{\~T} + C \right)\w(\loi\theta,\loi{\theta'})^\epsilon\notag\\
             &=A_0(1+\abs*{x}^{q}+\norm*{\theta}_{2q+2}^q+\norm*{\theta'}_{2q+2}^q)\w(\loi\theta,\loi{\theta'})^\epsilon.\notag
         \end{align}

         Now we want to prove by induction the following estimate: for all $k\in\llbracket 0,N\rrbracket,$ 
         \begin{equation}
             \abs*{u^T\left((N-k)\~T,x,\loi\theta\right)-u^T\left((N-k)\~T,x,\loi{\theta'}\right)}\leq A_{k}(1+\abs*{x}^q+\norm*{\theta}^q_{2q+2}+\norm*{\theta'}^q_{2q+2})\w(\loi\theta,\loi{\theta'})^\epsilon,
         \end{equation}
         where $A_k$ satisfies, for $C,C_{\~T}>0$ that do not depends on $T$ nor $N$, 
         $$A_{k}=C_{\~T}+Ce^{-\hat\eta \epsilon \~T}A_{k-1},\quad k\in \llbracket 1, N\rrbracket, \quad A_0=C_{\~T}.
         $$
 If $N=0$, we already have the desired result. Let us now look at the case $N\geq1$. It has already been proven the initial case $k=0$ and we just have to look at the inductive step: for all $k\in\llbracket 1,N\rrbracket$, we have 
         \begin{align*}
            u^T\left((N-k)\~T,x,\loi\theta\right) &= u^T\left((N+1-k)\~T,X^{(N-k)\~T,x,\loi\theta}_{(N+1-k)\~T},\loi{X^{(N-k)\~T,\loi\theta}_{(N+1-k)\~T}}\right) +\int^{(N+1-k)\~T}_{(N-k)\~T} f(X^{(N-k)\~T,x,\loi\theta}_s,\loi{X^{(N-k)\~T,\loi\theta}_s},Z^{T,x,\loi\theta}_s)\d s \\&\quad -\int^{(N+1-k)\~T}_{(N-k)\~T} Z^{T,x,\loi\theta}_s\d W_s.  
         \end{align*}
        Then, replacing $g$ by $u^T((N+1-k)\~T,\cdot,\cdot)$ in previous computations, \eqref{LTB: recup induction} becomes
         \begin{align*}
             &\abs*{u^T\left((N-k)\~T,x,\loi\theta\right)-u^T\left((N-k)\~T,x,\loi{\theta'}\right)}\\
             &\leq C(1+\abs*{x}^{q}+\norm*{\theta}_{2q+2}^q+\norm*{\theta'}_{2q+2}^q)\E^{{\Q}^{(N-k)\~T}}\seg*{\abs*{X^{(N-k)\~T,x,\loi{\theta}}_{(N+1-k)\~T}-X^{(N-k)\~T,x,\loi{\theta'}}_{(N+1-k)\~T}}^2}^{\epsilon/2}\\
             &\quad +A_{k-1}\left(1+\abs*{x}^q+\norm*{\theta}^q_{2q+2}+\norm*{\theta'}^q_{2q+2}\right)\w(\loi{X^{(N-k)\~T,\loi\theta}_{(N+1-k)\~T}},\loi{X^{(N-k)\~T,\loi{\theta'}}_{(N+1-k)\~T}})^{\epsilon}\notag\\
             &\quad+C(1+\abs*{x}^{q}+\norm*{\theta}_{2q+2}^q+\norm*{\theta'}_{2q+2}^q)\\&\quad\times\int^{(N+1-k)\~T}_{(N-k)\~T} \left(\E^{{\Q}^{(N-k)\~T}}\seg*{\abs*{X^{(N-k)\~T,x,\loi{\theta}}_s-X^{(N-k)\~T,x,\loi{\theta'}}_s}^2}^{\epsilon/2}+\w(\loi{X^{(N-k)\~T,\loi\theta}_s},\loi{X^{(N-k)\~T,\loi{\theta'}}_s})^{\epsilon}\right)\d s\\
             &\leq C(1+\abs*{x}^{q}+\norm*{\theta}_{2q+2}^q+\norm*{\theta'}_{2q+2}^q)\left( C_{\~T} + Ce^{-\hat\eta\epsilon {\~T}}A_{k-1}\right)\w(\loi\theta,\loi{\theta'})^\epsilon\\
             &:=A_{k}(1+\abs*{x}^{q}+\norm*{\theta}_{2q+2}^q+\norm*{\theta'}_{2q+2}^q)\w(\loi\theta,\loi{\theta'})^\epsilon,
         \end{align*}
         where we used \eqref{SDE: H2: Estim unif T decoupled} and \eqref{SDE: H2: Estim unif T decoupled Q}. 
         Moreover, for $\~T$ such that $Ce^{-\hat\eta\epsilon\~T}\leq 1$, $(A_{n})_{n\geq 0}$ is a contractive sequence and hence has a unique fixed point $A_\infty$ which does not depend on $N$ nor $T$. Hence, by setting $A^* = \max(A_0,A_\infty)$ which does not depend on $N$ nor $T$, we have for all $k\in\llbracket 0,N\rrbracket$,
         $$\abs*{u^T\left((N-k)\~T,x,\loi\theta\right)-u^T\left((N-k)\~T,x,\loi{\theta'}\right)}\leq A^* (1+\abs*{x}^q+\norm*{\theta}^q_{2q+2}+\norm*{\theta'}^q_{2q+2})\w(\loi\theta,\loi{\theta'})^\epsilon.
         $$ 
        In particular, 
        $$\abs*{u^T\left(0,x,\loi\theta\right)-u^T\left(0,x,\loi{\theta'}\right)}\leq A^* (1+\abs*{x}^q+\norm*{\theta}^q_{2q+2}+\norm*{\theta'}^q_{2q+2})\w(\loi\theta,\loi{\theta'})^\epsilon,
         $$ 
which concludes the proof of \autoref{LTB: lemme}.
     \end{proof}
As soon as \autoref{LTB: lemme} holds, the proof of \autoref{LTB: LTB 2} is the same under (\nameref{EBSDE: H1}) or (\nameref{EBSDE: H2}), up to the differences implied by the equicontinuous terms and the polynomial growth {w.r.t $x$ and $\theta$} in \eqref{LTB: estim wT H1}-\eqref{LTB: estim wT H2}. Thus, we will only do the proof in the strong dissipative framework. 
Firstly, let us remark that setting $\loi\theta=\mu^*$, with $\mu^*$ the unique invariant measure of the solution to \eqref{SDE: MKV SDE},  transposes the MKV's framework to a standard one since $\loi{X_t^{\mu^*}}=\mu^*$ for all $t\geq 0$. Nevertheless, we cannot apply \cite[Theorem 22]{Hu-Lemmonier}  to get the result for $\theta\sim \mu^*$ since our assumptions do not fit perfectly the ones requested in \cite{Hu-Lemmonier}. 
Using \autoref{LTB: lemme} and a diagonal procedure, there exist $w$ and an increasing sequence $(T_i)_i\to+ \infty$ such that for all $(x,\loi\theta)\in \mathbb{D}$ a dense and countable subset of $\R^d\times \Pcal_{2q+2}(\R^d)$,
\begin{equation*}
    \lim{i\rightarrow\infty}w^{T_i}(0,x,\loi\theta)=w(0,x,\loi\theta),
\end{equation*}
which can be extends on the whole product space due to the equicontinuous estimate \eqref{LTB: estim wT H1}.
Moreover, by replacing $T$ by $T_i$ in \eqref{LTB: estim wT H1}, we get that $w$ does not depend on $x$, i.e. there exists ${\bar w}$,
such that for all $x \in \R^d$, $\theta \in \L{2q+2}(\Omega;\R^d)$,
\begin{equation*}
    \lim{i\rightarrow\infty}w^{T_i}(0,x,\loi\theta)={\bar w(\loi\theta)}.
\end{equation*}
Using the same notation as in part 1.1. of the proof of \autoref{LTB: lemme}, we have, for all $T \geq 0$, $x \in \R^d$, $\theta \in \L{2q+2}(\Omega;\R^d)$,
$$w^T(0,x,\loi \theta)= \lim{n \to + \infty} w^{T,n}(0,x,\loi \theta) \quad \textrm{and} \quad w^{T,n}(0,x,\loi \theta) = \mathcal{P}^n_T[g-u](x,\loi \theta), \quad \forall n \in \N.$$
Thanks to the semigroup property of $\mathcal{P}^n$ and \autoref{LTB: lemme} one could get for all $T\leq S$, $x\in\R^d$, $\theta\in\L{2q+2}(\Omega;\R^d)$ and $n \in \N$, that 

\begin{align}
    \abs*{w^{T,n}(0,x,\mu^*)-w^{S,n}(0,x,\loi\theta)}&=\abs*{w^{T,n}(0,x,\mu^*)-\Pcal^n_{{S}-T}\seg*{w^{T,n}(0,\cdot,{\cdot})}(x,\loi\theta)}\notag\\
    &\leq\E\abs*{w^{T,n}(0,x,\mu^*)-w^{T,n}(0,X^{\beta^T_n,x,\loi\theta}_{{S}-T},\loi{X^{\loi\theta}_{{S}-T}})}\notag\\
    &= \E\abs*{w^{T,n}(0,x,\loi{X^{\mu^*}_{S-T}})-w^{T,n}(0,X^{\beta^T_n,x,\loi\theta}_{{S}-T},\loi{X^{\loi\theta}_{{S}-T}})}\notag\\
    &\leq  C\E\left[\left(1+\abs*{x}^{q+1/2} + \abs*{X^{\beta^T_n,x,\loi\theta}_{S-T}}^{q+1/2}+\norm*{X^{\mu^*}_{S-T}}_{2q+2}^{q+1/2} + \norm*{X^{\loi\theta}_{S-T}}_{2q+2}^{q+1/2}\right)\abs*{x-X^{\beta^T_n,x,\loi\theta}_{S-T}}^{\epsilon/2} \right] e^{-\bm\eta  T}\notag\\
&\quad + C\E\left[\left(1+\abs*{x}^{q} + \abs*{X^{\beta^T_n,x,\loi\theta}_{S-T}}^{q}+\norm*{X^{\mu^*}_{S-T}}_{2q+2}^{q} + \norm*{X^{\loi\theta}_{S-T}}_{2q+2}^{q}\right)\right]\W(\loi{X^{\mu^*}_{S-T}},\loi{X^{\loi\theta}_{S-T}})^\epsilon.\notag
\end{align}
Since the right hand side of the inequality does not depend on $n$, we can take $n \to + \infty$. Then, by using \autoref{SDE: H1: thm: Wp exponential contractivity}, \autoref{SDE: H1: prop: Estim unif T} and \autoref{SDE: H1: prop: Estim unif T Q}, we get, for all $T\leq S$, $x\in\R^d$ and $\theta\in\L{2q+2}(\Omega;\R^d)$, that 
\begin{equation}
\label{eq: wT wTi}
    \abs*{w^{T}(0,x,\mu^*)-w^{S}(0,x,\loi\theta)} \leq C(1+\abs*{x}^{q+1}+\WW_{2q+2}(\mu^*,0)^{q+1}+\norm*{\theta}^{q+1}_{2q+2})(e^{-\bm\eta T} + \W(\mu^*,\loi\theta)e^{-\Lambda\epsilon (S-T)}).
\end{equation}
We apply \eqref{eq: wT wTi} with $\theta \sim \mu^*$, $S=T_i$ and we take $i \to + \infty$ to get 
\begin{equation}
\label{eq: wT mu*}
 \abs*{w^{T}(0,x,\mu^*)-\bar w(\mu^*)} \leq C(1+\abs*{x}^{q+1}+\WW_{2q+2}(\mu^*,0)^{q+1})e^{-\bm\eta T},
\end{equation}
which implies $\lim{T \to +\infty} w^{T}(0,x,\mu^*) = \bar w(\mu^*)$. Returning to \eqref{eq: wT wTi}, we take, $S=T_i$,  $i \to + \infty$ and then $T \to + \infty$ to obtain
$\lim{i \to +\infty} w^{T_i}(0,x,\loi \theta) = \bar w(\mu^*)$. Since $(w^T(0,x,\loi \theta))_T$ is bounded and has a unique accumulation point, we conclude that,  for all $x\in\R^d$ and $\theta\in\L{2q+2}(\Omega;\R^d)$, $\lim{T \to +\infty} w^{T}(0,x,\loi \theta) = \bar w(\mu^*)$. 
Finally, by using \eqref{eq: wT wTi} and \eqref{eq: wT mu*}, we obtain
\begin{eqnarray*}
    \abs*{w^{T}(0,x,\loi \theta)-\bar w(\mu^*)} &\leq& \abs*{w^{T}(0,x,\loi \theta)-w^{T/2}(0,x,\mu^*)}+ \abs*{w^{T/2}(0,x,\mu^*)-\bar w(\mu^*)}\\
    &\leq& C(1+\abs*{x}^{q+1}+\WW_{2q+2}(\mu^*,0)^{q+1}+\norm*{\theta}^{q+1}_{2q+2})(e^{-\bm\eta T/2} + \W(\mu^*,\loi\theta)e^{-\Lambda\epsilon T/2})\\
    &&+ C(1+\abs*{x}^{q+1}+\WW_{2q+2}(\mu^*,0)^{q+1})e^{-\bm\eta T/2}
\end{eqnarray*}
which gives us the result by setting  $\ell=\bar{w}(\mu^*)$.
\end{proof}

\begin{thm}[LTB 3]\label{LTB: LTB 3} Let (\nameref{EBSDE: H1})-(\nameref{LTB: hyp g}) or (\nameref{EBSDE: H2})-(\nameref{LTB: hyp g}) be fulfilled. If we assume that $b,\sigma$ are $\C^{1,0}$ with globally Lipschitz derivative w.r.t. $x$, then $x \mapsto u^T(t,x,\loi \theta)$ and $x \mapsto u(x,\loi\theta)$ are $\mathscr{C}^1$ for all $T>0$, $t\in [0,T)$, $\theta \in \L{2q+2}(\Omega;\R^d)$. Moreover, for all $T\geq 1$,  there exists $C>0$ and $\bm\eta>0$, such that, for all $x\in\R^d,$ $\theta\in\L{2q+2}(\Omega;\R^d)$, the following holds

\begin{equation}\label{LTB: eq: LTB 3 grad}
    \abs*{\nabla_x u^T(0,x,\loi{\theta})-\nabla_xu(x,\loi\theta)}\leq C(1+\abs*{x}^{q+1}+\norm*{\theta}^{q+1}_{2q+2})e^{-\bm\eta T}.
\end{equation}
and there exists a continuous version of $Z^{T,x,\loi\theta}$ and $Z^{x,\loi\theta}$ satisfying
\begin{equation}\label{LTB: eq: LTB 3 Z}
    \abs*{Z^{T,x,\loi\theta}_0-Z^{x,\loi\theta}_0}\leq C\norm*{\sigma}_\infty(1+\abs*{x}^{q+1}+\norm*{\theta}^{q+1}_{2q+2})e^{-\bm\eta T}.
\end{equation}

\end{thm}

\begin{proof}
    First of all, \cite[Theorem 4.2]{Furhmann-Tessitore_BE-Formula} applied to BSDEs \eqref{LTB: Decoupled EBSDE}, seen as a finite horizon BSDE on $[0,T]$, and \eqref{LTB: Decoupled BSDE} gives us the $\mathscr{C}^1$ regularity of $u^T(t,.,\loi \theta)$ and $u(.,\loi \theta)$. Moreover we have the following gradient representation of $Z$ and $Z^T$, for $t \in [0,T)$,
\begin{align}
Z^{x,\loi\theta}_t&=\nabla_xY^{x,\loi\theta}_t\sigma(X^{x,\loi\theta}_t,\loi{X^\loi\theta_t})=\nabla_xu(X^{x,\loi\theta}_t,\loi{X^\loi\theta_t})\sigma(X^{x,\loi\theta}_t,\loi{X^\loi\theta_t})\label{LTB: eq: gradient Z}\\
Z^{T,x,\loi\theta}_t&=\nabla_xY^{T,x,\loi\theta}_t\sigma(X^{x,\loi\theta}_t,\loi{X^\loi\theta_t})=\nabla_xu^T(t,X^{x,\loi\theta}_t,\loi{X^\loi\theta_t})\sigma(X^{x,\loi\theta}_t,\loi{X^\loi\theta_t})\label{LTB: eq: gradient Z T}
\end{align}
that gives us some continuous version of $Z^{x,\loi\theta}$ and $Z^{T,x,\loi\theta}$.

    Let us set $\bar{w}^T:=w^T-\ell$ with $\ell$ given in \autoref{LTB: LTB 2}. By definition $\bar{w}^T$ solves the following finite horizon BSDE on $[0,T]$
        \begin{align*}
        \bar{w}^T(t,X^{x,\loi\theta}_t,\loi{X^\loi\theta_t})&=\bar{w}^T(T,X^{x,\loi\theta}_T,\loi{X^\loi\theta_T}) + \int^T_t \left(f(X^{x,\loi\theta}_s,\loi{X^\loi\theta_s},Z^{T,x,\loi\theta}_s)-f(X^{x,\loi\theta}_s,\loi{X^\loi\theta_s},Z^{x,\loi\theta}_s)\right) \d s\\
        &\quad-\int^T_t \left(Z^{T,x,\loi\theta}_s-Z^{x,\loi\theta}_s\right)\d W_s.
    \end{align*}
    Rewriting the previous BSDE on $[0,\delta]$ for some $\delta \in (0,T)$ and using the Markov property gives that $\bar{w}^T$ solves, for $t \in [0,\delta)$ the following equation
    \begin{align}
        \bar{w}^T(t,x,\loi\theta)&=\bar{w}^T(\delta,X^{t,x,\loi\theta}_\delta,\loi{X^{t,\loi\theta}_\delta}) + \int^\delta_t \left(f(X^{t,x,\loi\theta}_s,\loi{X^{t,\loi\theta}_s},Z^{T,t,x,\loi\theta}_s)-f(X^{t,x,\loi\theta}_s,\loi{X^{t,\loi\theta}_s},Z^{t,x,\loi\theta}_s)\right) \d s\notag \\
        &\quad-\int^\delta_t \left( Z^{T,t,x,\loi\theta}_s-Z^{t,x,\loi\theta}_s \right)\d W_s\notag\\
        &=\bar{w}^{T-\delta}(0,X^{t,x,\loi\theta}_\delta,\loi{X^{t,\loi\theta}_\delta})  + \int^\delta_t \left(f(X^{t,x,\loi\theta}_s,\loi{X^{t,\loi\theta}_s},Z^{T,t,x,\loi\theta}_s)-f(X^{t,x,\loi\theta}_s,\loi{X^{t,\loi\theta}_s},Z^{t,x,\loi\theta}_s)\right)\d s\notag\\
        &\quad-\int^\delta_t \left( Z^{T,t,x,\loi\theta}_s-Z^{t,x,\loi\theta}_s \right)\d W_s.\label{LTB: eq: v T}
    \end{align}
    We want to obtain a suitable bound on $\nabla_x\bar{w}^T(0,x,\loi\theta)$ by applying \cite[Theorem 4.2 - (4.6)]{Furhmann-Tessitore_BE-Formula}. In particular, we want to track precisely the constant that appears in this estimate. By checking carrefully the proof of \cite[Theorem 4.2 - (4.6)]{Furhmann-Tessitore_BE-Formula}, this bound follows from a standard regularization step and the application of \cite[Theorem 3.10 and Corollary 3.11]{Furhmann-Tessitore_BE-Formula}. In particular, we can assume that $f$ has the asked regularity by \cite[Theorem 3.10]{Furhmann-Tessitore_BE-Formula} in order to apply it. Due to \cite[Proposition 3.2]{McMurray-Crisan}, the regularity assumption on $b,\sigma$ w.r.t $x$ ensures that $X^{\cdot,\loi\theta}$ is differentiable for all $\theta\in\L{2q+2}(\Omega)$. Then, let us set 
    $$U(t,s):= \frac{1}{s-t}\int^t_s \braket*{\sigma(X^{t,x,\loi\theta}_{r},\loi{X^{t,\loi\theta}_{r}})^{-1}\nabla_xX^{t,x,\loi\theta}_r}{\d W_r}, \quad s \in (t,\delta].$$
    A direct application of Itô's isometry and the classical estimate  
    $\E\seg*{\abs*{\nabla_xX^{t,x,\loi\theta}_s}^2} \leq e^{2c\delta}$ give us, for all $x\in\R^d,\theta\in\L{2q+2}(\Omega;\R^d)$ and all $s\in (t,\delta]$, 
    $$\E\seg*{\abs*{U(t,s)}^2}^{1/2}\leq \frac{ \norm*{\sigma^{-1}}_\infty e^{c\delta}}{ \sqrt{s-t}}.
    $$
We apply \cite[Theorem 3.10]{Furhmann-Tessitore_BE-Formula} to \eqref{LTB: eq: v T}. Hence, we obtain by using Cauchy-Schwarz's inequality and \autoref{LTB: LTB 2}
    \begin{align}
        \nabla_x \bar{w}^T(t,x,\loi\theta)&=\E\seg*{\bar{w}^{T-\delta}(0,X^{t, x,\loi\theta}_\delta,\loi{X^{t,\loi\theta}_\delta}) U(t,\delta)}\notag \\
        &\quad+ \E\seg*{\int^\delta_t(f(X^{t,x,\loi\theta}_s,\loi{X^{t,\loi\theta}_s},Z^{T,t,x,\loi\theta}_s-f(X^{t, x,\loi\theta}_s,\loi{X^{t,\loi\theta}_s},Z^{t,x,\loi\theta}_s))U(t,s)\d s},\notag\\
        \abs*{\nabla_x \bar{w}^T(t,x,\loi\theta)}&\leq \E\seg*{C\left( 1+\abs*{X^{t,x,\loi\theta}_\delta}^{q+1}+\norm*{X^{t,\loi\theta}_\delta}^{q+1}_{2q+2} \right)e^{-\bm\eta (T-\delta)}\abs*{U(t,\delta)}}+ K^f_z\E\seg*{\int^\delta_t\abs*{Z^{T,t,x,\loi\theta}_s-Z^{x,t,\loi\theta}_s}\abs*{U(t,s)}\d s}. \label{LTB: eq: ineq grad v T}
    \end{align}
    Consequently, due to \eqref{LTB: eq: gradient Z}-\eqref{LTB: eq: gradient Z T}, the Markov property and the boundedness of $\sigma,$
    $$\abs*{Z^{T,t,x,\loi\theta}_s-Z^{t,x,\loi\theta}_s}\leq \norm*{\sigma}_\infty\abs*{\nabla_x\bar{w}^T(s,X^{t,x,\loi\theta}_s,\loi{X^{t,\loi\theta}_s})}.
    $$
    Then, \eqref{LTB: eq: ineq grad v T} writes, 
    \begin{align*}
        \abs*{\nabla_x \bar{w}^T(t,x,\loi\theta)}&\leq \E\seg*{C\left( 1+\abs*{X^{t,x,\loi\theta}_\delta}^{q+1}+\norm*{X^{t,\loi\theta}_\delta}^{q+1}_{2q+2} \right)\abs*{U(t,\delta)}}e^{-\bm\eta (T-\delta)}+ K^f_z\norm*{\sigma}_\infty \E\seg*{\int^\delta_t\abs*{\nabla_x\bar{w}^T(s,X^{t,x,\loi\theta}_s,\loi{X^{t,\loi\theta}_s})}\abs*{U(t,s)}\d s}.
    \end{align*}
    Hence, by mimicking the proof of \cite[Corollary 3.11]{Furhmann-Tessitore_BE-Formula}, and following their notations, we have 
    \begin{equation*}\begin{array}{l}
         I_1(t,x,\loi\theta)=I_2(t,x,\loi\theta)=0, \\[8pt]
         I_3(t,x,\loi\theta)= K^f_z\norm*{\sigma}_\infty \E\seg*{\displaystyle\int^\delta_t\abs*{\nabla_x\bar{w}^T(s,X^{t,x,\loi\theta}_s,\loi{X^{t,\loi\theta}_s}}\abs*{U(t,s)}\d s},\\[8pt]
         I_4(t,x,\loi\theta)=\E\seg*{C\left( 1+\abs*{X^{t,x,\loi\theta}_\delta}^{q+1}+\norm*{X^{t,\loi\theta}_\delta}^{q+1}_{2q+2} \right)e^{-\bm\eta(T-\delta)}\abs*{U(t,\delta)}} ,\\[8pt]
         \abs*{\norm*{\nabla_x \bar{w}^T}}:=\sup{t\in\seg*{0,\delta}}\sup{x\in\R^d}\sup{\theta\in\L{2q+2}(\Omega;\R^d)}\frac{\sqrt{\delta-t}}{(1+\abs*{x}^{q+1}+\norm*{\theta}_{2q+2}^{q+1})e^{-\bm\eta(T-\delta)}e^{c\delta}}\abs*{\nabla_x \bar{w}^{T}(t,x,\loi\theta)}
    \end{array}
    \end{equation*} we obtain, for $|||I_4|||$ defined the same way, due to Cauchy-Schwarz's inequality and \autoref{SDE: H1: prop: Estim unif T} or \autoref{SDE: H2: prop: Estim unif T}
    \begin{align*}
        I_4(t,x,\loi \theta)&\leq C\norm*{\sigma^{-1}}_\infty\frac{e^{c\delta}}{\sqrt{\delta-t}}(1+\abs*{x}^{q+1}+\norm*{\theta}_{2q+2}^{q+1})e^{-\bm\eta(T-\delta)},\\
        |||I_4|||&\leq C.
    \end{align*}
    Using again Cauchy-Schwarz and \autoref{SDE: H1: prop: Estim unif T} or \autoref{SDE: H2: prop: Estim unif T}, we also get
    \begin{align*}
        I_3(t,x,\loi\theta)&\leq  K^f_z\norm*{\sigma}_\infty \int^\delta_t\E\seg*{\abs*{\nabla_x\bar{w}^T(s,X^{t,x,\loi\theta}_s,\loi{X^{t,\loi\theta}_s})}^2}^{1/2}\E\seg*{\abs*{U(t,s)}^2}^{1/2}\d s\\
        &\leq K^f_z\norm*{\sigma}_\infty \abs*{\norm*{\nabla_x \bar{w}^T}}\int^\delta_t(1+\E\seg*{\abs*{X^{t,x,\loi\theta}_s}^{2q+2}}^{1/2}+\norm*{X^{t,\loi\theta}_s}^{q+1}_{2q+2})\frac{e^{-\bm\eta(T-\delta)}e^{c\delta}}{\sqrt{\delta-s}}\E\seg*{\abs*{U(t,s)}^2}^{1/2}\d s\\
        &\leq K^f_z\norm*{\sigma}_\infty (1+\abs*{x}^{q+1}+\norm*{\theta}_{2q+2}^{q+1})e^{-\bm\eta(T-\delta)}\abs*{\norm*{\nabla_x \bar{w}^T}}\int^\delta_t \frac{\norm*{\sigma^{-1}}_\infty e^{2c\delta}}{\sqrt{\delta-s}\sqrt{s}}\d s\\
        &=\pi K^f_z\norm*{\sigma}_\infty\norm*{\sigma^{-1}}_\infty  e^{2c\delta} (1+\abs*{x}^{q+1}+\norm*{\theta}_{2q+2}^{q+1})e^{-\bm\eta (T-\delta)}\abs*{\norm*{\nabla_x \bar{w}^T}},\\
        |||I_3|||&\leq \pi K^f_z\norm*{\sigma}_\infty\norm*{\sigma^{-1}}_\infty  \sqrt{\delta} e^{c\delta}\abs*{\norm*{\nabla_x \bar{w}^T}}.
    \end{align*}
 Hence, by taking $\delta$ small enough such that $\pi K^f_z\norm*{\sigma}_\infty\norm*{\sigma^{-1}}_\infty  \sqrt{\delta} e^{c\delta}<1$, we obtain, $||| \nabla_x\bar{w}^T(0,x,\loi\theta)|||\leq C$. Thus, we have for the well-chosen $\delta$
 \begin{align*}
 \abs*{\nabla_x \bar{w}^T(0,x,\loi\theta)}&\leq C(1+\abs*{x}^{q+1}+\norm*{\theta}^{q+1}_{2q+2})e^{-\bm\eta(T-\delta)}\frac{e^{c\delta}}{\sqrt{\delta}}\\
 &\leq C(1+\abs*{x}^{q+1}+\norm*{\theta}^{q+1}_{2q+2})e^{-\bm\eta T},
 \end{align*}
 which concludes the proof of \eqref{LTB: eq: LTB 3 grad}. For \eqref{LTB: eq: LTB 3 Z}, we use again the gradient representations \eqref{LTB: eq: gradient Z} and \eqref{LTB: eq: gradient Z T} to obtain, due to the boundedness of $\sigma,$
    $$\abs*{Z^{T,x,\loi\theta}_0-Z^{x,\loi\theta}_0}\leq\abs*{\nabla_xu^T(0,x,\loi\theta)\sigma(x,\loi\theta)-\nabla_xu(x,\loi\theta)\sigma(x,\loi\theta)}\leq \norm*{\sigma}_\infty\abs*{\nabla_x\bar{w}^T(0,x,\loi\theta)},
    $$
 which concludes the proof.
\end{proof}

\begin{rmq}
   One could remark that the above results holds for  $(Y^{T,\theta,\loi\theta},Z^{T,\theta,\loi\theta})=:(Y^{T,\loi\theta},Z^{T,\loi\theta})$ the solution to the following finite-horizon \emph{coupled} BSDE 
   $$Y^{T,\loi\theta}_t= g(X^{\loi\theta}_T,\loi{X^\loi\theta_T})+\int^T_t f(X^\loi\theta_s,\loi{X^\loi\theta_s},Z^{T,\loi\theta})\d s -\int^T_t Z^{T,\loi\theta}_s \d W_s, \quad\quad \forall t\in\seg*{0,T},\, \theta\in\L{2}(\Omega)
   $$
   towards $(Y^\loi\theta,Z^\loi\theta)$ part of the solution to \eqref{EBSDE: EBSDE}.
\end{rmq}

\section{Application to Optimal ergodic control problem}\label{OCP}
We introduce in this Section a partial McKean-Vlasov optimal ergodic control problem. Let us explain a little bit the denomination {\emph{partial}}. If we see our McKean-Vlasov dynamic as the limit of a particle system where the number of particles tends to infinity, we are considering a problem where we want to control in an optimal way only one particle of this system. Then the limiting McKean-Vlasov dynamic is not impacted by the control (which differs from the mean-field game framework). Let $\mathscr{A}$ be a closed, bounded subset of $\R^k$, with $k\in\N$ and let us define an admissible control $\bm{a}: (\omega,s)\in \Omega \times \R_+\rightarrow \mathscr{A}\ni a_s$ as an $\F_s$-progressively measurable process. We make the following assumptions throughout the section.
\begin{hyp}\label{OCP: hyp}Let $R\in\R^{d\times k}$ and $L,g$ be satisfying the following:
    \begin{enumerate}
        \item $L:\R^d\times \Pcal_{2q+2}\times \mathscr{A}\rightarrow\R$ and there exists $q\geq 0$, $C,K^L_x,K^L_\LL>0$, $\epsilon\in(0,1]$, such that for all $x,x'\in\R^d$, $\theta,\theta'\in\L{2q+2}(\Omega;\R^d)$, $a\in \mathscr{A}$, 
        \begin{enumerate}
            \item $\abs*{L(x,\loi{\theta},a)-L(x',\loi{\theta'},a)}\leq C(1+\abs*{x}^q+\abs*{x'}^q+\norm*{\theta}_{2q+2}^q+\norm*{\theta'}_{2q+2}^q)\left(K^L_x\abs*{x-x'}^\epsilon+K^L_\LL\W(\loi{\theta},\loi{\theta'})^\epsilon\right)$,
            \item $\abs*{L(x,\loi{\theta},a)}\leq C(1+\abs*{x}^{q+1}+\norm*{\theta}_{2q+2}^{q+1})$,
        \end{enumerate}
        \item $g:\R^d\times\Pcal_{2q+2}\mapsto\R$ and there exists $q \geq 0$, $C,K^g_x,K^g_\LL>0$, such that for all $x,x'\in\R^d$, $\theta, \theta'\in\L{2q+2}(\Omega;\R^d)$,
        \begin{enumerate}
            \item $\abs*{g(x,\loi{\theta})-g(x',\loi{\theta'})}\leq C(1+\abs*{x}^q+\abs*{x'}^q+\norm*{\theta}_{2q+2}^q+\norm*{\theta'}_{2q+2}^q)\left(   K^g_x\abs*{x-x'}^\epsilon+K^g_\LL\W(\loi{\theta},\loi{\theta'})^\epsilon  \right)$,
            \item $\abs*{g(x,\loi{\theta})}\leq C(1+\abs*{x}^{q+1}+\norm*{\theta}_{2q+2}^{q+1})$.
        \end{enumerate}
    \end{enumerate} 
    Without restriction we can take same constants $q$ and $C$ for $L$ and $g$.
    \end{hyp}
    Let us set, for any admissible control $\bm a$ and $T>0$ the following Girsanov's change of probability
    \begin{align*}
        \rho^{\bm a}_T&=\exp{-\frac{1}{2}\int^T_0 \abs*{Ra_t}^2 \d s +\int^T_0 (R a_t)^\top\d W_t}\quad \text{and} \quad
        \P^{\bm a}_T=\rho^{\bm a}_T\P.
    \end{align*}
 
    Thanks to our assumptions, 
    $W_t^{\bm a}:=W_t-\int^t_0 Ra_s\d s$ is a $\P^{\bm a}_T-$Brownian motion and so, we can write $X^{\loi \theta}$ and $X^{x,\loi\theta}$ the respective solutions to \eqref{SDE: MKV SDE} and \eqref{SDE: Decoupled MKV SDE} w.r.t to that new Brownian motion without control on the distribution. Namely, 
    \begin{align}\label{EDS control}
         \d X^{\loi \theta,\bm a}_t  &=\left(b(X^{\loi \theta,\bm a}_t,\loi{X^{\loi \theta,\bm a}_t})+\sigma(X^{\loi \theta,\bm a}_t,\loi{X^{\loi \theta,\bm a}_t})Ra_t\right)\d t +\sigma(X^{\loi \theta,\bm a}_t,\loi{X^{\loi \theta,\bm a}_t})\d W^{\bm a}_t,
    \end{align}
    and
    \begin{align}\label{EDS control decoupled}
        \d X^{x,\loi\theta,\bm a}_t &=\left(b(X^{x,\loi\theta,\bm a}_t,\loi{X^{\loi\theta,0}_t})+\sigma(X^{x,\loi\theta,\bm a}_t,\loi{X^{\loi\theta,0}_t})Ra_t\right)\d t +\sigma(X^{x,\loi\theta,\bm a}_t,\loi{X^{\loi\theta,0}_t})\d W^{\bm a}_t,
    \end{align}
    where $X^{\cdot,0}$ satisfies \eqref{EDS control} without control, i.e. $a_t=0$ for all $t\geq 0$.
    \begin{rmq}
        We stress the fact that \eqref{EDS control decoupled} is no longer of McKean-Vlasov's type but still have a distribution dependency w.r.t a McKean-Vlasov SDE.
    \end{rmq}
       Let us set the finite horizon cost as
    \begin{equation*}
        J^T(x,\loi\theta,\bm a)=\E^{\bm a}_T\seg*{\int^T_0 L(X^{x,\loi\theta,\bm a}_s,\loi{X^{\loi\theta,0}_s},a_s)\d s}+\E^{\bm a}_T\seg*{g(X^{x,\loi\theta,\bm a}_T,\loi{X^{\loi\theta,0}_T})}
    \end{equation*}
    and the associated optimal control problem as the minimization of $J^T$ over all admissible control $\bm a$. We can also define the ergodic cost as \begin{equation*}
        J(x,\loi\theta,\bm a)=\limsup_{T\rightarrow\infty}\frac{1}{T}\E^{\bm a}_T\seg*{\int^T_0 L(X^{x,\loi\theta,\bm a}_s,\loi{X^{\loi\theta,0}_s},a_s)\d s}
    \end{equation*}
    and the associated optimal control problem as the minimization of  $J$ over all admissible control $\bm a$.
    As usual, we introduce the associated Hamiltonian function as
        \begin{equation}\label{hamiltonian def}
        f(x,\loi{\theta},z):=\inf{a\in \mathscr{A}}\set*{L(x,\loi{\theta}, a)+ zR
        a}.
    \end{equation}
     Using Filippov's theorem (see \cite[Theorem 4]{filippov}), if the infimum is attained, there exists 
      $\varphi:\R^d\times\Pcal_{2q+2}\times\R^d\rightarrow \mathscr{A}$ measurable s.t. $$f(x,\loi{\theta},z)=L(x,\loi{\theta},\varphi(x,\loi{\theta},z))+zR(\varphi(x,\loi{\theta},z)).$$
      Following lemmas are direct applications of the definitions and assumptions given above.
\begin{lem}
\label{lem: f satisfies H0}
    Assume that \autoref{OCP: hyp} holds. Then the function $f$ given by \eqref{hamiltonian def} satisfies that
    \begin{enumerate}
        \item for all $x\in\R^d$ and all $\theta\in\L{2q+2}$, $\abs*{f(x,\loi{\theta},0)}\leq C(1+\abs*{x}^{q+1}+\norm*{\theta}_{2q+2}^{q+1})$,
        \item for all $x,x',z,z'\in\R^d$ and all $\theta,\theta'\in\L{2q+2}(\Omega;\R^d)$,
        \begin{align*}
            \abs*{f(x,\loi{\theta},z)-f(x',\loi{\theta'},z')}&\leq C(1+\abs*{x}^q+\abs*{x'}^q+\norm*{\theta}_{2q+2}^q+\norm*{\theta'}_{2q+2}^q)\left(K^L_x\abs*{x-x'}^\epsilon+K^L_\LL \W(\loi{\theta},\loi{\theta'})^\epsilon\right)\\&\quad+\abs*{R\mathscr{A}}_\infty\abs*{z-z'}.
        \end{align*}
    \end{enumerate}
    where $\abs*{R\mathscr{A}}_\infty=\sup{}\set*{\abs*{Ra},\,  a\in\mathscr{A}}$.    In particular $f$ satisfies (\nameref{EBSD: H0}) with $K^f_z=\abs*{R\mathscr{A}}_\infty$.
    \end{lem}
    \begin{proof}
        \textit{1.} is a direct consequence of \eqref{hamiltonian def} and \autoref{OCP: hyp}-1., \textit{2.} is a consequence of \cite[Lemma 5.2]{Furhmann-Tessitore_BE-Formula}.
    \end{proof}
\begin{hyp*}[$\H_{OCP}$]\label{OCP: hyp: OCP}Assume that $\autoref{OCP: hyp}$ holds and that one of the following is fulfilled 
\begin{enumerate}
    \item (\nameref{SDE: H1}) holds, $\sigma$ is uniformly elliptic and $\nu>K^\sigma_x+\sqrt{2K^\sigma_x} \abs*{R\mathscr{A}}_\infty$.
    \item (\nameref{SDE: H2-2}) holds.
\end{enumerate}
\end{hyp*}
\begin{lem}\label{OCP: Comparaison finite}We assume that  (\nameref{OCP: hyp: OCP}) holds.
    For $(Y^{T,x,\loi\theta},Z^{T,x,\loi\theta})$ the solution to \eqref{LTB: Decoupled BSDE}, 
    we have that, for all admissible control $\bm a$, $$J^T(x,\loi\theta,\bm a)\geq Y^{T,x,\loi\theta}_0.$$ Moreover, if the infimum is attained for every $x,z\in\R^d$ in \eqref{hamiltonian def}, there exists $\bar{\bm{a}}^T$ an admissible control s.t. $J^T(x,\loi\theta,\bar{\bm{a}}^T)=Y^{T,x,\loi\theta}_0$ where \begin{equation*}
        \bar a^T_t :=\varphi\left(X^{x,\loi\theta,\bar{\bm{a}}^T}_t,\loi{X^{\loi\theta,0}_t},Z^{T,x,\loi\theta}_t\right)=\varphi\left(X^{x,\loi\theta,\bar{\bm{a}}^T}_t,\loi{X^{\loi\theta,0}_t},\nabla_xu^T (t, X^{x,\loi\theta,\bar{\bm{a}}^T}_t,\loi{X^{\loi\theta,0}_t})\sigma(X^{x,\loi\theta,\bar{\bm{a}}^T}_t,\loi{X^{\loi\theta,0}_t})\right).
    \end{equation*}
\end{lem}
\begin{lem}\label{OCP: Comparaison ergodic}We assume that  (\nameref{OCP: hyp: OCP}) holds. For $(Y^{x,\loi\theta},Z^{x,\loi\theta},\lambda)$ the solution to \eqref{LTB: Decoupled EBSDE}, we have that,  for all admissible control $\bm a$,  
$$J(x,\loi\theta,\bm a)\geq \lambda.$$ 
Moreover, if the infimum is attained for every $x,z\in\R^d$ in \eqref{hamiltonian def},  there exists $\bar{\bm a}$ an admissible control s.t $J(x,\loi\theta,\bar{\bm{a}})=\lambda$, where 
    $$\bar a_t=\varphi(X^{x,\loi\theta,\bar{\bm{a}}}_t,\loi{X^{\loi\theta,0}_t},Z^{x,\loi\theta}_t)=\varphi(X^{x,\loi\theta,\bar{\bm{a}}}_t,\loi{X^{\loi\theta,0}_t},\nabla_xu^T (t,X^{x,\loi\theta,\bar{\bm{a}}}_t,\loi{X^{\loi\theta,0}_t})\sigma(X^{x,\loi\theta,\bar{\bm{a}}}_t,\loi{X^{\loi\theta,0}_t})).$$
\end{lem}
\begin{proof}[Proofs of \autoref{OCP: Comparaison finite} and \autoref{OCP: Comparaison ergodic}]
    The proofs follow from the definition of $f$, see for instance \cite[Theorem 7.1]{Furhman-Tessitore-Hu_EBSDE_Banach} in the strong dissipative and non-McKean-Vlasov's framework
\end{proof}
\begin{thm}\label{OCP: thm}We assume that  (\nameref{OCP: hyp: OCP}) holds.
    For every control $\bm a$, we have:
    \begin{equation}\label{eq 1}
        \liminf_{T\rightarrow\infty}\frac{J^T(x,\loi{\theta},\bm a)}{T}\geq \lambda.
    \end{equation}
        Furthermore, if the infimum is attained for every $x,\loi{\theta},z$ in \eqref{hamiltonian def}, we have that there exists $\ell \in \R$, $C> 0$, $\bm \eta >0$ such that:
    \begin{equation}\label{eq 2}
        \abs*{J^T(x,\loi{\theta},\bar{\bm{a}}^T) - J(x,\loi{\theta},\bar{\bm{a}})T -Y^{x,\loi\theta}_0 -\ell}\leq C(1+\abs*{x}^{q+1}+\norm*{\theta}^{q+1}_{2q+2})e^{-\bm{\eta} T}
    \end{equation}
    \end{thm}
    \begin{proof}
            Thanks to \autoref{lem: f satisfies H0}, $f$ satisfies (\nameref{EBSD: H0}). Then, \eqref{eq 1} is a direct consequence of \autoref{OCP: Comparaison finite} and \autoref{LTB: LTB 1} while \eqref{eq 2} is a consequence of \autoref{LTB: LTB 2}. 
    \end{proof}
    \begin{thm} 
     If we additionally assume that $b,\sigma$ are $\C^{1,0}$ with gradient globally Lipschitz w.r.t x, that $\mathscr{A}$ is convex and that $L$ is $\bm\mu$-strongly convex, then, the infimum is attained in \autoref{OCP: thm} and there exists $C$ that depends on $\bm{\mu}$ such that,  for all $T\geq 1$, $x\in \R^d$, $\theta \in \L{2q+2}(\Omega;\R^d)$,
     \begin{equation}\label{eq 3}
        \abs*{\bar a^T_0-\bar a_0}\leq C(1+\abs*{x}^{q+1}+\norm*{\theta}^{q+1}_{2q+2})e^{-\bm\eta  T}.
    \end{equation}
\end{thm}

\begin{proof}
Due to the additional $\bm\mu$-strongly convex of $L$ and convexity of $\A$ assumptions, (\nameref{OCP: hyp: OCP}) satisfies the assumptions of \cite[Lemma 3.3]{Delarue-Carmona_Book} which gives that the infimum is attained. Consequently, $\phi$ (which is the same for  $\bar{\bm{a}}$ and $\bar{\bm{a}}^T$) exists and is globally Lipschitz w.r.t $z$ with the constant depending on $\bm\mu$. Hence, by applying \autoref{LTB: LTB 3} we obtain the wanted result.
\end{proof}
\renewcommand{\thesection}{A}
\section{Appendix}\label{Appendix}
Let us consider the following non-homogeneous forward SDE,
\begin{equation}\label{Appendix: SDE}
    \d\bm X_t=\b(t,\bm X_t)\d t+\sigma(t, \bm X_t)\d W_t,\quad\bm X_0\in \L{p}(\Omega,\R^d),\, p\geq 2,
\end{equation}
where we make the following assumptions.
\begin{hyp}\label{Appendix: hyp}
    $\vspace{0pt}$
    \begin{enumerate}
        \item $\b:\R_+\times\R^d\mapsto\R^d$ is measurable w.r.t time, $K^\b_x $-Lipchitz w.r.t $x$ and satisfies that there exists $\bm\eta,R,M_\b>0$, such that for all $x,x'\in\R^d$, $t\geq 0$, 
        \begin{equation}\label{Appendix: hyp: dissipatif}
            \braket{x-x'}{\b(t,x)-\b(t,x')}\leq -\bm\eta\abs*{x-x'}^2\1{\abs*{x-x'}> R} + M_\b\abs*{x-x'}\1{\abs*{x-x'}\leq R}.
        \end{equation}
        \item $\sigma:\R_+\times\R^d\mapsto\R^{d\times d}$ is $\sqrt{2K^\sigma_x}$-Lipschitz and uniformly elliptic uniformly in time, i.e. there exists $\sigma_0>0$, such that for all $t\in\R_+$, $x\in\R^d$, $\sigma\sigma^\top(t, x)\geq\sigma_0^2 I_d $.
        \item $\bm\eta>K^\sigma_x$.
    \end{enumerate}
\end{hyp}

\begin{rmq}
    One could remark that it is possible to rewrite \eqref{Appendix: hyp: dissipatif} using the Lipschitz property of $\b$ as
    \begin{equation}\label{Appendix: dissipatif}
        \braket{x-x'}{\b(t,x)-\b(t,x')}\leq -\bm\eta\abs*{x-x'}^2 +\left(\left(M_\b \wedge K^\b_x\abs*{x-x'} \right)+\bm\eta\abs*{x-x'}\right)\abs*{x-x'}\1{\abs*{x-x'}\leq R}.
    \end{equation}
\end{rmq}
\begin{thm}\label{Appendix: thm general}
    Let us set $\bm X'$ be the solution to \eqref{Appendix: SDE} with drift $\b'$ that satisfies the suitable assumptions ensuring that $\bm X'$ exists. We set,  for all $t\geq0$, $\EE_t:=\norm*{\b(t,\cdot)-\b'(t,\cdot)}_\infty$.
    Then, by considering $\bm{c}>0$ that satisfies 
    \begin{equation}\label{Appendix: constraint}
        \bm{c} < (\bm\eta-K^\sigma_x)\exp{-\frac{\bm\eta+\frac{2M_\b }{R}}{2\sigma_0^2}R^2},
        \end{equation}
        there exists $\bm{\hat\eta}=(\bm\eta,M_\b, R,K^\sigma_x, \bm{c}, \sigma_0)>0$ and $C=C(\bm\eta,R,M_\b, K^\sigma_x, \bm{c}, \sigma_0)>0$ such that 
    \begin{align}
        \w(\loi{\bm X_t},\loi{\bm X'_t})&\leq C\w(\loi{\bm X_0},\loi{\bm X'_0})e^{-\bm{\hat{\eta}}t}+Ce^{-\bm{\hat\eta} t}\int^t_0 \left( \EE_s-\bm{c} \E\seg*{\abs*{\bm X_s-\bm X'_s}}\right)e^{\bm{\hat\eta} s}\d s.\label{Appendix: eq: thm general}
    \end{align}
\end{thm}
\begin{rmq}
    $\mathfrak{b}$ is assumed to be Lipschitz with respect to the second variable in order to ensure the existence and the uniqueness of a strong solution to \eqref{Appendix: SDE} but, importantly, \eqref{Appendix: constraint} and \eqref{Appendix: eq: thm general} do not depend on $K^\b_x$.
\end{rmq}
\begin{proof}
    We want to adapt the proof of \cite[Theorem 3.2]{Huang-Ma} to our framework, namely we assume, in their notations, $r_0=0$ since we are not path dependent in our case. However, we propose a slightly more general framework with two different drifts. 
    For $W^1,W^2,W^3$ three independent Brownian motions and for $R>\delta>0$, $\pi^1_\delta$ and $\pi^2_\delta:\R_+\mapsto\R_+$ two Lipschitz functions such that 
\begin{align*}
   \seg*{0,1}\ni \pi^1_\delta(r)&= \left\{
   \begin{array}{cc}
        1& \text{ if }r\geq \delta  \\
        0 & \text{ if }r\leq \delta/2
    \end{array},\right.\\
    \pi^1_\delta(r)^2+\pi^2_\delta(r)^2&=1, \quad\quad \forall\,r\in\R_+,
\end{align*} we define the coupling $(\X,\X')$ of $(\bm X,\bm X')$ as follows:
\begin{equation}\label{Approx reflection coupling general}\left\{\begin{array}{rl}
    \d \X_t &=\b(t,\X_t)\d t +\sigma_0\pi^1_\delta(r_t)\d W^1_t + \sigma_0\pi^2_\delta(r_t)\d W^3_t+\bar\sigma(t, \X_t)\d W^2_t, \\
    \d \X'_t&=\b'(t,\X'_t)\d t +\sigma_0\pi^1_\delta(r_t)(I_d -2e_te_t^\top)\d W^1_t + \sigma_0\pi^2_\delta(r_t)\d W^3_t+\bar\sigma(t, \X'_t)\d W^2_t
\end{array}\right.
\end{equation} where, we define, for $t\geq 0$, $$\~\X_t:=\X_t-\X'_t, \quad\quad r_t:=\abs*{\~\X_t},\quad\quad e_t:=\frac{\~\X_t}{r_t} \1{r_t \neq 0},\quad\text{ and }\quad \sigma\sigma^\top=\sigma_0^2 I_{d\times d} + \bar\sigma\bar\sigma^\top.$$

    Let us set, for all $r\geq 0$,
    $$\kappa(r):=\sup{t\geq 0,\abs*{x-x'}=r}\set*{\frac{\braket*{x-x'}{\b(t,x)-\b(t,x')}}{\abs*{x-x'}} + \frac{\norm*{\bar\sigma(t, x)-\bar\sigma(t, x')}^2}{2\abs*{x-x'}}},
    $$
    $$
    \kappa^*(r)=(M_\b\wedge K^\b_xr)\1{r\leq R}+\bm\eta r\1{r\leq R}-(\bm\eta-K^\sigma_x)r
    $$
    and \begin{equation}\label{Appendix: Lyapunov}
        \Phi(r):=\int^r_0 \exp{-\frac{1}{2\sigma_0^2}\int^s_0\kappa^*(v)\d v}\left( \int^\infty_s u\,\exp{\frac{1}{2\sigma_0^2}\int^u_0 \kappa^*(v) \d v } \d u\right)\d s.
\end{equation}
By using \eqref{Appendix: dissipatif}, we get $\kappa\leq \kappa^*.$
Moreover, $\Phi$ satisfies
\begin{lem}\label{Appendix: prop phi} For all $r\geq 0$, $\Phi$ satisfies
    $\vspace{0pt}$\begin{enumerate}
        \item $
            2\sigma^2_0\Phi''(r)+ \kappa(r) \Phi'(r)\leq -2\sigma_0^2 r$,
        \item $\Phi'(r)\geq0$ and $\Phi''(r)\leq 0$,
        \item $\frac{2\sigma^2_0}{\bm\eta-K^\sigma_x}r\leq \Phi(r)\leq \Phi'(0)r$
    \end{enumerate}
    and $\Phi'(0)$ satisfies 
    \begin{equation}\label{Appendix: eq: borne Phi'(0)}
        \Phi'(0)\leq \exp{\frac{\bm\eta+\frac{2M_\b}{R}}{4\sigma^2_0}R^2}\frac{2\sigma^2_0}{\bm\eta-K^\sigma_x}.
    \end{equation}
    \end{lem}
    \begin{proof}The proof is a straight forward adaptation of computations in {\cite[Theorem 3.2's proof]{Huang-Ma}}, namely equations \cite[(3.11),(3.12) and (3.13)]{Huang-Ma} where $\Phi  $ and $\kappa^*$ respectively correspond to $f$ and $\~\gamma$ in the cited reference, and the fact that $\kappa^*\geq \kappa$.
    \end{proof}
    First, by applying Itô-Tanaka's formula to $r_t:=\sqrt{r_t^2}$, we obtain, 
    \begin{align*}
        \d r_t = \frac{\braket*{\X_t-\X'_t}{\b(t,\X_t)-\b'(t,\X'_t)}}{r_t}\d t+ \frac{1}{2r_t}\norm*{\sigma(t, \X_t)-\sigma(t, \X'_t)}\d t +\frac{\braket*{\~\X_t }{\left(\bar\sigma(t, \X_t)-\bar\sigma(t, \X'_t)\right)\d W^2_t}}{r_t} + 2\sigma_0 \pi^1_\delta(r_t)e^\top_t\d W^1_t.
    \end{align*}
    Consequently, by applying again Itô's formula to $\Phi(r_t)$, one could get 
    \begin{align*}
        \d \Phi(r_t)&=\seg*{\Phi'(r_t)\left(\frac{\braket*{\X_t-\X'_t}{\b(t,\X_t)-\b'(t,\X'_t)}}{r_t}+ \frac{1}{2r_t}\norm*{\sigma(t, \X_t)-\sigma(t, \X'_t)}^2 \right)+2\sigma_0^2\Phi''(r_t)\pi^1_\delta(r_t)^2}\d t\\ 
        &\quad + \Phi'(r_t)\left( \frac{\braket*{\~\X_t }{\left(\bar\sigma(t, \X_t)-\bar\sigma(t, \X'_t)\right)\d W^2_t}}{r_t} + 2\sigma_0 \pi^1_\delta(r_t)e^\top_t\d W^1_t\right)\\
        &=\seg*{\Phi'(r_t)\left(\frac{\braket*{\X_t-\X'_t}{\b(t,\X_t)-\b(t,\X'_t)}}{r_t}+ \frac{1}{2r_t}\norm*{\sigma(t, \X_t)-\sigma(t, \X'_t)}^2 \right)+2\sigma_0^2\Phi''(r_t)\pi^1_\delta(r_t)^2}\d t\\
        &\quad+\Phi'(r_t)\frac{\braket*{\X_t-\X'_t}{\b(t,\X_t)-\b'(t,\X_t)}}{r_t}\d t + \Phi'(r_t)\left( \frac{\braket*{\~\X_t }{\left(\bar\sigma(t, \X_t)-\bar\sigma(t, \X'_t)\right)\d W^2_t}}{r_t} + 2\sigma_0 \pi^1_\delta(r_t)e^\top_t\d W^1_t\right).
    \end{align*} 
    We set,
    $$
    \d M_t:=\Phi'(r_t)\left( \frac{\braket*{\~\X_t }{\left(\bar\sigma(t, \X_t)-\bar\sigma(t, \X'_t)\right)\d W^2_t}}{r_t} + 2\sigma_0 \pi^1_\delta(r_t)e^\top_t\d W^1_t\right).
    $$
    Hence, by using \autoref{Appendix: prop phi}, \eqref{Appendix: dissipatif}, the facts that $\kappa\leq \kappa^*$ and $\kappa^*$ is increasing on $\seg*{0,R}$, we obtain
    \begin{align*}
        \d \Phi(r_t)&\leq \seg*{\Phi'(r_t) \kappa(r_t) +2\sigma_0^2\Phi''(r_t)\pi^1_\delta(r_t)^2}\d t  + \Phi'(0)\EE_t \d t +\d M_t\\
        &\leq \seg*{\Phi'(r_t) \kappa(r_t) +2\sigma_0^2\Phi''(r_t)}\pi^1_\delta(r_t)^2\d t  +\Phi'(r_t)\kappa(r_t)(1-\pi^1_\delta(r_t)^2)\d t+ \Phi'(0)\EE_t\d t +\d M_t\\
        &\leq-2\sigma_0^2 r_t \pi^1_\delta(r_t)^2\d t +\Phi'(0)\kappa^*(\delta)\d t+\Phi'(0)\EE_t \d t +\d M_t\\
        &=-2\sigma_0^2 r_t\d t + 2\sigma_0^2 r_t(1-\pi^1_\delta(r_t)^2)\d t+\Phi'(0)\kappa^*(\delta)\d t +\Phi'(0)\EE_t \d t+\d M_t\\
        &\leq -2\sigma_0^2 r_t\d t + \left(2\sigma_0^2 \delta+\Phi'(0)\kappa^*(\delta)\right)\d t +\Phi'(0)\EE_t \d t +\d M_t\\
        &\leq -\frac{2\sigma_0^2}{\Phi'(0)}\Phi(r_t)\d t + \ell(\delta)\d t+\Phi'(0)\EE_t +\d M_t,
    \end{align*}
    where we set $\ell(\delta):=2\sigma_0^2 \delta+\Phi'(0)\kappa^*(\delta)\overset{\delta\rightarrow 0}{\longrightarrow}0$. Due to the non-increasing property of $\Phi$ and the Lipschitz property of $\bar\sigma$, we have that the stochastic integral $\d M_t$ is a martingale. 
    Thus, by \eqref{Appendix: constraint} and \eqref{Appendix: eq: borne Phi'(0)}, we can  set $\bm{\hat\eta} := \frac{2\sigma_0^2}{\Phi'(0)}- \frac{\bm{c} \Phi'(0)(\bm\eta-K^\sigma_x)}{2\sigma^2_0}>0$. Thus, applying Itô's formula to $e^{\hat{\eta}t}\Phi(r_t)$ and taking the expectation lead to,
\begin{align*}
    e^{\bm{\hat\eta} t}\E\seg*{\Phi(r_t)}&\leq  \E\seg*{\Phi(r_0)} + \frac{\ell(\delta)}{\bm{\hat\eta}}e^{\bm{\hat\eta} t} + \Phi'(0)\int^t_0\EE_s e^{\bm{\hat\eta} s}\d s-\frac{\bm{c} \Phi'(0)(\bm\eta-K^\sigma_x)}{2\sigma^2_0}\int^t_0\E\seg*{\Phi(r_s)}e^{\bm{\hat\eta} s}\d s\\
   & \quad\overset{\delta\rightarrow 0}{\longrightarrow}\E\seg*{\Phi(r_0)}+ \Phi'(0)\int^t_0\EE_s e^{\bm{\hat\eta} s}\d s-\frac{\bm{c} \Phi'(0)(\bm\eta-K^\sigma_x)}{2\sigma^2_0}\int^t_0\E\seg*{\Phi(r_s)}e^{\bm{\hat\eta} s}\d s.
\end{align*}
Hence, by using \autoref{Appendix: prop phi}-3. and taking $(\X_0,\X'_0)$ an optimal coupling for $\w$-distance, last inequality becomes
\begin{align}
    \w(\loi{\X_t},\loi{\X'_t})\leq \E\seg*{r_t}  &\leq \frac{\bm\eta-K^\sigma_x}{2\sigma_0^2}e^{-\bm{\hat\eta} t} \left(\E\seg*{\Phi(r_0)}+ \Phi'(0)\int^t_0\left(\EE_s -\bm{c}\frac{\bm\eta-K^\sigma_x}{2\sigma_0^2}\E\seg*{\Phi(r_s)}\right)e^{\bm{\hat\eta} s}\d s\right)\notag\\
    &\leq \frac{\bm\eta-K^\sigma_x}{2\sigma_0^2}\Phi'(0)\E\seg*{\abs*{\X_0-\X_0'}}e^{-\bm{\hat\eta} t} + e^{-\bm{\hat\eta} t}\Phi'(0)\int^t_0\left(\EE_s -\bm{c}\E\seg*{r_s}\right)e^{\bm{\hat\eta} s}\d s\notag\\
    &=C\w(\loi{\X_0},\loi{X'_0})e^{-\bm{\hat\eta} t} +\Phi'(0)e^{-\bm{\hat\eta} t}\int^t_0 \left( \EE_s-\bm{c} \E\seg*{r_s}\right)e^{\bm{\hat\eta} s}\d s.\label{Appendix: eq: dernière eq}
\end{align}
\end{proof}
\bibliographystyle{abbrv}
\bibliography{biblio}

\end{document}